\documentclass[a4paper, 11pt]{amsart}
\usepackage{graphicx}
\usepackage[utf8]{inputenc}
\usepackage{amsmath, amssymb, mathrsfs}
\usepackage{bbm}
\usepackage{mathtools}
\usepackage{csquotes}
\usepackage{tikz-cd}
\usepackage{cite}
\usepackage{amsthm}

\usepackage[all]{xy}
\usepackage{ stmaryrd }
\usepackage{graphicx}
\usepackage{tikz}
\usetikzlibrary{babel}
\usepackage{longtable}
\usepackage{ stmaryrd }
\usetikzlibrary{decorations.pathmorphing,arrows.meta,decorations.markings,positioning,calc}

\usepackage{comment}
\usepackage{float}
\newfloat{myquote}{t}{} 
\usepackage[top=3cm, bottom=3cm, left=2cm, right=2cm]{geometry}

\usetikzlibrary{decorations.pathmorphing}

\usepackage{imakeidx}
\usepackage[initsep=0pt]{idxlayout}
\makeindex
\usepackage{hyperref}
\hypersetup{
	colorlinks=true,
	linkcolor=blue,
	filecolor=magenta,      
	urlcolor=cyan,
}
\usepackage{cleveref}
\usepackage{thmtools}

\newcommand{\bA}{\mathbb{A}}

\newcommand{\bG}{\mathbb{G}}

\newcommand{\bP}{\mathbb{P}}

\newcommand{\bZ}{\mathbb{Z}}

\newcommand{\cB}{\mathcal{B}}

\newcommand{\cG}{\mathcal{G}}
\newcommand{\cH}{\mathcal{H}}

\newcommand{\cM}{\mathcal{M}}

\newcommand{\cO}{\mathcal{O}}

\newcommand{\cU}{\mathcal{U}}

\newcommand{\cX}{\mathcal{X}}
\newcommand{\cY}{\mathcal{Y}}
\newcommand{\cZ}{\mathcal{Z}}

\def\Spec{\operatorname{Spec}}

\newcommand{\spec}{\operatorname{Spec}}

\newcommand{\Hom}{\operatorname{Hom}}

\newcommand{\Ga}{\bG_{{\rm a}}}
\newcommand{\GL}{\operatorname{GL}}

\newcommand{\Aut}{\operatorname{Aut}}

\newcommand{\sgn}{\operatorname{sgn}}

\newcommand{\Stab}{\operatorname{Stab}}

\newcommand{\Def}{{\rm Def}}
\newcommand{\cHom}{\mathcal{H}\kern -.5pt om}
\newcommand{\cExt}{\mathcal{E}\kern -.5pt xt}

\newcommand{\ST}{\overline{\operatorname{ST}}}
\newcommand{\Gm}{\mathbb{G}_m}

\newtheoremstyle{thmcite}
  {}
  {}
  {\itshape}
  {}
  {\bfseries}
  {.}
  { }
  {\thmname{#1}\thmnumber{ #2}\thmnote{ \normalfont\checkcite{#3}}}

\newtheoremstyle{defcite}
  {} 
  {} 
  {\normalfont} 
  {} 
  {\bfseries} 
  {.} 
  { } 
  {\thmname{#1}\thmnumber{ #2}\thmnote{ \normalfont\checkcite{#3}}}

\ExplSyntaxOn
\NewDocumentCommand{\checkcite}{m}
 {
  \regex_if_match:nnTF { \A \cB\{* \c{cite} [^\c{cite}]* \Z } {#1}
   {
    #1
   }
   {
    (#1)
   }
 }
\ExplSyntaxOff

\theoremstyle{thmcite}
\newtheorem{theorem}{Theorem}[section]
\newtheorem*{theorem*}{Theorem}
\newtheorem*{maintheorem}{Main Theorem}
\newtheorem{lemma}[theorem]{Lemma}
\newtheorem{corollary}[theorem]{Corollary}
\newtheorem{proposition}[theorem]{Proposition}

\theoremstyle{defcite}
\newtheorem{remark}[theorem]{Remark}
\newtheorem{example}[theorem]{Example}
\newtheorem{definition}[theorem]{Definition}

\newtheorem{question}[theorem]{Question}

\theoremstyle{definition}

\newtheoremstyle{modusstyle}
  {3pt}{3pt}{\itshape}{}{\bfseries}{.}{.5em}{}

\theoremstyle{modusstyle}
\newtheorem*{modusoperandi}{Modus Operandi}

\newcommand{\labelmodus}[1]{\hypertarget{#1}{}}
\newcommand{\refmodus}[1]{\hyperlink{#1}{\textit{Modus Operandi}}}

\let\oldtocsection=\tocsection

\let\oldtocsubsection=\tocsubsection

\let\oldtocsubsubsection=\tocsubsubsection

\renewcommand{\tocsection}[2]{\hspace{0em}\oldtocsection{#1}{#2}}
\renewcommand{\tocsubsection}[2]{\hspace{1em}\oldtocsubsection{#1}{#2}}
\renewcommand{\tocsubsubsection}[2]{\hspace{2em}\oldtocsubsubsection{#1}{#2}}

\DeclareRobustCommand{\SkipTocEntry}[5]{} 

\title{Obstructions to the Existence of\\ Good Moduli Spaces of $A_r$-Stable Curves}



\author{Davide Gori}
\address{Universität Duisburg-Essen, Essen, Germany}
\email{davide.gori@uni-due.de}
\author{Ludvig Modin}
\address{Leibniz Universität Hannover, Institut für Algebraische Geometrie, Welfengarten 1, 30167 Hannover}
\email{modin@math.uni-hannover.de}
\author{Michele Pernice}
\address{Department of Mathematics, University of Washington, 4110 E Stevens Way NE, Seattle, WA 98195}
\email{mpernice@uw.edu}

\begin{document}

\begin{abstract}
We study obstructions to the existence of separated good moduli spaces for open substacks of the moduli stack $\mathcal{M}_{g,n}^r$ of $A_r$-stable curves. Our approach is based on an analysis of families of curves over $\Theta_R$ and $\overline{\text{ST}}_R$, building on prior work on the local geometry of $\mathcal{M}_{g,n}^r$. We prove that $\mathcal{M}_{g,n}^r$ is neither $\Theta$- nor $\textsf{S}$-complete. We then construct an open substack $\mathcal{U}_{g,n}^r \subset \mathcal{M}_{g,n}^r$ and show that the counterexamples identified in $\mathcal{M}_{g,n}^r$ do not occur within this substack. Moreover, we prove that $\mathcal{U}_{g,n}^r$ cannot be strictly contained in any other substack of $\mathcal{M}_{g,n}^r$ that admit a separated good moduli space. Furthermore, we show that the inclusion $\mathcal{U}_{g,n}^r \subset \mathcal{M}_{g,n}^r$ is both $\Theta$- and $\textsf{S}$-complete. These results will be used in a forthcoming paper to prove that $\mathcal{U}_{g,n}^r$ admits a separated, and indeed proper, good moduli space for $r \leq 5$.
\end{abstract}

\maketitle
\tableofcontents

\section{Introduction}
This paper is the second in a series dedicated to the problem of constructing modular compactifications using intrinsic techniques within the framework of the ``Beyond GIT'' program introduced in \cite{StructureOfInstability}. For a comprehensive presentation of this recent approach to moduli theory, see \cite{AlpHalLei}. Specifically, we aim to apply the criteria introduced in \cite{ExistenceOfModuli} and verify $\Theta$- and $\textsf{S}$-completeness in order to study new modular compactifications of $\cM_{g,n}$ that admit a proper good moduli space. In the first paper of the series \cite{GMSArI}, we studied the local geometry of the algebraic stack $\cM_{g,n}^r$ parametrizing $A_r$-stable curves, i.e.\ curves with at worst $A_r$-singularities and such that the log-dualizing line bundle is ample. As in \cite{GMSArI}, we work over a field of characteristic $0$; the results generalize to fields of \emph{characteristic big enough with respect to $(g,n,r)$}, similarly to what is done in \cite[Definition~4.1.1]{AltCompClustAlg}.

Building on this previous work, in the present paper we produce counterexamples to $\Theta$-completeness and $\textsf{S}$-completeness for $\cM_{g,n}^r$, see for instance \Cref{rem:not-S-complete} (as was already announced in \cite[Section~6.3]{AlpHalLei}). Using these counterexamples, we define an open substack $\cU_{g,n}^r \subset \cM_{g,n}^r$ and prove the following result.

\begin{maintheorem}\label{theo:main}
The open embedding $\cU_{g,n}^r \subset \cM_{g,n}^r$ is $\Theta$-complete and $\textsf{S}$-complete. Suppose now that $g \geq 2$ and $r \leq 2g - 5$, then if $\cU_{g,n}^r \subset \cU \subset \cM_{g,n}^r$ is an open substack admitting a separated good moduli space, then $\cU_{g,n}^r = \cU$.
\end{maintheorem}

The maximality result clearly does not imply that $\cU_{g,n}^r$ has a separated good moduli space. Nevertheless, the $\Theta$-completeness and $\textsf{S}$-completeness of the open embedding are reassuring properties: indeed, if $\cU_{g,n}^r$ has a separated good moduli space, then the open embedding satisfies the two valuative criteria. In the third paper of the series, we will prove that $\cU_{g,n}^r$ has a proper good moduli space for $r\leq 5$, which is non-projective for $r\geq 5$ (if it exists) and non-schematic if $n=0$ and $g \geq 9$. We would like to emphasize that in the proofs of $\Theta$-completeness and $\textsf{S}$-completeness in the third paper, we will need the good moduli space properties of the open embedding.

Using the restrictions on the geometry of the isotrivial degenerations proved in \cite{GMSArI}, we identified several counterexamples to $\Theta$-completeness and $\textsf{S}$-completeness that guided us to the definition of an open substack $\cU_{g,n}^r\subset \cM_{g,n}^r$ that does not contain these counterexamples. We summarize our choices in \Cref{def:admin}. We then use those same counterexamples to prove the maximality statement.

In \cite[Proposition 2.12]{GMSArI} (see also \Cref{prop:local-Theta-complete}), the good moduli space properties of the open embedding are proven to be equivalent to a combinatorial condition on the closed complement of the open substack $\cU_{g,n}^r$ \'etale-locally around any of its points. The combinatorial condition has its roots in the following two properties:
\begin{itemize}
    \item the stabilizer group at every point is linearly reductive (indeed a finite extension of a torus),
    \item $A$-singularities are unobstructed;
\end{itemize}
The problem can thus be interpreted using toric representation theory. We then solve the combinatorial problem in \Cref{sec:combinatorics}.

\addtocontents{toc}{\SkipTocEntry}
\subsection*{Obstruction to families of curves over \texorpdfstring{$\Theta_R$}{Theta\_R} and \texorpdfstring{$\ST_R$}{ST\_R}}
The underlying idea for finding the right open substack, showing that it is maximal, and proving that the inclusion \(\cU_{g,n}^r\subset \cM_{g,n}^r\) is \(\Theta\)- and \(\textsf{S}\)-complete is to understand the interplay between the two valuative criteria and the local geometry of \(\cM_{g,n}^r\). To explain this in more detail, we briefly recall the geometry of the test objects $\Theta_R$ and $\ST_R$, cf.~\cite[Sections~3.3, 3.5]{ExistenceOfModuli}. Both of these have a unique closed point, which we denote by $0$. For a morphism of algebraic stacks $f \colon \cX \to \cY$, being $\Theta$- and $\textsf{S}$-complete corresponds to the existence of a unique lift for any diagram
\begin{equation}
\label{eq:thetascpl}
    \begin{tikzcd}
    \Theta_R \setminus \{0\} \arrow[r] \arrow[d, hook] & \cX \arrow[d, "f"] \\
    \Theta_R \arrow[r] \arrow[ru, dashed] & \cY
    \end{tikzcd}
    \qquad \qquad \qquad
    \begin{tikzcd}
    \ST_R \setminus \{0\} \arrow[r] \arrow[d, hook] & \cX \arrow[d, "f"] \\
    \ST_R \arrow[r] \arrow[ru, dashed] & \cY
    \end{tikzcd}.
\end{equation}
 In the discussion that follows, we explain how to visualize these test objects (and thus the valuative criteria) using specializations over either $\Theta:=[\bA^1/\Gm]$ or $\Spec R$, with $R$ a dvr. We explain the strategy in the case when $\cY$ is a point; the relative case works in the same way. We start with the description of $\Theta_R \coloneq \Theta \times \spec R$: this has four topological points, and we can visualize the specialization of points in the diagram below, where we specify whether the degenerations are over the dvr or over $\Theta$:
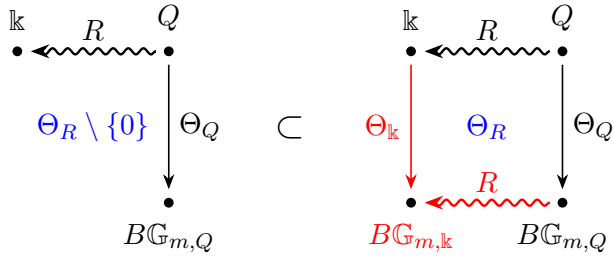
\begin{figure}[H]
    \centering
    \begin{tikzpicture}[
    >=Stealth,
    decoration={snake, amplitude=.4mm, segment length=2mm, pre length=1mm, post length=2mm},
    dot/.style={circle, fill, inner sep=1.3pt},
    wavy/.style={->, decorate, line width=0.9pt, shorten >=3pt, shorten <=3pt},
    straightarr/.style={->, line width=0.6pt, shorten >=3pt, shorten <=3pt}
]

\node[dot] (k1) at (-0.2,2) {};
\node[dot] (q1) at (1.8,2) {};
\node[dot] (b1) at (1.8,0) {};

\node[above=2pt of k1] {$\Bbbk$};
\node[above=2pt of q1] {$Q$};
\node[below=2pt of b1] {$B\mathbb G_{m,Q}$};

\draw[wavy] (q1) -- (k1) node[midway, above] {$R$};
\draw[straightarr] (q1) -- (b1) node[midway, right] {$\Theta_{Q}$};

\node[blue] at (0.8,1) {$\Theta_{R}\setminus \{0\}$};

\node at (3.4,1) {\Large $\subset$};

\node[dot] (k2) at (5,2) {};
\node[dot] (q2) at (7,2) {};
\node[dot] (c2) at (5,0) {};
\node[dot] (d2) at (7,0) {};

\node[above=2pt of k2] {$\Bbbk$};
\node[above=2pt of q2] {$Q$};
\node[below=2pt of c2, red] {$B\mathbb G_{m,\Bbbk}$};
\node[below=2pt of d2] {$B\mathbb G_{m,Q}$};

\draw[wavy] (q2) -- (k2) node[midway, above] {$R$};
\draw[straightarr, red] (k2) -- (c2) node[midway, left] {$\Theta_{\Bbbk}$};
\draw[straightarr] (q2) -- (d2) node[midway, right] {$\Theta_{Q}$};
\draw[wavy, red] (d2) -- (c2) node[midway, above] {$R$};

\node[blue] at (6,1) {$\Theta_{R}$};

\end{tikzpicture}
    \caption{Diagram of specializations for $\Theta$-completeness}
    \label{fig:theta-complete}
\end{figure}
where the diagram is to be read as ``for all degenerations in $\cX$ as in the leftmost diagram, there exists a unique filling as in the rightmost diagram''. We stress that a map $\Theta_R \setminus \{0\} \to \cU$ is exactly the datum of two maps $\spec R \to \cU$ and $\Theta_Q \to \cU$, together with an isomorphism of their images over the generic point. On the contrary, the requirement of lifting the map in \eqref{eq:thetascpl} is in general stronger than the existence of a point together with the two required specializations in the diagram above. In future work of David Rydh, similar ideas are used to formulate existence criteria for topological moduli spaces (defined there); we also use a version of these criteria in the third part of this series of paper for the proof of \(\textsf{S}\)-completeness.

Similarly, the \(S\)-completeness criterion for $\cX$ can be rephrased as the following diagram of specializations
\begin{figure}[H]
    \centering
    \begin{tikzpicture}[
    >=Stealth,
    decoration={snake, amplitude=.4mm, segment length=2mm, pre length=1mm, post length=2mm},
    dot/.style={circle, fill, inner sep=1.3pt},
    wavy/.style={->, decorate, line width=0.9pt, shorten >=5pt, shorten <=5pt},
    straightarr/.style={->, line width=0.7pt, shorten >=3pt, shorten <=3pt}
]

\node[dot] (kt1) at (0,1.4) {};
\node[dot] (kb1) at (0,-1.4) {};
\node[dot] (q1) at (1.9,0) {};

\node[above=2pt of kt1] {$\Bbbk$};
\node[below=2pt of kb1] {$\Bbbk$};
\node[right=2pt of q1] {$Q$};

\draw[wavy] (q1) -- (kt1) node[midway, above right] {$R$};
\draw[wavy] (q1) -- (kb1) node[midway, below right] {$R$};

\node[blue] at (0.3,0) {$ST_R$};

\node at (3.1,0) {\Large $\subset$};

\begin{scope}[shift={(6,0)}]

\node[dot] (kt2) at (0,1.4) {};
\node[dot] (kb2) at (0,-1.4) {};
\node[dot] (q2) at (1.9,0) {};
\node[dot] (bg2) at (-0.9,0) {};

\node[above=2pt of kt2] {$\Bbbk$};
\node[below=2pt of kb2] {$\Bbbk$};
\node[right=2pt of q2] {$Q$};
\node[left=2pt of bg2] {$B\mathbb G_{m,\Bbbk}$};

\draw[wavy] (q2) -- (kt2) node[midway, above right] {$R$};
\draw[wavy] (q2) -- (kb2) node[midway, below right] {$R$};
\draw[straightarr, red] (kt2) -- (bg2) node[midway, above left] {$\Theta_{\Bbbk}$};
\draw[straightarr, red] (kb2) -- (bg2) node[midway, below left] {$\Theta_{\Bbbk}$};

\node[blue] at (0.2,0) {$\overline{ST_R}$};

\end{scope}

\end{tikzpicture}
    \caption{Diagram of specializations for $S$-completeness}
    \label{fig:S-complete}
\end{figure}
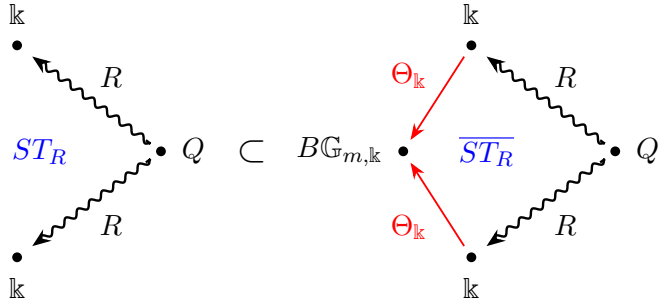
to be read as described in the $\Theta$-completeness case. In particular, whenever one can show that such fillings cannot exist, then \(\mathcal{X}\) does not admit a separated good moduli space.
It is therefore clear that, in order to study these criteria, one of the key tools is to understand maps $\nu \colon \Theta \to \cX$. Such a map is called an \emph{isotrivial degeneration} of the topological point corresponding to the image of $1 \in |\Theta|$. Moreover, such a map induces a cocharacter $\Gm \to \Aut_{\cX}(p)^{\circ}$, where $p \in |\cX|$ is the topological point corresponding to $\nu(0)$. For this reason, studying the \emph{basin of attraction} of all pairs \((p,\Gm\to \Aut_{\cX}(p)^{\circ})\) is crucial for this analysis; see \cite[Definition~2.5]{GMSArI}. A further motivation for studying basins of attraction comes from considering the restriction of a map $\ST_R \to \cX$ to the closed substack obtained by removing the generic point. This corresponds to two copies of $\Theta$ glued at the origin, such that the one-parameter subgroups induced on the closed point of $\ST_R$ on the two copies are opposite.

The study of such isotrivial degenerations for $\cM_{g,n}^r$ is part of the content of \cite{GMSArI}. Furthermore, in our case, since \(\cM_{g,n}^r\) is a smooth stack with \emph{all} stabilizers being extensions of a torus by a finite group (with the exception of \(r\geq 2g+1\) and \(n=0\)), this makes it possible to linearize the problem, reducing it to understanding the \(\Aut_{\cX}(p)^{\circ}\)-equivariant deformation theory of $p$; see, for example, \cite[Lemma 2.10]{GMSArI}.

In this way, we can describe the basin of attraction of $(p,\lambda:\Gm\to \Aut_{\cX}(p)^{\circ})$ as the positive-weight summand of the space of first-order deformations of $p$ seen as a $\Gm$-representation through $\lambda$.

\begin{modusoperandi}\labelmodus{rem:modus-operandi}
To define \(\cU_{g,n}^r\) (see \Cref{def:admin}), we investigate which diagrams as above can be constructed in \(\cM_{g,n}^r\), and subsequently remove curves that would obstruct their filling. In this analysis, the natural stratification of $\cM_{g,n}^r$ provided by the combinatorial type of the curves plays a crucial role, although we do not formally define it here in order to keep the notation lighter (for a precise definition cf.~\cite[Section~3.2]{AltCompClustAlg}). More precisely, our \emph{modus operandi} in defining $\cU_{g,n}^r$ is the following:
\begin{itemize}
    \item[(1)] we find counterexamples to $\Theta$-completeness in $\cM_{g,n}^r$;
    \item[(2)] we decide what to remove from the stack so that the counterexamples cannot appear, preferably avoiding the removal of curves with positive-dimensional automorphism groups;
    \item[(3)] if we remove curves with positive-dimensional automorphism groups, we check $\textsf{S}$-completeness to determine whether one of the basins of attraction of the removed curve must also be removed.
\end{itemize}
\end{modusoperandi}

Note that in this process we made choices: there is no unique maximal open substack of curves in $\cM_{g,n}^r$ admitting a proper good moduli space. For example, $\overline{\cM}_{g,n}\subset \cM_{g,n}^r$ (parametrizing nodal stable curves) is a compactification of $\cM_{g,n}$ that is not contained in $\cU_{g,n}^r$ for $r \geq 3$. Moreover, there is no containment among these $\cU_{g,n}^r$ for $r \leq 2g+1$. In our approach, we always choose to keep the largest possible class of curves with singularities of type $A_i$ for $i \leq r$ while defining $\cU_{g,n}^r$. 

This motivates the definition of an open immersion $\cU_{g,n}^r \subset \cM_{g,n}^r$ that satisfies the maximality part of the main theorem. It should also be added that, in the case of small \(r\), we compared our curves with, and were inspired by, those appearing in the first flips of the Hassett--Keel program. For a thorough comparison, see \Cref{rem:HKP-comparison} and the next section.

\addtocontents{toc}{\SkipTocEntry}
\subsection*{Relation with previous works and Hasset--Keel program}
Classically, many compactifications of $\operatorname{M}_{g,n}$ arise from Deligne--Mumford stacks that, thanks to the Keel--Mori theorem, admit a coarse moduli space. In this sense, there are many famous compactifications. The first to appear was $\overline{\operatorname{M}}_{g,n}$ in \cite{DeligneMumford}, established using nodal curves, followed by the Schubert compactification \cite{PseudostableSchubert} (given by curves with $A_1$ and $A_2$ singularities), and later the compactifications constructed in \cite{smyth2011modular} using more general singularities, which are therefore not included in our stack. Among other constructions, we mention the possibility of generalizing the notion of nodal curves by assigning weights to marked points, leading to the construction in \cite{hassett2003moduli}, which goes in a different direction from the one we pursue here.

The literature is less abundant on compactifications arising from non-Deligne--Mumford stacks. In particular, the known compactifications for arbitrary genus are related to the Hassett--Keel program for $\overline{\operatorname{M}}_{g,n}$. Specifically, the first steps of the Hassett--Keel program were constructed in \cite{hassett2009log, hassett2013log, AlpFedSmyWyck, AlpFedSmyExistence, AlpFedSmyProjectivity}, and further compactifications were later established in \cite{Viviani, AltCompClustAlg}. All these compactifications arise from stacks classifying curves with singularities of type $A_i$ for $i \leq 4$. Our work is deeply inspired by these compactifications and aims to understand compactifications with $A_i$-singularities for arbitrary $i$. We briefly compare them (cf.~\Cref{rem:HKP-comparison}):
\begin{itemize}
    \item if $r=0$, we have $\cU_{g,n}^0 = \cM_{g,n}$.
    \item $\cU_{g,n}^1 = \overline{\cM}_{g,n}$, the stack classifying stable nodal curves.
    \item $\cU_{g,n}^2 = \overline{\cM}_{g,n}^{wps}$, the stack of weakly pseudostable curves, cf.~\cite{hassett2009log}.
    \item $\cU_{g,n}^3 = \overline{\cM}_{g,n}(7/10)$. This stack was constructed in \cite{hassett2013log}, and admits a good moduli space that, for $n=0$, is the log canonical model for $K_{\overline{\cM}_g} + \frac{7}{10}\delta$, where $\delta$ is the total boundary divisor. The construction was generalized to the case of marked points in \cite[Definition 2.5]{AlpFedSmyWyck}.
    \item $\overline{\cM}_{g,n}(2/3) \subsetneq \cU_{g,n}^4$. In the third article of this series, we will prove that this induces a birational proper morphism between the good moduli spaces. We recall that $\overline{\cM}_{g,n}(2/3)$ was first constructed in \cite{AlpFedSmyWyck}.
\end{itemize}

\addtocontents{toc}{\SkipTocEntry}
\subsection*{Maximal opens admitting good moduli spaces}
    We emphasize the interest in finding an open substack that admits a separated good moduli space and is maximal. This significantly simplifies the classification of all open substacks with separated good moduli spaces: namely, if $\cU \subset \cX$ and $\cX$ admits a separated good moduli space, then $\cU$ has a good moduli space if and only if the inclusion is $\Theta$- and $\textsf{S}$-complete, which is a weaker statement and usually easier to prove (e.g., using \Cref{prop:local-Theta-complete}).
    
    This strategy follows the methods used to prove that the stacks of semistable vector bundles over a curve admit a separated good moduli space in \cite[Propositions 3.8 and 3.10]{GMS-VectBun}, where it is noted that the stack of coherent sheaves is \(\Theta\)- and \(\textsf{S}\)-complete and that the inclusion of the stack of semistable vector bundles as a substack is \(\Theta\)- and \(\textsf{S}\)-complete as well. In the case of curves, this method has been used in \cite{AltCompClustAlg} to classify stacks of curves with singularities of type $A_i$, $i \leq 3$, admitting a projective good moduli space.

    Finally, we highlight that the maximality of the open substack implies some ``minimality'' phenomena at the level of the good moduli space. Namely, if some of the points added have non-trivial basins of attraction, all points in these basins will be mapped to the same point in the good moduli space. For example, the inclusion of stacks \(\mathbb{P}^n\subset [\mathbb{A}^{n+1}/\mathbb{G}_m]\) corresponds to the structure map \(\mathbb{P}^n\to \Spec \kappa\) on the level of good moduli spaces.

    The presence of many points with non-trivial basins of attraction explains the absence of non-torsion line bundles on the good moduli space of $\cU_{g,n}^r$, $g,n$ for $r \geq 5$ (when it exists), which we prove in the third and last part of this series. In particular, such a moduli space is not projective and can therefore not be obtained via GIT methods.

\addtocontents{toc}{\SkipTocEntry}
\subsection*{Outline of the paper}
In \Cref{sec:Ar-stable-curves} we recall several definitions and results developed in \cite{GMSArI} that will be used throughout the paper. We roughly retrace the results in \emph{loc.\ cit.}: in \Cref{sub:equivariant-geometry-A-singularities} we discuss the crimping datum for $A$-type singularities; in \Cref{sub:hyper-A-stable} we review cyclic covers of degree $2$; in \Cref{sub:pointwise geometry} we recall the classification of automorphism groups in $\cM_{g,n}^r$; and in \Cref{sub:local-geometry} we describe the equivariant deformation theory and how it induces natural geometric constraints on possible isotrivial degenerations. Finally, in \Cref{sub:global geometry} we briefly explain why the local results arising from the study of deformation theory can be globalized, yielding the classification.

\Cref{sec:counterexamples} is devoted to the substack $\cU_{g,n}^r \subset \cM_{g,n}^r$ parametrizing admissible $n$-pointed $A_r$-stable curves. We begin by defining the locus of admissible curves and proving that it is open. In the remainder of the section, we prove the maximality of $\cU_{g,n}^r$ (see \Cref{theo:maximal-open}), where each subsection addresses the conditions defining $\cU_{g,n}^r$, clarifying why they are imposed in the first place.

Finally, in \Cref{sec:open-theta} we prove that the embedding $\cU_{g,n}^r \subset \cM_{g,n}^r$ is $\Theta$- and $\textsf{S}$-complete (see \Cref{theo:local-completeness}). We emphasize that we translate the lifting of morphisms into a purely combinatorial condition, which is then addressed in \Cref{sec:combinatorics}.

\addtocontents{toc}{\SkipTocEntry}
\subsection*{Future directions}
 As we already mentioned, one of our initial motivations was to understand how to interpret the valuative criteria for $\textsf{S}$-completeness and $\Theta$-completeness in the case of moduli of worse-than-nodal curves.
 
 As we already pointed out, in the case of curves it seems that we are missing an analogue of the stack of coherent sheaves, i.e. an ambient algebraic stack which is $\textsf{S}$-complete and $\Theta$-complete, in which to search for our compactifications. This would make the study of modular compactifications much easier: instead of needing to work our way up from the nodal case (as we do), we could simply produce open substacks such that the corresponding embeddings are $\textsf{S}$-complete and $\Theta$-complete. We believe that any classification will have to start with the difficult problem of finding the open substacks that are maximal among those having a separated good moduli space. Once that is done, one can try to establish a classification by combinatorially parametrizing closed subsets whose embedding of the open complement is $\textsf{S}$-complete and $\Theta$-complete. Notice that, in the case of linearly reductive stabilizers at every point, one can use \Cref{prop:local-Theta-complete} to actually translate the problem into a representation-theoretic one, as we do in \Cref{sec:open-theta}. An instance of this idea is in the first named author's PhD thesis \cite{AltCompClustAlg}, which approaches the classification of all modular compactifications inside $\cM_{g,n}^3$ (thus up to tacnodal curves), establishing three maximal open substacks that are both $\textsf{S}$- and $\Theta$-complete using similar techniques. The classification then proceeds within these maximal stacks, where more combinatorics enters the picture, and a dictionary with the combinatorics of cluster algebras is established. 

 Thus, the question remains: 

\begin{question}
     Is there a good equivalent of the moduli stack of coherent sheaves for the moduli problem of curves?
\end{question}

 Although we have no clear answer, one natural candidate would be $\cG_{g,n}$, the moduli stack of reduced connected Gorenstein curves with ample (log)-dualizing line bundle. Nevertheless, the counterexamples we produce may still be valid in $\cG_{g,n}$. To understand this, we formulate the following more precise question.
 
 \begin{question}\label{conj:open}
     Is the open embedding $\cM_{g,n}^r \subset \cG_{g,n}$ $\textsf{S}$-complete and $\Theta$-complete?
 \end{question}

 An affirmative answer would imply that our counterexamples to the valuative criteria will remain valid in the bigger stack $\cG_{g,n}$. At the time of writing, we see two main obstacles to the proof of an affirmative answer to \Cref{conj:open}. First of all, to handle the problem one would like to reduce to the local picture, where deformation theory can be a powerful tool. Nevertheless, the \'etale local structure theorem is still a mystery for non-linearly reductive stabilizers (even though future work of David Rydh on topological moduli spaces could shed light on the question), which may appear depending on the choice of $(g,n,r)$. Secondly, even assuming the local reduction step, it is not true that, given a Gorenstein singularity with a $\Gm$-action, all weights in the deformation space have the same sign (e.g., the planar singularity $y^4-x^5$). Although seemingly unrelated, it is not too hard to convince oneself that such a property would imply that the answer to \Cref{conj:open} is affirmative. Since such a property holds for $A$-singularities, it implies that the open embedding $\cM_{g,n}^r \subset \cM_{g,n}^s$ is $\textsf{S}$-complete and $\Theta$-complete for any $r\leq s$: this is the reason why the problem of classifying isotrivial degenerations for $A$-singularities felt difficult, but not out of reach.

\addtocontents{toc}{\SkipTocEntry}
\subsection*{Acknowledgments}

The authors are tremendously indebted to Jarod Alper, Luca Battistella, Andrea Di Lorenzo, Andres Fernandez Herrero, Jochen Heinloth, Giovanni Inchiostro, David Rydh, Filippo Viviani and Dario Wei{\ss}mann for helpful discussions along the way. We further want to thank Jochen Heinloth for the opportunity to work one week in person on the project during a research visit of the first and third named authors in Essen. The first named author was partially supported by INdAM group GNSAGA. The second named author was supported by DFG-Research Training Group 2553 at Universität Duisburg-Essen during the first phase of this project and is currently a member of the Institut für Algebraische Geometrie in Leibniz Universität Hannover. The third named author was supported by the Knut and Alice Wallenberg foundation 2021.0291.
\section{\texorpdfstring{$A_r$}{A\_r}-stable curves and their moduli stack}\label{sec:Ar-stable-curves}
In this section, we will recall the notion of $A_r$-stable curves and the local geometry of the moduli stack $\cM_{g,n}^r$ of $A_r$-stable curves. This study is carried out thoroughly in the first paper of the series, \cite{GMSArI}, thus here we only recall what is needed, without proof. 

The first paper provided a \emph{combinatorial classification of isotrivial degenerations} of $A_r$-stable curves. This is achieved by: 
\begin{itemize}
    \item[1)] classifying all the curves with positive-dimensional stabilizers (see \cite[Corollary 4.26]{GMSArI});
    \item[2)] studying their equivariant deformation theory, which allows us to control how the combinatorics of the curve changes along $\Gm$-equivariant deformations (reported in this paper in \Cref{prop:defor-rational-chain} and \Cref{prop:iso-deg-rosary});
    \item[3)] globalizing the classification, which allows us to determine whether an isotrivial degeneration between two curves exists or not, based only on the combinatorics of the curves (see the results in \cite[Section 4.4]{GMSArI}).
\end{itemize}

The combinatorial classification of isotrivial degenerations will also be crucial to deal with $\Theta$-completeness and $\textsf{S}$-completeness in the future paper of the series, and it relies heavily on the strong connection between $A$-singularities and (reduced) $2\!:\!1$ covers of the projective line, which we call honestly hyperelliptic curves (following \cite{Cat}). More precisely, given a generically smooth family of curves over a dvr whose special fiber contains an $A$-singularity, we can locally modify the family by replacing such a singularity with a hyperelliptic curve (up to a dvr extension). This process is called semistabilization and can be done explicitly with $ADE$ singularities; see for instance \cite{Fedor} and \cite{CasmarLaza}. In the case of an $A_r$-singularity, the semistabilization procedure can be described depending on the parity of $r$:

\begin{itemize}
    \item[(e)] if $r$ is even, then a blowup in a smoothing family of the singularity replaces the singularity with an honestly hyperelliptic curve of genus $r/2$ (with at worst $A_{r-1}$-singularities) attached to the rest of the curve at a Weierstrass point;
    \item[(o)] if $r$ is odd, then an appropriate blowup (described in loc.~cit.) in a smoothing family of the singularity replaces the singularity with an honestly hyperelliptic curve of genus $(r-1)/2$ (with at worst $A_{r-1}$-singularities) attached to the rest of the curve at two points exchanged by the hyperelliptic involution.
\end{itemize}
The semistable replacement can be obtained by repeating the process in order to end up with only nodal singularities. This strong connection is also reflected in the deformation space of $A_r$-singularities with a nontrivial $\Gm$-action, which we call \emph{hyperelliptic} $A$-singularities (see \Cref{def:hyp-even-sing}). Informally, the $\Gm$-action induces a $\Gm$-equivariant decomposition of the deformation space of the singularity: the positively weighted isotypic component corresponds to deformations that preserve the singularity but change the crimping datum, whereas the negatively weighted isotypic component corresponds to deformations that deform the singularity but produce an honestly hyperelliptic component as a result (similarly to how it occurs in the semistabilization procedure described above). The precise statement is reported in this paper in \Cref{rem:alternating-def}.

After introducing the classical notation, we will recall the pointwise geometry of $\cM_{g,n}^r$, namely the description of the stabilizer groups of curves. Secondly, we will recall the equivariant \'etale local geometry of the stack, determining isotrivial generalizations. Finally, we will complete the study with the geometry of isotrivial degenerations.

Let us start with the classical definitions.

\begin{definition}
    A curve $C$ over an algebraically closed field $k$ is a connected, reduced, one-dimensional scheme proper over $k$. An $n$-pointed curve $(C,p_1,\dots,p_n)$ over a field $k$ is a curve $C$ together with $n$ rational smooth distinct points $p_1,\dots,p_n$. 
\end{definition}

\begin{definition}
    A closed point $p \in C$ over an algebraically closed field $k$ is called an \emph{$A_h$-singularity} if 
    $$ \widehat{\cO_{C,p}} \simeq k[[x,y]]/(y^2-x^{h+1}).$$
    If we do not need to emphasize the integer $h$, we will drop the subscript and write \emph{$A$-singularity}. An $A$-singularity $p$ of $C$ is \emph{separating} if the partial normalization of $C$ at $p$ is disconnected.
\end{definition}
    
\begin{definition}[Partial Normalization]
    Let $C$ be a curve over an algebraically closed field $k$, and let
    $S \subset \operatorname{Sing}(C)$ be a subset of its singular points.
    A finite morphism
    \[
    \nu \colon \widetilde{C} \to C
    \]
    is called the \emph{partial normalization of $C$ along $S$} if $\nu$ is an isomorphism over $C \setminus S$ and the induced map $\widetilde{C} \setminus \nu(S^c) \to C\setminus S^c$ is the normalization, where $S^c = \operatorname{Sing}(C) \setminus S$.
\end{definition}

\begin{definition}[Pointed partial normalization]
    \index{Pointed partial normalization}
    Let $(C,\{p_i\}_{i\in I})$ be a pointed curve, and let $S \subset \operatorname{Sing}(C)$ be a finite subset.
    The \emph{pointed partial normalization of $(C,\{p_i\})$ along $S$} is the triple $(\widetilde{C},\{q_j\}_{j\in J}, \nu)$ where 
    \[
    \nu \colon \widetilde{C} \to C
    \]
    is the partial normalization of $C$ along $S$ and $\{q_j\}_{j\in J} \coloneq \nu^{-1}\!\left(S \cup \{p_i \mid i \in I\}\right)$. If $S = \operatorname{Sing}(C)$, then $\nu$ is called the \emph{pointed normalization}.
\end{definition}

Notice that the normalization morphism is not a morphism between pointed curves; indeed, it sends markings to special points, which are either markings or singularities.

\begin{definition}[$A_r$-(pre)stable curve]\index{$A_r$-stable curve}
    Fix three nonnegative integers $g, n$, and $r$. Let $k$ be an algebraically closed field and $(C,p_1,\dots,p_n)/k$ be an $n$-pointed curve. We say that $(C,p_1,\dots,p_n)$ is $A_r$-prestable if every singular point is an $A_h$-singularity for some $h\leq r$. We will say that $(C,p_1,\dots,p_n)$ is \emph{$A_r$-stable} if it is $A_r$-prestable and the line bundle $\omega_C(p_1+\dots+p_n)$ is ample.
    If $(C,p_1,\ldots,p_n)$ is $A_{r}$-(pre)stable for some $r$, then we say that it is $A$-(pre)stable.
\end{definition}

\begin{remark}
    Notice that an $A$-prestable curve $C$ is l.c.i. by definition; therefore, the dualizing complex $\omega_C$ is in fact a line bundle. 
\end{remark}

\begin{remark}\label{rem:genus-count}
    Let $C$ be a curve over an algebraically closed field and let $p$ be an $A_r$-singularity.  We denote by $b:\widetilde{C}\rightarrow C$ the partial normalization at the point $p$ and by $J_b$ the conductor ideal of $b$. Then a classical computation shows that 
    \begin{enumerate}
        \item if $r=2h$, then $g(C)=g(\widetilde{C})+h$;
        \item if $r=2h+1$ and $\widetilde{C}$ is connected, then $g(C)=g(\widetilde{C})+h+1$,
        \item if $r=2h+1$ and $\widetilde{C}$ is not connected, that is $p$ is a separating singularity, then $g(C)=g(\widetilde{C})+h$.
    \end{enumerate}
    Furthermore, Noether's conductor formula gives us that $b^*\omega_C \simeq \omega_{\widetilde{C}}(J_b^{\vee})$. See for instance \cite[Proposition 1.2]{Cat}.
\end{remark}

We can define $\cM_{g,n}^r$ as the fibered category over $\kappa$-schemes whose objects are the data of $A_r$-stable curves over $S$ with $n$ distinct smooth sections $p_1,\dots,p_n$ such that every geometric fiber over $S$ is an $n$-pointed $A_r$-stable curve. These families are called \emph{$n$-pointed $A_r$-stable curves} over $S$. Morphisms are just isomorphisms of $n$-pointed curves.

We recall the following description of $\cM_{g,n}^r$. See \cite[Theorem 2.2]{Per1} for the proof. 

\begin{theorem}\label{theo:descr-quot}
    $\cM_{g,n}^r$ is a smooth connected algebraic stack of finite type over $\kappa$. Furthermore, it is a quotient stack: that is, there exists a smooth quasi-projective scheme X equipped with an action of $\GL_N$ for some positive integer $N$, such that 
    $ \cM_{g,n}^r \simeq [X/\GL_N]$.
\end{theorem}

\begin{remark}\label{rem: max-sing}
Recall that we have an open embedding $\cM_{g,n}^r  \subset \cM_{g,n}^s$ for every $r\leq s$. Notice that $\cM_{g,n}^r=\cM_{g,n}^{2g+1}$ for every $r\geq 2g+1$, because of \Cref{rem:genus-count}. Thus, we can consider $\cM_{g,n}^{2g+1}$ as the moduli stack of all $A$-stable curves (with no restriction on the singularities).
\end{remark}

\subsection{Equivariant geometry of \texorpdfstring{$A$}{A}-singularities}\label{sub:equivariant-geometry-A-singularities}

In this subsection, we describe the crimping datum needed to construct an $A_r$-stable curve starting from its normalization. This is a short summary of what has been proven in Section 3.1 and Section 3.2 of \cite{GMSArI}, and we hope it will help the reader who is not familiar with \emph{loc. cit.} to at least gain an intuition of what is needed for the main results of this work.

Contrary to the nodal and cuspidal case, it is not enough to remember only the closed points mapping to a given singularity if it is of type $A_l$ with $l\geq 3$. Depending on whether $l$ is even or odd, the equivariant geometry of $A_l$-singularities and their relation to the crimping datum varies, thus we treat the two cases separately in what follows.

\subsubsection*{Even $A_r$-singularities}

 We start by discussing the local geometry of even $A$-singularities, and more specifically the data one needs to construct an $A_{2h}$-singularity \emph{crimping} a smooth point of a curve, over an algebraically closed field, with $h\geq 1$. If $C$ is a curve with an $A_{2h}$-singularity $p$, then there exists a Cartesian diagram
 $$
 \begin{tikzcd}
\spec D_{2h} \arrow[d, "e"] \arrow[r, hook] & \widetilde{C} \arrow[d, "\nu"] \\
\spec D_h \arrow[r, hook]                   & C                             
\end{tikzcd}
$$
which is also a pushout, where $\nu$ is the partial normalization at $p$, the image of the lower horizontal morphism is $p$ in $C$ and $D_n:= k[t]/(t^n)$ for every $n\geq 0$. To construct $C$ from $\spec D_{2h} \subset \widetilde{C}$, we need to give a finite flat extension $q_h:D_h \hookrightarrow D_{2h}$ of degree $2$, up to isomorphisms of $D_h$. This is the same as giving a polynomial $q(t) \in k[t]/(t^{2h-2})$ such that $q(t)$ is odd: the extension $e$ of degree $2$ is determined by the association $t\mapsto t^2(1+q(t))$. This odd polynomial constitutes the datum needed to \emph{crimp} the point $\widetilde{p}$ in $\widetilde{C}$ to recover $C$; this is why it is called a \emph{crimping datum}. The space of crimping data will be denoted by ${\rm Cr}_{p}$: it is an affine space of dimension $h-1$.

From the partial normalization, we can relate the automorphisms of $(\widetilde{C},\widetilde{p})$ to the automorphisms of $(C,p)$: there is a natural injective group homomorphism $\Aut(\nu)^{\circ}:\Aut(C,p)^{\circ} \hookrightarrow \Aut(\widetilde{C},\widetilde{p})^{\circ}$. Fixing a crimping datum $e$ (and thus an induced curve $C$), we have that $\Aut(C,p)^{\circ}$ coincides with the elements in $\Aut(\widetilde{C},\widetilde{p})^{\circ}$ which commute with the crimping datum: indeed $\Aut(\widetilde{C}, \widetilde{p})^{\circ}$ acts naturally on ${\rm Cr}_p$ (by acting on $D_{2h}$) and $\Aut(C,p)^{\circ}$ is the stabilizer group of the point in ${\rm Cr}_p$ represented by the extension $e$. This can be upgraded to the following statement: 
\begin{itemize}
    \item[$(\star)$] the moduli stack of curves $C$ with an $A_{2h}$-singularity $p$ and fixed partial pointed normalization $(\widetilde{C},\widetilde{p})$ is isomorphic to $[{\rm Cr}_p/\Aut(\widetilde{C},\widetilde{p})^{\circ}]$.
\end{itemize}

This tells us a great deal about the possible automorphisms of $(C,p)$: for instance, if $p$ does not lie in an irreducible component $\widetilde{\Gamma}$ of $C$ whose normalization is of genus $0$, then the automorphism group does not change through the normalization, independently of which crimping datum we choose. On the contrary, suppose that there is a cocharacter $\Gm \rightarrow \Aut(\widetilde{C},\widetilde{p})^{\circ}$ which induces a non-trivial cocharacter for $\Aut(\widetilde{\Gamma},\widetilde{p})^{\circ}\simeq \Gm \ltimes \Ga$. It is not difficult to show that the $\Gm$-action induced on ${\rm Cr}_p$ has all positive (or negative) weights. Therefore, there exists a unique crimping datum, called the $\Gm$-equivariant crimping datum, which allows us to descend the non-trivial $\Gm$-action to $C$. The $\Gm$-equivariant crimping datum corresponds to the $0$-section of the crimping space ${\rm Cr}_p$.

\subsubsection*{Odd $A_r$-singularities}

 We proceed by discussing the local geometry of odd $A$-singularities. If $C$ is a curve with an $A_{2h+1}$-singularity $p$ for $h \geq 1$, then there exists a Cartesian diagram
 $$
 \begin{tikzcd}
\spec D_{h}\bigsqcup\spec D_h \arrow[d, "o"] \arrow[r, hook] & \widetilde{C} \arrow[d, "\nu"] \\
\spec D_h \arrow[r, hook]                   & C                             
\end{tikzcd}
$$
which is also a pushout, where $\nu$ is the partial normalization at $p$, the image of the lower horizontal morphism is $p$ in $C$ and $D_n:= k[t]/(t^{n+1})$ for every $n\geq 0$. To construct $C$ from $\spec D_{h} \bigsqcup\spec D_h \subset \widetilde{C}$, we need to give a finite \'etale extension $q_h:D_h \hookrightarrow D_h\times D_h$, or equivalently an automorphism of $D_h$. This is the same as giving a polynomial $q(t) \in k[t]/(t^{h})$ such that $q(0) \neq 0$: the morphism $o$ is determined by the association $t\mapsto tq(t)$. The polynomial $q(t)$ constitutes the datum needed to \emph{glue} the infinitesimal thickenings of the points $\widetilde{p}_1$ and $\widetilde{p}_2$ in $\widetilde{C}$ to recover $C$, and it is called a \emph{crimping datum}. The space of crimping data will be denoted by ${\rm Cr}_{p}$: it is of the form $\Gm \times \bA^{h-1}$.

From the partial normalization, we can relate the automorphisms of $(\widetilde{C},\widetilde{p}_1,\widetilde{p}_2)$ to the automorphisms of $(C,p)$: as in the even case, there is a natural injective group homomorphism $\Aut(\nu)^{\circ}:\Aut(C,p)^{\circ} \hookrightarrow \Aut(\widetilde{C},\widetilde{p}_1,\widetilde{p}_2)^{\circ}$. If we fix a crimping datum $o$ (and thus an induced curve $C$), we have that $\Aut(C,p)^{\circ}$ coincides with the elements in $\Aut(\widetilde{C},\widetilde{p}_1,\widetilde{p}_2)^{\circ}$ which commute with the crimping datum. This can be upgraded to the following statement:
\begin{itemize}
    \item[$(\star\star)$] the moduli stack of curves $C$ with an $A_{2h-1}$-singularity $p$ and fixed partial normalization $(\widetilde{C},\widetilde{p}_1,\widetilde{p}_2)$ is isomorphic to $[{\rm Cr}_p/\Aut(\widetilde{C},\widetilde{p}_1,\widetilde{p}_2)^{\circ}]$, where $(\widetilde{C},\widetilde{p}_1,\widetilde{p}_2)$ is the partial normalization of $C$ at $p$.
\end{itemize} 
We can again understand the possible automorphisms of $(C,p)$. Let $\widetilde{\Gamma}$ be the normalization of the minimal subcurve of $\widetilde{C}$ containing $\widetilde{p}_1$ and $\widetilde{p}_2$ and suppose that there is a cocharacter $\Gm \rightarrow \Aut(\widetilde{C},\widetilde{p}_1,\widetilde{p}_2)^{\circ}$ which induces a non-trivial cocharacter of $\Aut(\widetilde{\Gamma},\widetilde{p}_1,\widetilde{p}_2)^{\circ}$. Then $\Gm$ acts non-trivially on the factor $\Gm$ of ${\rm Cr}_p$ if and only if the $\Gm$-action does not descend to $C$. As in the even case, a crimping datum is called $\Gm$-equivariant if the cocharacter action descends to $C$. Notice that, contrary to the even case, the $\Gm$-equivariant crimping datum is not unique. Nevertheless, different crimping data may give the same curve (see for instance \cite[Lemma~3.37]{GMSArI}). 

\subsection{Hyperelliptic \texorpdfstring{$A_r$}{A\_r}-stable curves}\label{sub:hyper-A-stable}

After introducing the crimping datum for $A$-singularities, we shift our focus to hyperelliptic $A_r$-stable curves. Although seemingly unrelated, there exists a strong connection between $A$-singularities and (honestly) hyperelliptic $A_r$-stable curves (see \Cref{sub:local-geometry} for a more detailed discussion).

\begin{definition}
    A hyperelliptic $n$-pointed curve is a tuple $(C,p_1,\dots,p_n,\sigma)$ where $(C,p_1,\dots,p_n)$ is an $n$-pointed curve and $\sigma$ is a hyperelliptic involution, i.e. a two-torsion element in $\Aut(C)$ such that 
    \begin{itemize}
        \item the scheme of fixed points of $\sigma$ is finite;
        \item the geometric quotient $Z:=C/\sigma$ is a nodal curve of genus $0$.
    \end{itemize}
    We say that a point $q \in C$ is a Weierstrass (possibly singular) point if it is fixed by the involution. If $Z$ is irreducible, then we say that $C$ is an $n$-pointed \emph{honestly hyperelliptic} $A_r$-stable curve, namely a cyclic cover of $\bP^1$ of degree $2$.
\end{definition}

The moduli stack parametrizing $n$-pointed hyperelliptic $A_r$-stable curves will be denoted by $\cH_{g,n}^r$. In \cite{Per2}, it is proven that it is a closed substack of $\cM_{g,n}^r$ and coincides with the closure of the locus of smooth hyperelliptic curves. The substack of $\cH_{g,n}^r$ classifying honestly hyperelliptic curves $\cH_{g,n}^{r,\circ}$ is a dense open inside $\cH_{g,n}^r$. It is important to keep in mind the following statement: 
\begin{itemize}
    \item[(h)] every honestly hyperelliptic curve can be identified with the branching divisor of length $2g+2$ in $\bP^1$ associated with the cyclic cover of degree $2$.
\end{itemize}
A point of length $h$ in this divisor corresponds to an $A_{h-1}$-singularity in the (honestly) hyperelliptic curve. This simple observation helps us describe stacks of pointed cyclic covers of $\bP^1$. 

Let us start by considering the case $n=1$. We introduce a stratification of $\cH_{g,1}^{r,\circ}$. We denote by $\cH_{g,w}^{r,\circ} \subset \cH_{g,1}^{r,\circ}$ the closed substack parametrizing pairs $(C,p,\sigma)$ such that $p$ is a smooth Weierstrass point, i.e. a point fixed by the involution. We denote by $\cH_{g,g_1^2}^{r,\circ}$ its open complement, which can be identified with the moduli stack parametrizing $A_r$-stable honestly hyperelliptic curves of genus $g$ with two smooth and distinct markings which are exchanged by the involution. Indeed, if the marked point is not Weierstrass, then its orbit under the involution contains exactly one more point.
\begin{proposition}\label{prop:w-global-descr}
    If $r\geq 2g$, the moduli stack $\cH_{g,w}^{r,\circ}$ is isomorphic to the quotient $[\bA^{2g-1}/\Gm]$, where $\Gm$ acts on $\bA^{2g-1}$ with positive weights. 
\end{proposition}
The $0$-section of the $\Gm$-representation corresponds to the unique closed point of $\cH_{g,w}^{r,\circ}$, which, by construction, is the curve obtained as a cyclic cover of degree $2$ of $\bP^1$ with the branching divisor of length $1$ at $0$ and $2g+1$ at $\infty$: by the discussion in \Cref{sub:equivariant-geometry-A-singularities}, this is the curve obtained by pinching $\bP^1$ at $\infty$ using a $\Gm$-equivariant crimping datum, and the connected component of the identity of the automorphism group of such a curve is a one-dimensional torus.

For the non-Weierstrass case, we have a similar result.
\begin{proposition}\label{prop:g12-global-descr}
    If $r\geq 2g+1$, the moduli stack $\cH_{g,g_1^2}^{r,\circ}$ is isomorphic to the quotient $[\bA^{2g+1}/\Gm]$, where $\Gm$ acts on $\bA^{2g+1}$ with positive weights. 
\end{proposition}
The $0$-section of the $\Gm$-representation corresponds to the unique closed point of $\cH_{g,g_1^2}^{r,\circ}$, which, by construction, is the curve obtained as a cyclic cover of degree $2$ of $\bP^1$ with branching divisor of length $2g+2$ at $\infty$: by the discussion in \Cref{sub:equivariant-geometry-A-singularities}, this is the curve obtained by gluing two copies of $\bP^1$ at $\infty$ using a $\Gm$-equivariant crimping datum (although the crimping datum is not unique, the curve obtained in this case is), and the connected component of the identity of the automorphism group of such a curve is a one-dimensional torus.

Finally, we briefly recall the case $n=0$. This is more complicated: it was shown in \cite[Theorem 3.44]{GMSArI} that, to have a good moduli space, one needs to avoid cyclic covers of $\bP^1$ with branch divisors having a single point of length $>g+1$. In the search for a maximal open substack with separated good moduli space, this forces us to remove what we denote as $A_h$-self-attached hyperelliptic curves of genus $g<h$, as explained in \Cref{lem:self-attached-condition}. Indeed, the problem comes from the cyclic cover of $\bP^1$ whose branch divisor is concentrated at one point: the curve has a non-reductive automorphism group, thus it needs to be removed. With it, we are forced to remove some of its basins of attraction, the self-attached hyperelliptic curves. 

\subsection{Pointwise geometry}\label{sub:pointwise geometry}

In this section, we recall the structure of the curves that have positive-dimensional stabilizers. They will play a crucial role in the rest of the paper. We start by recalling the following definition.
\begin{definition}\label{def:hyp-even-sing}
    An $A_m$-singularity $p$ of an $n$-pointed $A_r$-stable curve $(C,p_1,\dots,p_n)$ with $m\geq 2$ is called \emph{hyperelliptic} if there exists a cocharacter $\Gm \rightarrow \Aut(C,p_1,\dots,p_n)^{\circ}$ such that $\Gm$ acts non-trivially on $T_pC$.
\end{definition}

The reason why we call the singularity \emph{hyperelliptic} will be explained in what follows. We will start by describing the minimal $A_r$-stable curves that contribute to the dimension of the stabilizer group, which we call an \emph{atom}. The term \emph{atom} was first used in \cite{AlpFedSmyWyck}. 

Let $r\geq 2h$ with $h$ a positive integer. Consider $(\bP^1,0,\infty)$ and define $(C,q)$ to be the $1$-pointed $A_r$-stable curve obtained by crimping $\bP^1$ at $\infty$ using the $\Gm$-equivariant crimping datum forming an $A_{2h}$-singularity (so that the action of $\Gm\simeq \Aut(\bP^1,0,\infty)$ descends to $C$), and denote by $p$ the $A_{2h}$-singularity and by $q$ the image of $0$. Notice that the singularity $p$ is by definition hyperelliptic. As a motivation for such a definition, this curve coincides with the cyclic cover of degree $2$ of $\bP^1$ described by the branch divisor with a point of length one (corresponding to the Weierstrass marking) and a point of length $2g+1$ (see \cite[Lemma 4.6]{GMSArI}). 
\begin{definition}[Even atom]
    \label{def:even-atoms}\index{even atom}
    The curve constructed above will be called the $1$-pointed \emph{even atom} of genus $h$.
\end{definition}
\begin{figure}[H]
        \centering
        \includegraphics[width=0.4\textwidth]{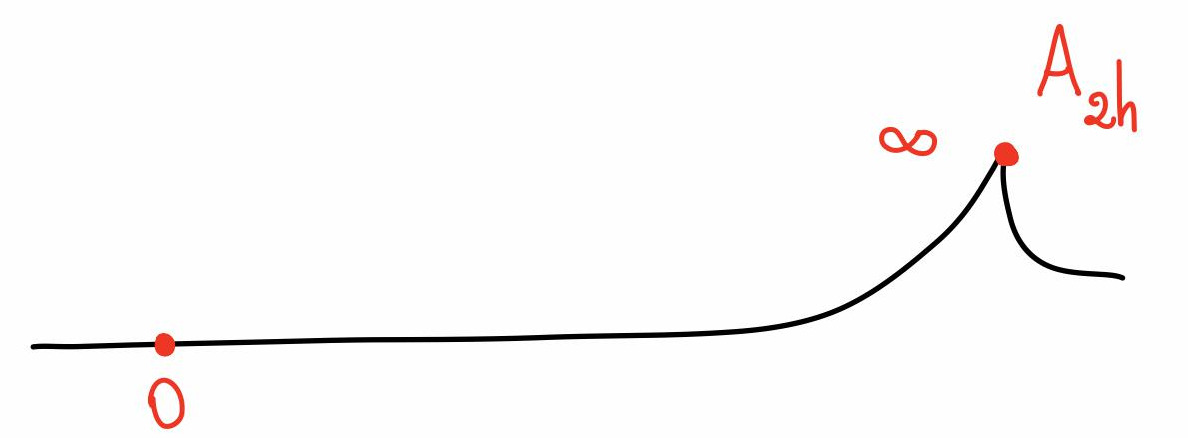}
        \caption{Even atom}
        \label{fig:even-atom}
\end{figure}

We deal now with the odd case. Consider two copies of $(\bP^1,0,\infty)$ and define $(C,q_1,q_2)$ to be the $2$-pointed $A_r$-stable curve obtained by identifying the two infinitesimal thickenings of length $h+1$ of the two copies of $\bP^1$ at $\infty$ using a $\Gm$-equivariant crimping datum forming an $A_{2h+1}$-singularity (so that the action of $\Gm\simeq \Aut(\bP^1,0,\infty)$ descends to $C$), and denote by $p$ the $A_{2h+1}$-singularity and by $q_1,q_2$ the images of the two $0$-sections of the two copies of $\bP^1$. Thanks to \cite[Lemma~3.37]{GMSArI}, we have that such a curve is unique up to isomorphism. Notice that the singularity $p$ is by definition hyperelliptic. As a motivation for such a definition, this curve coincides with the cyclic cover of degree $2$ of $\bP^1$ described by the branch divisor with a point of length $2g+2$ (see \cite[Lemma 4.9]{GMSArI}). 

\begin{definition}[Odd atom]\index{Odd atom}
    \label{def:odd-atoms}
    The curve constructed above will be called the $2$-pointed \emph{odd atom} of genus $h$. 
\end{definition}
\begin{figure}[H]
        \caption{}
        \centering
        \includegraphics[width=0.4\textwidth]{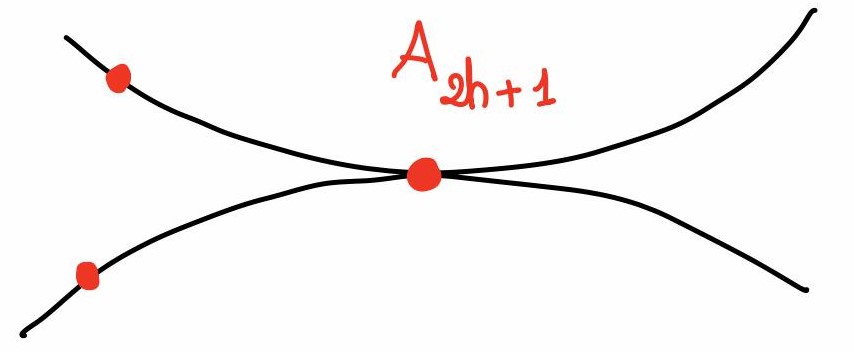}
        \label{fig:odd-atom}
    \end{figure}

\begin{remark}
    Both in the even and odd cases, we can similarly define a $1$-pointed atom or a $0$-pointed atom. We leave the details on how to do this to the interested reader; see \cite[Section 4.1]{GMSArI}.
\end{remark}
Before recalling the classification of positive-dimensional stabilizer curves in $\cM_{g,n}^r$, we recall some classical notions related to the combinatorics of curves.

\begin{definition}[Subcurve]\index{subcurve}
    Let $C$ be a curve over $k$. A \emph{subcurve} $\Gamma$ of $C$ is a proper, reduced, one-dimensional subscheme of $C$, i.e. the union of some irreducible components of $C$ with the reduced scheme structure. 
    
    Let $(C,p_1,\dots,p_n)$ be an $n$-pointed $A$-prestable curve over an algebraically closed field $k$ and let $\Gamma\subset C$ be a subcurve. We denote by $C-\Gamma$ the subcurve of $C$ induced by the closure of $C\setminus \Gamma$ in $C$. Moreover, we denote by $I_{\Gamma}$ the set of points $\Gamma \cap (C-\Gamma)$ of $\Gamma$ and by $P_{\Gamma}$ the subset of the markings $\{p_1\dots,p_n\}$ that lie in the subcurve $\Gamma$. Notice that every point in $I_{\Gamma}$ is smooth in $\Gamma$ because there are only $A$-singularities. 
\end{definition}

\begin{definition}[Tails and bridges]\index{tails}\index{bridges}
     Let $(C,p_1,\dots,p_n)$ be an $n$-pointed $A$-prestable curve over an algebraically closed field $k$ and let $\Gamma\subset C$ be a connected subcurve. We say that $(\Gamma, I_{\Gamma}\cup P_{\Gamma})$ is a \emph{tail} if $\vert I_{\Gamma}\cup P_{\Gamma} \vert =1$, whereas we say that it is a \emph{bridge} if $\vert I_{\Gamma}\cup P_{\Gamma} \vert =2$.   We say that a bridge $\Gamma$ is separating if the curve $C-\Gamma$ is disconnected; otherwise, we say it is non-separating.
\end{definition}

Notice that in the previous definition of a subcurve, we do not require $\Gamma$ to be properly contained in $C$. Therefore, given a $1$-pointed curve $(C,p)$, we have that $(C,p)$ itself is a tail. The same is true for bridges and $2$-pointed curves.

\begin{definition}\index{Outer singularity} \index{Inner singularity} \index{Separating singularity} \index{Lonely singularity}
    We say that an $A$-singularity $p$ of $C$ is:
    \begin{itemize}
        \item[(a)]  \emph{outer} if it lies in the intersection of (exactly) two irreducible components of $C$;
        \item[(b)] \emph{lonely} if it is outer and it is the unique (set-theoretic) intersection of the two irreducible components;
        \item[(c)] \emph{separating} if the partial normalization at $p$ is disconnected;
        \item[(d)] \emph{inner} if it is not outer.
    \end{itemize}
    Clearly, $(c)\implies (b)\implies (a)$ and $p$ has to be an odd $A$-singularity, i.e. an $A_{2i+1}$-singularity for some $i\geq 0$.
\end{definition}

It will be useful to study degenerations of $A$-prestable curves, specifically when they are generically outer. The following proposition will play a fundamental role in the paper, thus we report it here.

\begin{proposition}[{\cite[Prop.~2.10]{AlpFedSmyWyck}}]\label{prop:alp-fund}
Let $C\rightarrow \spec R$ be a family of $A$-prestable curves over a dvr $R$ and let $p$ be a section of the family. Suppose that the (geometric) generic fiber $p_{Q}$ is an outer $A_{2k_1+1}$-singularity. Then 
\begin{itemize}
    \item the (geometric) special fiber $p_k$ is an outer $A_{2m+1}$-singularity;
    \item each singularity of the (geometric) generic curve $C_Q$ that approaches $p_k$ must be outer as well and must lie on the same two irreducible components of $C_Q$ as $p_Q$;
    \item the collection of such (generic) singularities approaching $p_k$ is necessarily of the form\newline $\{A_{2k_1+1},\dots, A_{2k_s+1}\}$
    where $m=s-1+\sum_{i=1}^sk_i$;
\end{itemize}
\end{proposition}

\begin{remark}
    With the notation of \Cref{prop:alp-fund}, if $p_Q$ is lonely (respectively separating), then $p_k$ will also be lonely (respectively separating). 
\end{remark}

We are finally ready to recall the classification of positive-dimensional stabilizer curves in $\cM_{g,n}^r$. It is crucial that the reader becomes familiar with all the different cases, as they are going to appear repeatedly in the rest of the paper. 

First of all, notice that if $(C,p_1,\dots,p_n)$ is an $n$-pointed $A_r$-stable curve and $q$ is a node, then the pointed partial normalization does not change the connected component of the identity of the automorphism group: this is due to the fact that in the push-out construction of the node, there are no infinitesimal thickenings, and thus no balancing condition is needed to show that the automorphisms descend through the normalization. Therefore, given an $n$-pointed $A_r$-stable curve $(C,p_1,\dots,p_n)$, we can normalize along the outer nodes to get a disjoint union of connected components $(\Gamma_i, P_{\Gamma_i} \cup I_{\Gamma_i})$ such that
$$ \Aut(C,p_1,\dots,p_n)^{\circ} \simeq \prod_{i \in I} \Aut(\Gamma_i, P_{\Gamma_i} \cup I_{\Gamma_i})^{\circ}. $$

The subcurves $\Gamma_i$ that have positive-dimensional automorphism groups are the ``molecules'' we can construct using our atoms and $A$-singularities. The combinatorics of the molecules is not too complicated, depending on the special points lying in $\Gamma_i$. For instance, if $r\leq 2g$, \cite[Proposition 4.23]{GMSArI} gives us that to have a non-trivial positive-dimensional automorphism group, $\Gamma_i$ has to verify $n_i:=|P_{\Gamma_i} \cup I_{\Gamma_i}|\leq 2$. Moreover, the same result tells us that $\Aut(\Gamma_i,  P_{\Gamma_i} \cup I_{\Gamma_i})^{\circ}$ is a one-dimensional torus (if not trivial). Notice that this implies that, except when $r=2g+1$, the automorphism group of an $A_r$-stable curve is an extension of a finite \'etale group by a torus. This is one of the fundamental properties of the stack $\cM_{g,n}^r$, used extensively throughout the 
three papers in the series.

Let us now show the geometry of such subcurves, which will be called \emph{rosaries} of hyperelliptic $A$-singularities: they are a straightforward generalization of \cite[Def.~6.1]{hassett2013log}. Notice that by construction, they do not have outer nodes. Let us start with the $2$-pointed case, namely $n_i=2$. 

Consider a $2$-pointed odd atom and recall that the connected component of the identity of the automorphism group is $\Gm$. We now glue the atom together with a $2$-pointed $\bP^1$ at one of the markings using an odd (worse-than-nodal) $A$-singularity with a $\Gm$-equivariant crimping datum, so that the automorphism group of the $2$-pointed curve obtained is an extension of a finite group by $\Gm$. Although the choice of such a crimping datum is not unique, the curve obtained is. We can continue to attach $2$-pointed $\bP^1$ with $\Gm$-equivariant crimping data to obtain a chain of $2$-pointed $\bP^1$ glued together by worse-than-nodal odd $A$-singularities. 

\begin{definition}
    The curve constructed above will be called a $2$-pointed \emph{rosary} of hyperelliptic $A$-singularities. The length of the rosary is the number of irreducible components. A similar definition holds for $1$-pointed and $0$-pointed rosaries. An $A_r$-stable curve $C$ is said to be a closed rosary of $A$-singularities if there exists a worse-than-nodal odd $A$-singularity $p$ such that the pointed partial normalization of $C$ at $p$ is a $2$-pointed rosary of hyperelliptic $A$-singularities.
\end{definition}

Notice that the singularities of the closed rosary are not necessarily hyperelliptic. Indeed, if they are, the length of the rosary should be even, otherwise the $\Gm$-action of the $2$-pointed rosary does not descend. 

\begin{figure}[H]
        \centering
        \includegraphics[width=0.7\textwidth]{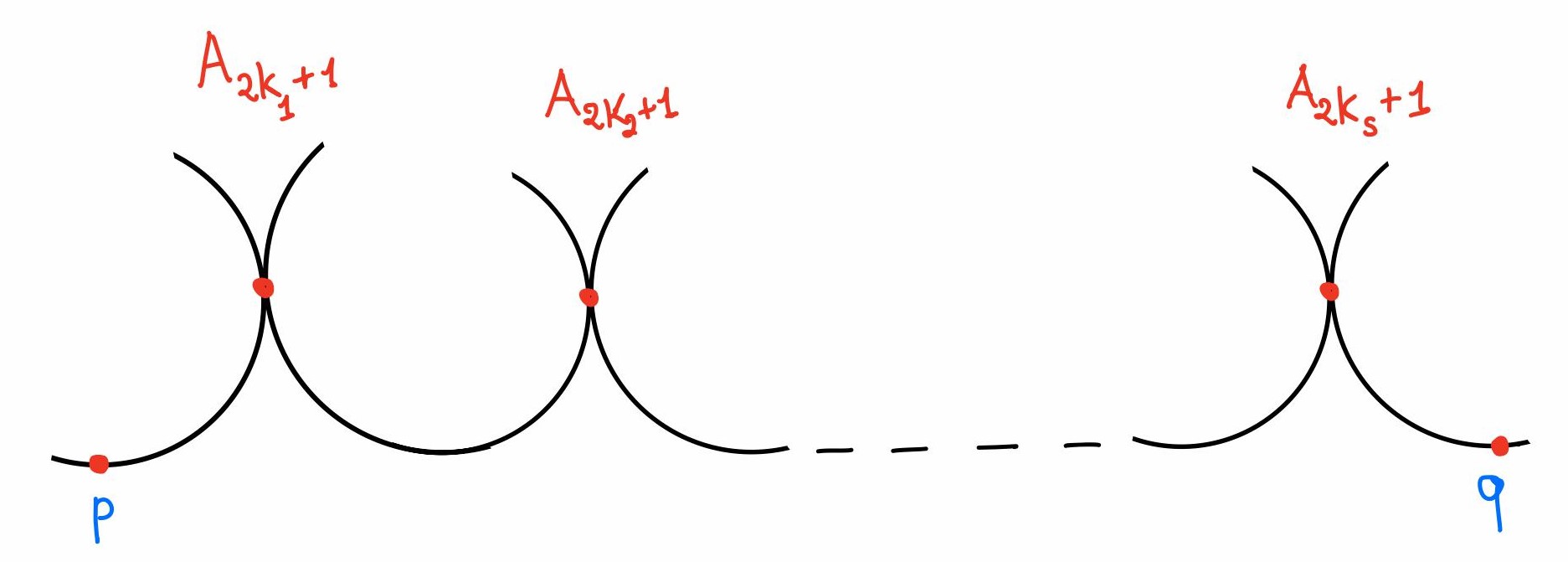}
        \label{fig:attached-rosary}
        \caption{Example of a $2$-pointed rosary of $A$-singularities}
    \end{figure}

Now that we have defined the molecules, it is important that we understand how they lie inside the curve. For instance, suppose $\Gamma \subset C$ is a subcurve such that $P_{\Gamma}=\emptyset$ and $I_{\Gamma}$ consists of two (outer) nodes, and assume $(\Gamma,I_{\Gamma})$ is a $2$-pointed rosary of hyperelliptic $A$-singularities. Then we say that $\Gamma$ is an $A_1/A_1$-attached rosary of (hyperelliptic) $A$-singularities of $C$. 

\begin{definition}
    We say that a subcurve $\Gamma \subset C$ is an $A_{k_1}/A_{k_2}$-attached rosary of $A$-singularities
    \begin{itemize}
        \item[1)] if $k_1$ and $k_2$ are odd, then $P_{\Gamma} = \emptyset$ and $I_{\Gamma}$ consists of an $A_{k_1}$-singularity and an $A_{k_2}$-singularity of $C$ (possibly equal) and $(\Gamma,I_{\Gamma}\cup P_{\Gamma})$ is a (possibly closed) rosary of hyperelliptic $A$-singularities;
        \item[2)] if $k_1$ is even and $k_2$ is odd, then there exists an $A_{k_1}$-singularity $q \in \Gamma$ such that the pointed partial normalization of $(\Gamma, I_{\Gamma} \cup P_{\Gamma}) $ at $q$ is a rosary of $A$-singularities;
        \item[3)] if $k_1$ and $k_2$ are both even, there exists an $A_{k_1}$-singularity $q_1\in \Gamma$ (respectively an $A_{k_2}$-singularity $q_2 \in \Gamma$) such that the pointed normalization of $(\Gamma, I_{\Gamma} \cup P_{\Gamma})$ at $q_1$ and $q_2$ is a $2$-pointed rosary of $A$-singularities. 
    \end{itemize}
    By abuse of notation, we extend the definition to the case $k_1=0$ by defining an $A_0$-singularity as a smooth marking. If the two attaching singularities are hyperelliptic, then we say that $\Gamma$ is an $A_{k_1}/A_{k_2}$-attached rosary of hyperelliptic $A$-singularities.
\end{definition}

\begin{remark}
Notice that
    \begin{itemize}
        \item if $C$ has an $A_{k_1}/A_{k_2}$-attached rosary of hyperelliptic $A$-singularities with $k_1$ and $k_2$ even, then $C=\Gamma$;
        \item the $1$-pointed even atom of genus $h$ is an $A_0/A_{2h}$-attached rosary of hyperelliptic $A$-singularities of length $1$;
        \item the $2$-pointed odd atom of genus $h$ is a $2$-pointed rosary of hyperelliptic $A$-singularities of length $2$.
    \end{itemize}
\end{remark}

We finally recall the following result; for a complete description of stabilizer of points in $\cM_{g,n}^r$ see \cite[Corollary~4.26]{GMSArI}.
\begin{proposition}\label{prop:gm-decomp-aut}
    Let $(C,p_1,\dots,p_n)$ be an $n$-pointed $A_r$-stable curve and suppose $r \leq 2g$. Then $\Aut(C,p_1,\dots,p_n)^{\circ}$ is an extension of a finite group by a torus. Every \ $A_i/A_j$-attached rosary of hyperelliptic $A$-singularities contributes with a copy of $\Gm$, whenever $i, j \in \{0, 1\}$.
\end{proposition}

If $r=2g+1$, then the automorphism group may be non-reductive, see for instance \cite[Proposition 4.20]{GMSArI}. The only problematic curves appear in the case $n=0$: for instance, the automorphism group of the $0$-pointed odd atom of genus $g$ is isomorphic to $(\Gm\ltimes \Ga)\ltimes C_2$. As explained already in \Cref{sub:hyper-A-stable}, this forces us to remove the curves mentioned in condition $(e_r)$ of \Cref{def:admin}.

\subsection{Local geometry}\label{sub:local-geometry}

In this section, we will discuss the equivariant deformation theory of $A_r$-stable curves, as previously done in \cite[Section 4.3]{GMSArI}. It will provide yet another strong connection between $A$-singularities and honestly hyperelliptic curves, i.e., cyclic covers of degree $2$ of $\bP^1$.

In \Cref{sub:pointwise geometry}, we recalled the classification of curves with positive-dimensional stabilizers. 
We are now going to recall the equivariant deformation theory of these curves. The goal is to explain briefly the classification of isotrivial degenerations (carried out in \cite[Section 4.3]{GMSArI}). Before we explain the classification, we recall the following notation for $\Gm$-representations.

\begin{definition}\label{def:multi-rep-notation}
    Let $V$ be a $\bG_m$-representation. We denote by $V^>$ (respectively $V^{\geq}$, $V^<$, $V^{\leq}$) the subrepresentation of $V$ defined as the sum of the irreducible subrepresentation of $V$ with positive (respectively non negative, negative and non positive) weights. We denote by $V^0$ the subspace of $\bG_m$-invariants vectors inside $V$.
\end{definition}

Let us start with the even case. Let $(C,q)$ be a $1$-pointed even atom of genus $h$, and let $p$ be the $A_{2h}$-singularity. We have the following $\Gm$-equivariant decomposition:
$$ \Def_{(C,q)} \simeq T^1_p \oplus {\rm Cr}_p $$
where $T^1_p$ is the first order deformation space of the singularity $p$ and ${\rm Cr}_p$ is its crimping space, as introduced in \Cref{sub:equivariant-geometry-A-singularities} (see \cite[Proposition~4.30]{GMSArI}). The decomposition is not only $\Gm$-equivariant but also \emph{sign-preserving}, meaning that, if $\Gm=\Aut(C,q)^{\circ}$ acts positively on the tangent space of $q$, then  $T_p^1$ (respectively ${\rm Cr}_p$) coincides with the positively graded (respectively negatively graded) isotypic component of $\Def_{C,q}$. This determines the two basins of attraction of the $1$-pointed even atom: given an isotrivial degeneration with special point $(C,q)$, we have that the generic point lies in $T_p^1\oplus 0$ or in $0\oplus {\rm Cr}_p$. This alternating behavior is the reason why the geometry of isotrivial degenerations is complicated, but not out of reach.

Finally, \cite[Proposition 4.29]{GMSArI} gives the promised bridge between (deformations of) $A_r$-singularities and honestly hyperelliptic curves. Indeed, it states that the tangent space of $\cH_{h,w}^{2h,\circ}$ in $\cM_{2h,1}^{2h}$ at $(C,p)$ is identified with the sub-representation $T_p^1\oplus 0$. Therefore, if we smoothen the singularity $p$ along an isotrivial degeneration, the curve remains a cyclic cover of $\bP^1$ of degree $2$ with a smooth Weierstrass marking. A similar result holds for the $2$-pointed odd atom of genus $h$, with the generic point of the isotrivial degeneration being a cyclic cover of $\bP^1$ where the two markings are exchanged by the involution.

Now that we have dealt with the case of the atoms, we can focus on understanding the deformation theory of the curves with positive-dimensional automorphism group, specifically the subcurves that contribute to its dimension.
Let $C_0$ be a $2$-pointed rosary of hyperelliptic $A$-singularities of length $3$ and denote by $q_1,q_2$ the two singularities of type $A_{2h_1+1}$ and $A_{2h_2+1}$ respectively. Moreover, suppose there is an isotrivial degeneration $C_{\Theta}\rightarrow \Theta$ whose special fiber is isomorphic to $C_0$. Then the combinatorial structure of the generic fiber $C_1$ is controlled by that of $C_0$. In fact, we have a decomposition 
$$\Def_C = T^1_{q_1}\oplus T^1_{q_2}\oplus {\rm Cr}_{q_1}\oplus {\rm Cr}_{q_2}$$
of the deformation space of $C$. Furthermore, the isotrivial degeneration induces a cocharacter $\Gm \rightarrow \Aut(C_0)\simeq \Gm$ which also induces a sign decomposition of $\Def_{C_0}$. Because of the explicit description of the action of $\Aut(C_0)^{\circ}$ on $C_0$, we have that the $\Gm$-action on the tangent space of $q_1$ has opposite signs compared to the one on the tangent space of $q_2$:
$$ (\Def_C)^> =T^1_{q_1}\oplus{\rm Cr}_{q_2} $$
and 
$$ (\Def_C)^<=T^1_{q_2}\oplus {\rm Cr}_{q_1}$$
or vice versa. Therefore, if the isotrivial degeneration deforms one singularity, then the other one needs to be preserved (with a possibly different crimping datum) in the generic fiber $C_1$. Moreover, the deformation of the odd atom of genus $h_2$ in the chain will be an (honestly) hyperelliptic curve of genus $h_2$. 
    \begin{figure}[H]
        \caption{}
        \centering
        \includegraphics[width=0.8\textwidth]{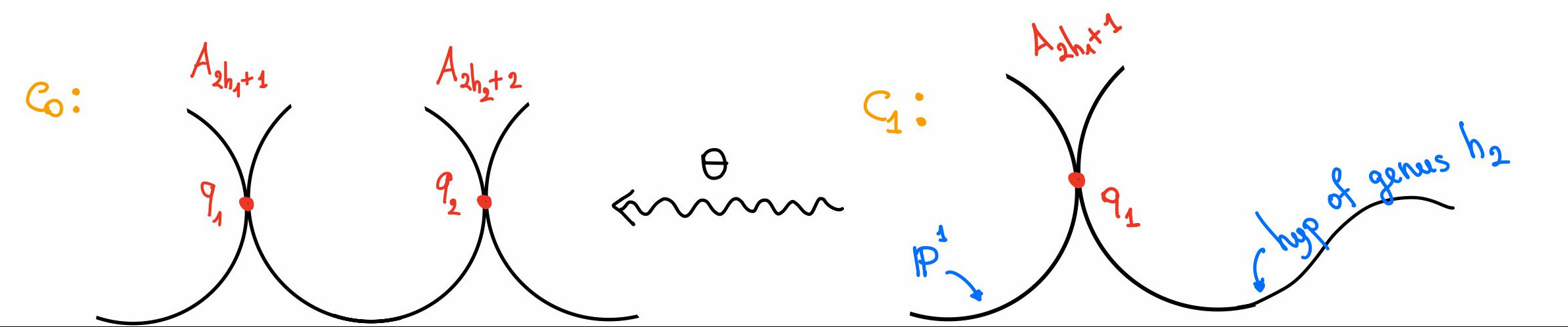}
        \label{fig:deg-hyp-chain}
    \end{figure} 

The same alternating strategy works if the length of the chain is greater than $3$. Suppose that $C_0$ is a rosary of hyperelliptic $A$-singularities $q_1,\dots,q_n$ where $q_i$ and $q_{i+1}$ share one irreducible component for every $i=1,\dots, n-1$. Moreover, suppose that the isotrivial degeneration deforms $q_{i_0}$ where $i_0$ is odd. Then the singularities $q_2,q_4,\dots,q_{2k}$ for $k=\lfloor n/2 \rfloor$ will also appear in the generic fiber $C_1$ of the isotrivial degeneration (possibly with a different crimping datum), while the odd atom of genus $h_{2i+1}$ associated with the singularity $q_{2i+1}$ will deform to an (honestly) hyperelliptic curve of genus $h_{2i+1}$ where the two attaching points are exchanged by the involution. A similar example can be constructed starting with an $A_0/A_{2k}$-attached rosary of hyperelliptic $A$-singularities of length $3$ (see \cite[Example 4.34]{GMSArI}).

Before going further, it is crucial to point out that the following remark is a key point in our study of isotrivial degenerations of $A$-stable curves. Indeed, it gives a significant constraint to the geometry of isotrivial degenerations, and it will play a fundamental role in all that follows.

\begin{remark}\label{rem:alternating-def}
    The discussion above implies that any isotrivial degeneration cannot deform two adjacent $A$-singularities. If one is deformed, then the ones adjacent to it have to be preserved, and the curve between them is honestly hyperelliptic. 
    
    This is true also for nodes, with the extra hypothesis that the $\Gm$-action is trivial on one of the two components meeting at the node. Namely, suppose we are given a curve $C_0$ with a subcurve $\Gamma\subset C_0$ that is an $A_1$-attached even atom (or equivalently an $A_1/A_{2h}$-attached rosary of length $1$), and suppose we have an isotrivial degeneration to $C_0$ such that the cocharacter $\Gm \rightarrow \Aut(C_0)$ acts trivially on the component $\Gamma'\subset C_0-\Gamma$ that intersects $\Gamma$ at the node. Then, if the singularity of the atom is deformed, the node has to be preserved, and $\Gamma$ deforms to an (honestly) hyperelliptic tail attached to the rest of the curve at a node. 
\end{remark}

\begin{definition}[Hyperelliptic tails and bridges]\label{def:tails and bridges}
Let $(C,p_1,\dots,p_n)$ be an $n$-pointed $A$-prestable curve over an algebraically closed field $k$, and let $\Gamma\subset C$ be a connected subcurve.
    \begin{itemize}
        \item we say that $\Gamma$ is an $A_{2k+1}$-attached hyperelliptic tail if $\Gamma$ is a hyperelliptic curve of positive genus, $P_{\Gamma}=\emptyset$, and $I_{\Gamma}$ consists of a single $A_{2k+1}$-singularity on $C$ which corresponds to a smooth Weierstrass point on $\Gamma$ (i.e., a point fixed by the hyperelliptic involution of $\Gamma$); we say that it is $A_0$-attached if $I_{\Gamma}=\emptyset$ (thus $\Gamma=C$) and $P_{\Gamma}$ consists of a single point (thus $n=1$);

        \item  we say that $\Gamma$ is an $A_{2h+1}/A_{2k+1}$-attached hyperelliptic bridge if $\Gamma$ is a hyperelliptic curve of positive genus, $P_{\Gamma}=\emptyset$, and $I_{\Gamma}$ consists of an $A_{2h+1}$-singularity and an $A_{2k+1}$-singularity on $C$ which correspond to two smooth points on $\Gamma$ exchanged by the involution.  
    \end{itemize}
\end{definition}
Exactly as for the rosaries, we need to define the notion of $2$-pointed hyperelliptic chains.
\begin{definition}
We say that $(\Gamma,p_1,p_2)$ is a $2$-pointed hyperelliptic chain if it is constructed by gluing $2$-pointed hyperelliptic bridges at odd $A$-singularities, creating a chain of attached hyperelliptic bridges. If the attached bridges are honestly hyperelliptic, we will say that $(\Gamma,p_1,p_2)$ is a \emph{chain of honestly hyperelliptic curves}. The same definition can be made for $1$-pointed or $0$-pointed chains.
We say that the hyperelliptic chain $(\Gamma,p_1,p_2)$ has \emph{genus bounded by $g$} if every subcurve that is a hyperelliptic bridge has genus \textbf{strictly less} than $g$. Moreover, the number of $2$-pointed hyperelliptic bridges is called the \emph{length} of $\Gamma$.
\end{definition}

\begin{definition}[Attached hyperelliptic chain]
\label{def:attached-hyp-chain}
    Let $(C,p_1,\dots,p_n)$ be an $A_r$-stable curve and $\Gamma\subset C$ a subcurve. We say that $\Gamma$ is an $A_{k_1}/A_{k_2}$-attached hyperelliptic chain if
    \begin{itemize}
        \item[1)] if $k_1$ and $k_2$ are odd, then $P_{\Gamma} = \emptyset$ and $I_{\Gamma}$ consists of an $A_{k_1}$-singularity and an $A_{k_2}$-singularity of $C$ (possibly equal), and $(\Gamma,I_{\Gamma}\cup P_{\Gamma})$ is a (possibly closed) hyperelliptic chain;
        \item[2)] if $k_1$ is even and $k_2$ is odd, then there exists an $A_{k_1}$-singularity $q \in \Gamma$ such that the pointed partial normalization of $(\Gamma, I_{\Gamma} \cup P_{\Gamma}) $ at $q$ is a hyperelliptic chain;
        \item[3)] if $k_1$ and $k_2$ are both even, there exists an $A_{k_1}$-singularity $q_1\in \Gamma$ (respectively, an $A_{k_2}$-singularity $q_2 \in \Gamma$) such that the pointed normalization of $(\Gamma, I_{\Gamma} \cup P_{\Gamma})$ at $q_1$ and $q_2$ is a $2$-pointed hyperelliptic chain. 
    \end{itemize}
    By abuse of notation, we extend the definition to the case $k_1=0$ by defining an $A_0$-singularity as a smooth marking. The attached subcurve is said to be \emph{separating} if $C-\Gamma$ is disconnected, otherwise it is called \emph{non-separating}.
\end{definition}

\begin{remark}
    If the attached hyperelliptic chain has length $1$, we interchangeably call it a bridge (even though the attachments may be even singularities).
\end{remark}

The following proposition (see \cite[Proposition 4.38]{GMSArI}) explains the geometry of the generic point of an isotrivial degeneration to a $2$-pointed rosary of hyperelliptic $A$-singularities.
\begin{proposition}\label{prop:defor-rational-chain}
    Let $(C_0,p_1^0,p_2^0)$ be a $2$-pointed rosary of hyperelliptic (odd) $A$-singularities and $(C_1,p_1^1,p_2^1)$ be a $2$-pointed $A_r$-stable curve. If $(C_1,p_1^1,p_2^1)$ admits an isotrivial degeneration to the curve $(C_0,p_1^0,p_2^0)$, then it is one of the curves described below, depending on the parity of the length of the rosary. Suppose that the length of the rosary is even, then $(C_1,p_1^1,p_2^1)$ is either
    \begin{itemize}
        \item[e1)] a $2$-pointed chain of honestly hyperelliptic curves;
        \item[e2)] an $A_{2h+1}/A_{2k+1}$-attached chain of honestly hyperelliptic curves, which is attached to two rational $1$-pointed smooth curves;
    \end{itemize}
    moreover, if the length is odd, we have that
    \begin{itemize}
        \item[o)] $(C_1,p_1^1,p_2^1)$ is a $1$-pointed chain of honestly hyperelliptic curves attached to a $1$-pointed rational smooth curve at an odd $A$-singularity.
    \end{itemize} 
\end{proposition}

\begin{figure}[H]
        \caption{}
        \centering
        \includegraphics[width=0.6\textwidth]{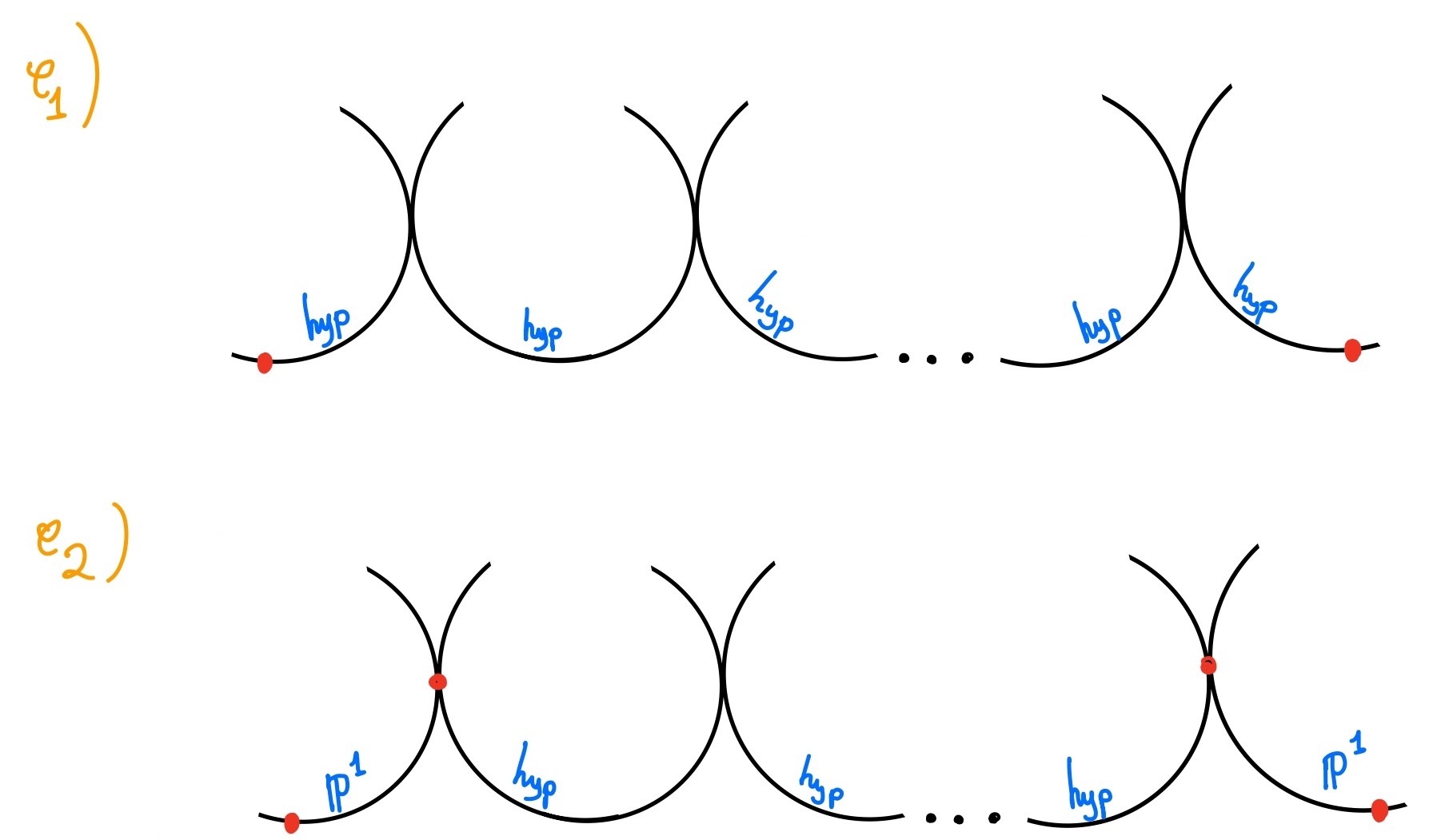}
    \end{figure} 
\begin{figure}[H]
        \caption{}
        \centering
        \includegraphics[width=0.6\textwidth]{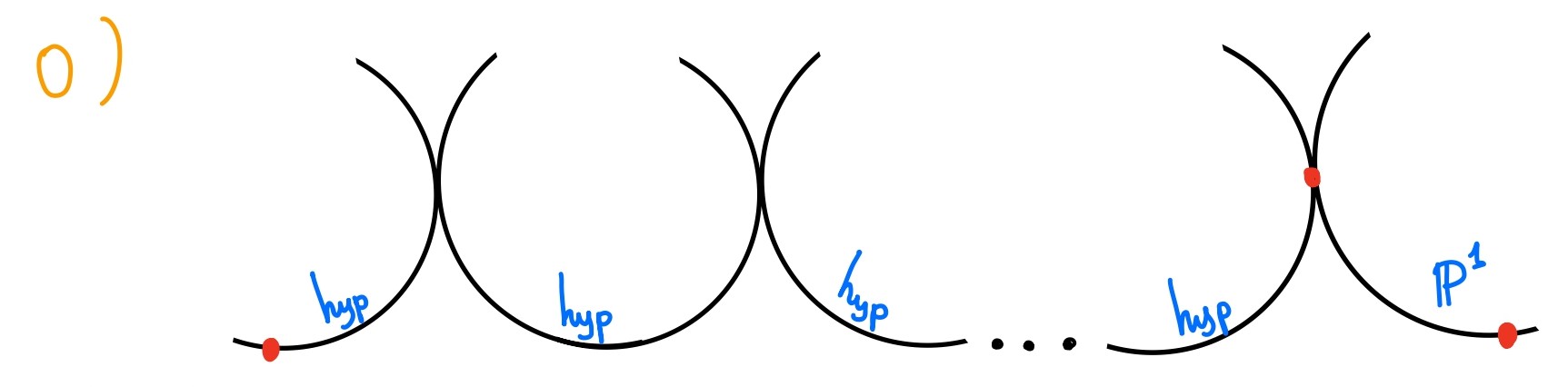}
    \end{figure}
    
Finally, we note that the same statement works for closed rosaries. 
\begin{definition}
     We say that $C$ is a \emph{closed chain of hyperelliptic curves} of genus bounded by $g$ if there exists an $A_{2k+1}$-singularity with $k\geq 1$ such that the (pointed) normalization $(\widetilde{C},p_1,p_2)$ is a $2$-pointed chain of hyperelliptic curves of genus bounded by $g$.
\end{definition}

We then have a similar result.
\begin{proposition}\label{prop:iso-deg-rosary}
    Let $(C_0,p_1^0,p_2^0)$ be a closed rosary of hyperelliptic $A$-singularities (of even length) and $(C_1,p_1^1,p_2^1)$ be an $A_r$-stable curve which admits an isotrivial degeneration to $C_0$. Then $(C_1,p_1^1,p_2^1)$ is a closed chain of hyperelliptic curves. 
\end{proposition}

\subsection{Global geometry}\label{sub:global geometry}

We end the section by globalizing the study of isotrivial degenerations. To be more specific, in the previous section we described all the possible isotrivial degenerations to a rosary of hyperelliptic $A_r$-singularities, giving a local description of the basins of attraction. In this section, we are going to make the description global, namely, producing isotrivial degenerations from a fixed generic point.

Recall that if $\Gamma\subset C$ is an $A_1$-attached even atom of genus $h$, any isotrivial degeneration either preserves the $A_{2h}$-singularity, or the generic point is an $A_1$-attached honestly hyperelliptic tail of genus $h$.
\begin{question}
    Given a curve lying in one of the two \emph{local} strata, can we find an isotrivial degeneration to the curve with the $A_1$-attached even atom?
\end{question} 
The answer is yes, see \cite[Lemma 4.42]{GMSArI} and \cite[Lemma 4.43]{GMSArI}. This gives us a full description of the possible isotrivial degenerations to an $A_1$-attached even atom: for instance, every curve with an $A_1$-attached hyperelliptic tail of genus $h$ admits an isotrivial degeneration to the same curve where we replaced the $A_1$-attached hyperelliptic tail of genus $h$ with the $A_1$-attached even atom of genus $h$. The same is true if we have a curve $C$ with an even $A$-singularity.

Similarly to the even case, recall that if $\Gamma\subset C$ is an $A_1/A_1$-attached odd atom of genus $h$, any isotrivial degeneration to \(C\) either preserves the $A_{2h+1}$-singularity, or the generic point has an $A_1/A_1$-attached honestly hyperelliptic bridge of genus $h$. The globalization holds: every curve with an $A_1/A_1$-attached hyperelliptic bridge of genus $h$ admits an isotrivial degeneration to the same curve where we replaced the bridge with an $A_1/A_1$-attached odd atom of genus $h$. The same is true if we start with a curve $C$ with an odd $A$-singularity. For a more detailed statement, see \cite[Lemma 4.44]{GMSArI} and \cite[Lemma 4.45]{GMSArI}.

Finally, we describe the basins of attraction of $2$-pointed closed rosaries of hyperelliptic $A$-singularities. This is the globalization of \Cref{prop:defor-rational-chain}. 

\begin{proposition}\label{prop:iso-deg-chain}
    Let $(C_1,p_1^1,p_2^1)$ be a $2$-pointed $A_r$-stable curve. Then $(C_1,p_1^1,p_2^1)$ admits an isotrivial degeneration to a $2$-pointed rosary of hyperelliptic $A$-singularities if it is one of the curves described below, depending on the parity of the length of the rosary. Suppose that the length of the rosary is even; then $(C_1,p_1^1,p_2^1)$ is either
    \begin{itemize}
        \item[e1)] a $2$-pointed chain of (honestly) hyperelliptic curves;
        \item[e2)] an $A_{2h+1}/A_{2k+1}$-attached chain of (honestly) hyperelliptic curves, which is attached to two rational $1$-pointed smooth curves;
    \end{itemize}
    moreover, if the length is odd, we have that
    \begin{itemize}
        \item[o)] $(C_1,p_1^1,p_2^1)$ is a $1$-pointed chain of (honestly) hyperelliptic curves attached by an odd $A$-singularity to a $1$-pointed rational smooth curve.
    \end{itemize} 
\end{proposition}

\begin{remark}
    Notice that a rosary of hyperelliptic $A$-singularities of even length is a chain of hyperelliptic curves. 
\end{remark}
\section{The open substack of admissible curves}\label{sec:counterexamples}

This section is dedicated to studying how $\Theta$-completeness and $\textsf{S}$-completeness fail for the moduli stack $\cM_{g,n}^r$. Moreover, by analyzing the possible counterexamples to $\Theta$-completeness and $\textsf{S}$-completeness, we define an open substack $\cU_{g,n}^r$ of $\cM_{g,n}^r$ which we conjecture has a separated good moduli space. Furthermore, we prove the following maximality result, showing that the list of counterexamples we found produces a maximal open substack.

For the rest of the paper, we will assume $r\leq 2g+1$, since \Cref{rem: max-sing} shows that nothing new happens if $r$ is greater than $2g+1$.

\begin{theorem}\label{theo:maximal-open}
    Let $g,r,n$ be three nonnegative integers, and suppose that $g\geq 2$. Let $\cU \subset \cM_{g,n}^r$ be an open substack which admits a separated good moduli space. If $r\leq 2g-5$ and $\cU_{g,n}^r\subset \cU$, then $\cU=\cU_{g,n}^r$.
\end{theorem}

\begin{proof}
    The result is a consequence of \Cref{lem:self-attached-condition}, \Cref{cor:maximality-chains}, \Cref{ex:proble-tail-last} and \Cref{ex:n=1-hyp}.
\end{proof}
If $\cU_{g,n}^r$ has a separated good moduli space, then \Cref{theo:maximal-open} shows that $\cU_{g,n}^r$ is maximal among open substacks of $\cM_{g,n}^r$ that have a separated good moduli space. In the final paper of the series, we prove that $\cU_{g,n}^r$ is $\Theta$-complete for all $r$ and that it is $\textsf{S}$-complete for $r \leq 5$. Moreover, we prove that such a good moduli space (when it exists) does not have ample line bundles if $r\geq 5$, thus giving an example (specifically for $r=5$) where GIT cannot be applied.

\begin{remark}\label{rem:not-S-complete}
 Before going forward, we invite the reader to familiarize themselves with the strategy used to find $\cU_{g,n}^r$ outlined in \refmodus{rem:modus-operandi}. Very similar considerations go into the proof of \Cref{theo:maximal-open} in the coming subsections. As the aforementioned strategy seems to force us to remove a closed subset, always starting with the failure of $\Theta$-completeness, one may wonder whether the stack $\cM_{g,n}^r$ is $\textsf{S}$-complete. We claim that this is not the case. Let $(\Gamma,p)$ be a $1$-pointed even atom of genus $h$ with singularity $q$ and $(\Gamma',p,q')$ be a $2$-pointed smooth hyperelliptic curve of genus $k\geq 1$ such that $p$ and $q'$ are both Weierstrass points. We denote by $(C_0,q')$ the $1$-pointed $A_r$-stable curve obtained by nodally gluing $(\Gamma,p)$ with $(\Gamma',p,q')$ at $p$.
 \begin{figure}[H]
		\caption{}
		\centering
		\includegraphics[width=0.5\textwidth]{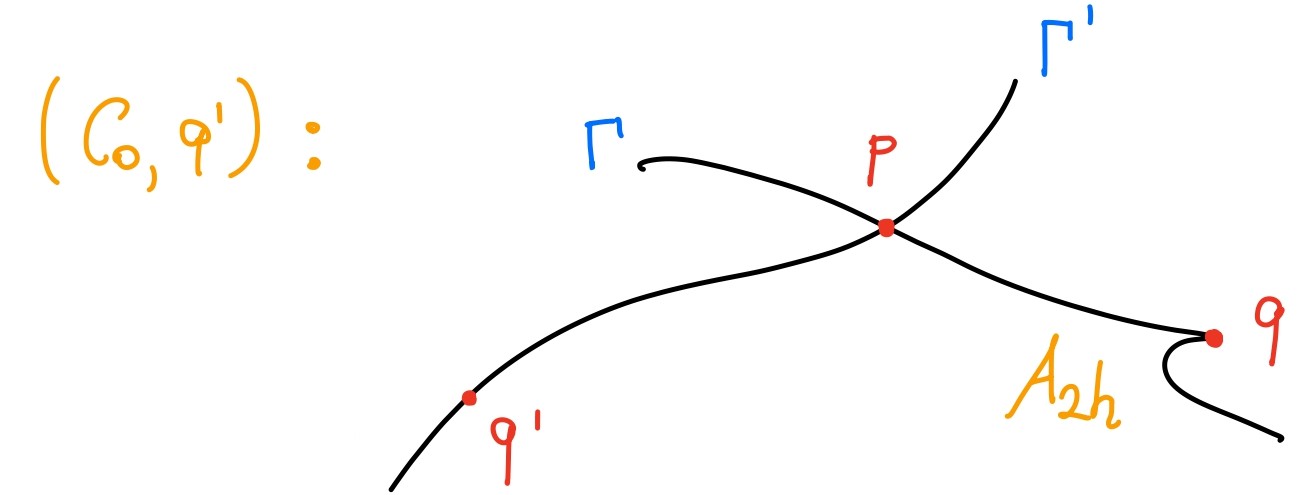}
		\label{fig:Chapter 4_1}
	\end{figure} 
 We have $\Aut(C_0,q')^{\circ}\simeq \Gm$. Moreover, we have a decomposition 
 $$\Def_{(C,p)} \simeq T_q^1 \oplus {\rm Cr}_q \oplus T_p^1 \oplus \Def_{(\Gamma',p,q')}$$ where the $\Gm$ acts positively on $T_q^1$, negatively on ${\rm Cr}_q\oplus T_p^1$ and trivially on $\Def_{(\Gamma',p,q')}$. Consider the isotrivial degeneration given by the line $(0\oplus0\oplus T_p^1\oplus 0)\subset \Def_{(C,p)}$, which corresponds to the degeneration that deforms the node $p$, but fixes the singularity $q$. The generic point of the degeneration corresponds to an honestly hyperelliptic curve $(C_1,q_1')$ of genus $h+k$ with an $A_{2h}$-singularity and a Weierstrass marking $q_1'$, therefore an object of $\cH_{h+k,w}^{r,\circ}$ with $r\geq 2h$. \Cref{prop:w-global-descr} implies that there exists an isotrivial degeneration with generic point $(C_1,q_1')$ and with special point $(C_0',q_0')$, the $1$-pointed even atom of genus $h+k$, which is a closed point.
\begin{figure}[H]
		\caption{}
		\centering
		\includegraphics[width=0.7\textwidth]{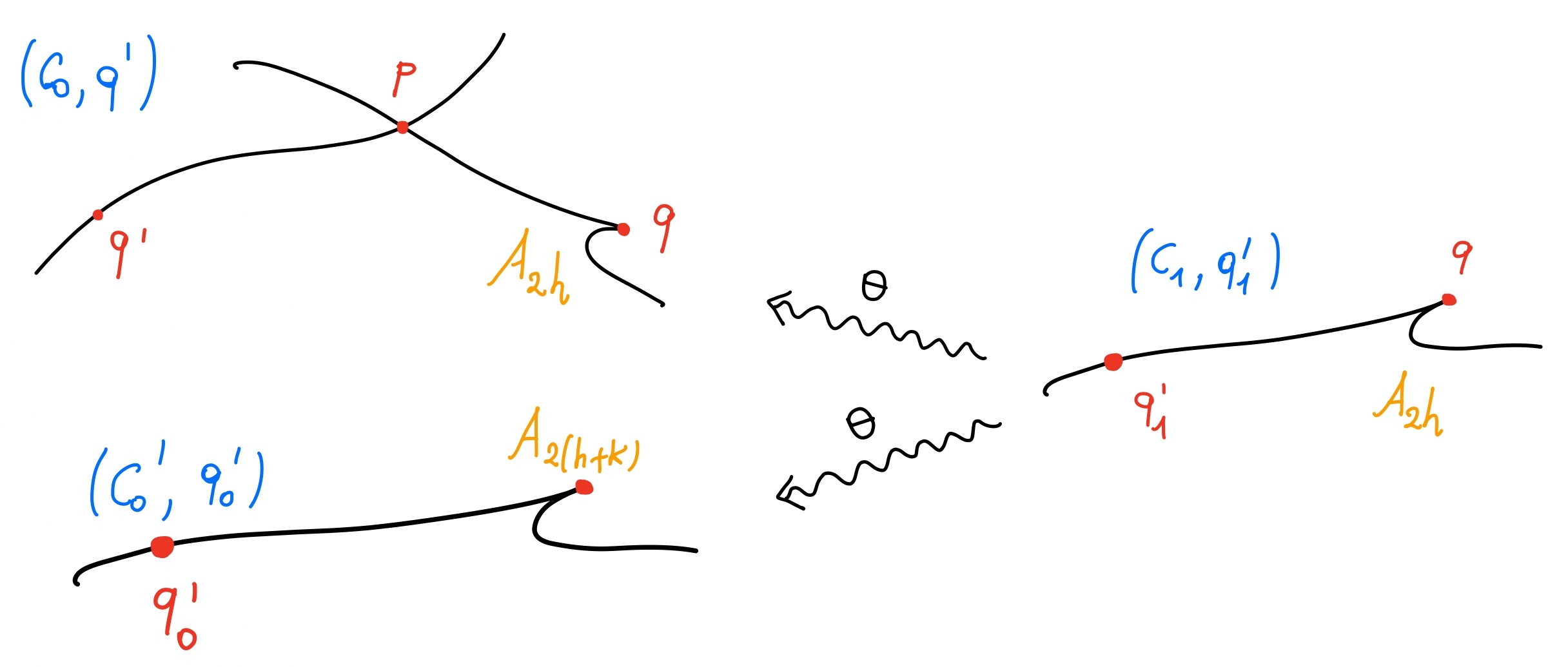}
		\label{fig:Chapter 4_2}
	\end{figure}

 If we apply \Cref{prop:alp-fund} to the separating node of $(C_0,q')$, one can prove that there are no degenerations from $(C_0,q')$ to $(C_0',q_0')$. This contradicts $\Theta$-completeness and $\textsf{S}$-completeness: specifically, it contradicts the filling criterion for morphisms from the open embedding $\Theta^2\setminus 0 \subset \Theta^2$, which is a weaker condition than both $\textsf{S}$-completeness and $\Theta$-completeness. This in particular shows that $\cM_{g,n}^r$ is neither $\textsf{S}$- nor $\Theta$-complete, as previously announced in \cite[Section~6.3]{AlpHalLei}.
\end{remark}

\begin{remark}
     In this remark, we will use definitions and notations as they will appear in a future work by David Rydh, where the author introduces the notion of a \emph{topological moduli space}. The $\Theta^2$-filling condition can be regarded as a valuative criterion for $N$-separatedness (for locally reductive stacks). We conjecture that this is the only condition in the existence result for topological moduli spaces that we are missing, i.e., $\cM_{g,n}^r$ is $M$-separated and $N$-closed, but not $N$-separated. In the third paper of the series, we discuss a valuative criterion for $M$-separatedness, which we call \emph{weak} \textsf{S}-completeness.
\end{remark}

Before discussing the counterexamples, we introduce the conditions on $A_r$-stable curves in \Cref{def:admin}, depending on the integer $r$. Then we explain why the conditions we provide are needed for the existence of the separated good moduli space, referring to the strategy above. 
 
\begin{definition}[The stack $\cU_{g,n}^r$ of admissible curves]
\label{def:admin}\index{Admissible curve}
Given non-negative integers $r, g, n$, let $C$ be an $n$-pointed $A_r$-stable curve\index{$A_r$-stable curve} of genus $g$ over an algebraically closed field $k$. We list the following conditions on $C$:    
    \begin{itemize}
        \item[($t_r$)] $C$ does not contain $A_0$-attached or $A_{2k+1}$-attached hyperelliptic tails of genus $1\leq h<r/2$ for any $k\geq 0$;
        \item[($c_r$)] $C$ does not contain non-separating $A_1/A_1$-attached  hyperelliptic chains of genus bounded by $(r-1)/2$;
        \item[($d_{r,1}$)] if \(r\) is odd, then $C$ does not contain non-separating $A_1/A_r$-attached hyperelliptic chains of genus bounded by  $(r-1)/2$;
        \item[($d_{r,2}$)] if \(r\) is odd, then $C$ does not contain non-separating $A_r/A_r$-attached hyperelliptic chains of genus bounded by  $(r-1)/2$;
        \item[$(e_r)$] $C$ is not an $A_h$-self-attached hyperelliptic curve with $h>g$ (see \Cref{def:self-attached}).
    \end{itemize}
We will often refer to conditions $(d_{r,1})+(d_{r,2})$ as condition $(d_r)$.

If an \(A_r\)-stable curve satisfies conditions $(t_r)$, $(c_r)$, $(d_r)$ and $(e_r)$, we call it an \emph{admissible} curve.
We denote by $\cU_{g,n}^r$ the moduli stack classifying $n$-pointed admissible curves of genus $g$, unless \((g,n)=(1,1)\), in which case we set $\cU_{1,1}^r\coloneq \cM_{1,1}^r(=\cM_{1,1}^2)$ for all $r\geq 2$ (see \Cref{rem:def-self-attached}).

\end{definition}

\begin{definition}[$A_h$-self-attached hyperelliptic curves]\label{def:self-attached}
    Let $g, h$ be positive integers with $h>g$. We say that an $A_r$-stable curve $C$ of genus $g$ is an \emph{$A_h$-self-attached hyperelliptic curve} if one of the following is true:
\begin{itemize}
    \item[(1)] $h=2h_0$ and $C$ is obtained from a hyperelliptic $A_r$-prestable $1$-pointed curve $(C_0,p)$,  where $p$ is a Weierstrass point, by pinching $p$ and thus creating an $A_{2h_0}$-singularity; we also include the case where $(C_0,p)$ has genus $0$;
    \item[(2)] $h=2h_0+1$ and $C$ is obtained from a hyperelliptic $A_r$-prestable $2$-pointed curve $(C_0,p_1,p_2)$, where $p_1$ and $p_2$ are exchanged by the hyperelliptic involution, by gluing $p_1$ and $p_2$ and thus creating an $A_{2h_0+1}$-singularity; we also include the case where $(C_0,p_1,p_2)$ has genus $0$;
    \item[(3)] $h=2g+1$ and $C$ is obtained by gluing two $\bP^1$'s along $\infty$ to create an $A_{2g+1}$-singularity.
\end{itemize}
\end{definition}

\begin{remark}
    Notice that we do not require the curve $C$ to be hyperelliptic, and case (3) above only appears for $r=2g+1$.
\end{remark}

\begin{remark}\label{rem:def-self-attached}
    We note that condition $(e_r)$ only appears for $n=0$ and $r>g$. Moreover, notice that if $C$ has an $A_0$-attached hyperelliptic tail $\Gamma$, it means that $\Gamma=C$, $n=1$, and thus if $C$ does not verify $(t_r)$, then $g<r/2$, which implies $r\geq 2g+1$. Furthermore, we should mention that if $(g,n)=(1,1)$, condition $(t_r)$ would imply $\cU_{1,1}^3= \emptyset$. This is a pathology which only occurs in this case, as \Cref{ex:n=1-hyp} produces a counterexample as soon as $g\geq 2$.
\end{remark}

\begin{remark}\label{rem:HKP-comparison}
    As the reader may notice, \Cref{def:admin} bears a strong resemblance to the definition of $\alpha_c$-stable curves appearing in the Hassett-Keel program (see \cite{AlpFedSmyWyck}). Since the last explicit step of the program uses $A_h$-singularities with $h\leq 4$, we can compare the two definitions for $r\leq 4$. 
    
    If $r=0$ (respectively $r=1$), the two definitions coincide and they are just the usual definitions of smooth (respectively stable) curves. If $r=2$, then our open corresponds to $\alpha=9/11$-stable curves \cite[Definition 2.5]{AlpFedSmyWyck}. If $r=3$, our open corresponds to \(\alpha=7/10\)-stable curves. Indeed, conditions $(c_r)$ and $(d_r)$ are empty for $r\leq 3$. Condition $(t_r)$ is empty for $r\leq 2$, and condition $(t_3)$ states that there are no $A_0$-, $A_1$-, and $A_3$-attached elliptic tails. Notice that if an $A_r$-stable curve has an $A_0$-attached elliptic tail, then the curve itself is a $1$-pointed genus $1$ curve, which is in $\cM_{1,1}^r$ (see the discussion in \Cref{rem:def-self-attached}). Moreover, condition $(e_r)$ is empty if $r\leq 2$, and $(e_3)$ states that $C$ is not an $A_3$-self-attached hyperelliptic curve of genus $g<3$. Nevertheless, the definition of $\alpha$-stable curves is stated in genus $\geq 3$ (see \cite[Remark 1.1]{AlpFedSmyWyck}).

    If $r=4$, then every $2/3$-stable curve is actually admissible, but the converse is not true, as we now show. First, let us recall the definition in the $2/3$-stable case: they do not allow
    \begin{itemize}
        \item[(1)] $A_1$-, $A_3$-, and $A_4$-attached elliptic tails;
        \item[(2)] $A_1/A_1$-, $A_1/A_4$-, and $A_4/A_4$-attached elliptic chains.
    \end{itemize}
    Condition $(t_4)$ forces us to remove $A_1$- and $A_3$-attached elliptic tails. Condition $(d_4)$ is trivial, whereas condition $(c_r)$ forces us to remove $A_1/A_1$-attached \emph{non-separating} elliptic chains. Moreover, condition $(e_r)$ forces us to remove $A_4$-attached elliptic tails (and nothing else in genus $\geq 3$).

\end{remark} 
\begin{lemma}\label{lem:cyclic-covers}
    If $r\geq  2g-1+n$, then the intersection $\cH_{g,n}^r \cap \cU_{g,n}^r$ consists of honestly hyperelliptic curves. 
\end{lemma}

\begin{proof}
    First, we know that every hyperelliptic curve in $\cH_{g,n}^r$ is a curve $C$ with an involution $\sigma$ and $n$ markings. Recall that the quotient $Z:=C/\sigma$ is a genus $0$ nodal curve, thus a tree of $\bP^1$'s. We want to prove that the quotient $Z$ has only one irreducible component. If $Z$ is not irreducible, we can consider a leaf of our tree of rational curves, namely a component $\Gamma$ which intersects the rest of the curve $Z-\Gamma$ in one unique point $q$. Let \(\pi:C\to Z\) be the quotient morphism. Notice that the fiber $\pi^{-1}(q)$ is either a node (which corresponds to a Weierstrass point in $\pi^{-1}(\Gamma)$), a tacnode (which corresponds to a Weierstrass point in $\pi^{-1}(\Gamma)$), or two nodes (which correspond to a pair of points in $\pi^{-1}(\Gamma)$ exchanged by the involution). See Proposition 3.2 of \cite{Per2}. Thus, $\pi^{-1}(\Gamma)$ is either a non-separating $A_1/A_1$-attached bridge of genus $h\leq g-1$ or an $A_1$- or $A_3$-attached tail of genus $h\leq g$. Because $r\geq 2g-1$, conditions $(t_r)$ and $(c_r)$ imply that $h=g$, $r\leq 2g$ (thus $n\leq 1$), and $\Gamma$ is attached to $C-\Gamma$ at a simple node. Finally, the stability condition yields the desired contradiction.
\end{proof}

\begin{remark}\label{rem:cyclic-covers}
    Notice that in the previous proof, we have only used condition $(t_r)$ for $A_1$- and $A_3$-attached hyperelliptic tails and condition $(c_r)$.
\end{remark}

Finally, we end up with a straightforward but useful lemma, which will be crucial in the inductive proof of \textsf{S}-completeness.

\begin{lemma}\label{lem:r-implies-r-1}
    Let $C$ be an $n$-pointed $A_r$-stable curve of genus $g$ over an algebraically closed field $k$. Then 
    \begin{itemize}
        \item if $C$ satisfies $(t_r)$, then it satisfies $(t_{r-1})$;
        \item if $C$ satisfies $(c_r)$, then it satisfies $(c_{r-1})$;
        \item if $C$ satisfies $(e_r)$, then it satisfies $(e_{r-1})$.
    \end{itemize}
    Moreover, $(d_r)$ is an empty condition if $r$ is even.
\end{lemma}
\begin{proof}
    It follows straightforwardly from the definitions.
\end{proof}

\subsection{Openness of the stack of admissible curves}\label{sub:openness}
In this subsection, we prove that the locus of admissible curves is an open substack of $\cM_{g,n}^r$. The proof is similar to \cite[Proposition~2.15]{AlpFedSmyWyck}: indeed, the main contribution is the openness of condition $(e_r)$.

\begin{lemma}\label{lemma:open_t_c_d}
    The locus defined by conditions $(t_r)$, $(c_r)$, and $(d_r)$ is open in $\cM_{g,n}^r$.
\end{lemma}
\begin{proof}
    The locus can be realized as the finite union of substacks, described in \cite[Theorem~1.105]{VanDerWyck}, determined by the combinatorial types of the curves, therefore the constructibility of the locus follows by applying Chevalley's theorem (cf. \cite[Théorème~5.9.4]{LauMorBai}). Hence, it suffices to prove that its complement is closed under specialization.
    
    We start with condition $(t_r)$. Suppose we are given a family $C\rightarrow \spec R$ of $n$-pointed $A_r$-stable curves over a dvr $R$ such that the generic fiber $C_Q\rightarrow \spec Q$ contains an $A_{2k+1}$-attached hyperelliptic tail $\Gamma_Q$ of genus $1\leq h<r/2$. The $A_{2k+1}$-singularity $p_Q$ (which exists up to passing to a finite extension of dvrs) is separating by definition. Thus, it gives rise to an equisingular separating section $p_R$, because \Cref{prop:alp-fund} ensures that it remains a separating $A_{2k+1}$-singularity.

    Let $\widetilde{C}\rightarrow \spec R$ be the partial normalization along $p_R$, and denote by $\Gamma_R\subset \widetilde{C}$ the connected component which restricts to $\Gamma_Q$ on the generic fiber. If $\Gamma_R$ is $A_r$-stable, then $\Gamma_R$ is also a hyperelliptic curve because $\cH_{h,w}^r$ is closed inside $\cM_{h,1}^r$, and $p_R$ is a Weierstrass point. If $\Gamma_R$ is not stable, it suffices to note that the moduli stack of partial normalizations is isomorphic to $\cM_{g-k,n+1}^r \times [\bA^1/\Gm]$, and the same closure property holds. Notice that the same argument implies that being an $A_0$-attached hyperelliptic tail is a closed condition. Thus, \((t_r)\) is an open condition. 

    Suppose now that we are given a family $C\rightarrow \spec R$ of $n$-pointed $A_r$-stable curves over a dvr $R$ such that $C_Q$ does not satisfy condition $(c_r)$. There are two possibilities. 

    First, if $C_Q$ contains a non-separating $A_1/A_1$-attached hyperelliptic chain of length two or more (with genus bounded by $(r-1)/2$), then any attaching singularity of the hyperelliptic bridges is lonely. We can then apply \Cref{prop:alp-fund} analogously to the $(t_r)$ case. In this situation, it is enough to prove that the property of being a $2$-pointed hyperelliptic curve (with the two markings exchanged by the involution) is closed. 

    It remains to address the case of length 1, i.e., the case of a bridge. Here, the two attaching singularities can either be preserved or can degenerate through the family $C\rightarrow \spec R$ to the same singularity $p$, which will be an $A_3$-singularity by \Cref{prop:alp-fund}. In the first case, we conclude as before. In the second case, we conclude because $p$ is now an attaching Weierstrass point of a hyperelliptic tail of genus $h<(r-1)/2<r/2$ (see, for instance, \cite[Theorem~3.42]{GMSArI}), which violates condition $(t_r)$. Condition $(d_r)$ follows by the same argument.
\end{proof}

\begin{theorem}\label{lem:open}
    The stack $\cU_{g,n}^r$ is an open substack of $\cM_{g,n}^r$.
\end{theorem}
\begin{proof}
    For $n>0$, condition $(e_r)$ is vacuous, hence the claim follows directly from \Cref{lemma:open_t_c_d}. Assume now that $n=0$. For the same reason as in the previous lemma, the locus is constructible by Chevalley's theorem, so it remains to prove that the complement is closed under specialization. Thanks to \Cref{lemma:open_t_c_d}, we can reduce the problem to proving that if a family of curves $C \to \spec R$ (where $R$ is a dvr) satisfies $(t_r)$, $(c_r)$, and $(d_r)$, and its generic fiber does not satisfy $(e_r)$, then the special fiber also does not satisfy $(e_r)$.

    Suppose that the generic fiber $C_Q$ is an $A_h$-self-attached hyperelliptic curve where $h>g$. The $A_h$-singular section $p_Q$ extends to a section of $C$ that lands in an $A_k$-singularity in the special fiber, with $k\geq h$.

    If $k=h$, then $p_R$ is an equisingular section. We can normalize along it, and the result follows in the same way as \Cref{lemma:open_t_c_d}. Thus, we can assume $k>h$. For simplicity, we deal with the odd case $h=2h_0+1$, leaving it to the reader to verify that the same proof works in the even case.

    First, consider the iterated blowup procedure described in \cite[Definition~4.3]{Per1}, which yields a curve $\widetilde{C}\rightarrow C$ over $\spec R$ where we have resolved the generic fiber of $C$ along the $A_h$-singularity. Suppose that the genus of $\widetilde{C}$ is greater than or equal to $1$. We claim that the (non-pointed) curve $\widetilde{C}$ is an $A_r$-stable hyperelliptic curve. First of all, the generic fiber $\widetilde{C}_Q$ is a hyperelliptic $A_r$-prestable curve. Since it is the partial normalization of an $A_r$-stable curve $C_Q$, the unstable components in the odd case must be two $A_1$-attached $\bP^1$'s exchanged by the involution (by the definition of self-attached). Therefore, the rest of the curve would be an $A_1/A_1$-attached hyperelliptic curve of genus $g-h_0-1 < h_0 = (h-1)/2 \leq (r-1)/2$ that also appears in $C_Q$ (since the morphism $\widetilde{C}\rightarrow C$ is an isomorphism outside the section $p$). This is impossible since $C_Q$ satisfies $(c_r)$.

    Therefore, the curve $\widetilde{C}_Q$ is a hyperelliptic $A_r$-stable curve with two smooth markings $q_1$ and $q_2$ exchanged by the involution. A priori, the special fiber $\widetilde{C}_k$ over the closed point may only be $A_r$-prestable rather than $A_r$-stable. We aim to prove that it is, in fact, $A_r$-stable.

    Consider the stabilization $\widetilde{C} \rightarrow C'$. Then $C'$ is an $A_r$-stable curve over $R$ whose generic fiber is hyperelliptic. Since the stack of hyperelliptic \(A_r\)-stable curves is a closed substack of \(\cM_{g}^r\) (see \cite[Theorem~3.42]{GMSArI}), we obtain that $C'_k$ is a hyperelliptic $A_r$-stable curve. If $\widetilde{C}_k$ had unstable components, this would again imply that it contains an $A_1/A_1$-attached hyperelliptic curve of genus $g-h_0-1 < (r-1)/2$, which would then also appear in $C_k$. This is impossible since we assume that $C_k$ satisfies $(c_r)$ and $(t_r)$.

    Thus, $\widetilde{C}$ is $A_r$-stable. Using the closedness of hyperelliptic curves, we deduce that $\widetilde{C}$ is hyperelliptic and $A_r$-stable of genus $g-h_0$. Since $g-h_0-1 < (r-1)/2$, we can apply \Cref{lem:cyclic-covers} to conclude that $\widetilde{C}$ is honestly hyperelliptic. Therefore (up to taking finite extensions), $\widetilde{C}_k$ has a $2\!:\!1$ morphism to $\bP^1$ and is connected, thus it is an honestly hyperelliptic curve. 

    By definition, the curve $C_k$ is an $A_k$-self-attached curve of genus $g < h < k$, which completes the proof for this case. We leave the genus $0$ case to the interested reader; it is simpler to handle since the section $p_R$ is equisingular.
\end{proof}

\subsection{Self-attached hyperelliptic curves}\label{sub:e_r}
Let us start with the condition due to the honestly hyperelliptic locus in $\cM_{g,n}^r$ not having a good moduli space: as discussed at the end of \Cref{sub:pointwise geometry}, when $n=0$ and $r\geq 2g+1$, there is a closed point in $\cH_g^r$ which has stabilizers of the form $\Gm \ltimes \Ga$. The study of the basins of attraction of this point leads to the definition of \emph{self-attached curves}, see \Cref{def:self-attached}. Note that these curves are not marked, therefore that condition is trivially satisfied if \(n\neq 0\). 

First, we construct isotrivial degenerations for self-attached hyperelliptic curves.
\begin{lemma}\label{lem:deg-self-attached}
   Let $g, h$ be positive integers and let $C_1$ be an $A_h$-self-attached honestly hyperelliptic curve of genus $g$, with singularity $q_1$. Then there exists an isotrivial degeneration $C\rightarrow \Theta$ such that
   \begin{itemize}
       \item the generic fiber is isomorphic to $C_1$,
       \item the section $q_1$ can be extended to an equisingular section $q$,
       \item the $A_h$-singularity $q_0$ is hyperelliptic.
    \end{itemize}
    Moreover, $C_0$ is a $2:1$-cover of $\bP^1$ with $q_0$ a branching point of multiplicity $h+1$.
\end{lemma}

\begin{proof}
    The proof of this result is similar to the one of \cite[Lemma~4.46]{GMSArI}.
\end{proof}

\begin{lemma}\label{lem:self-attached-condition}
    Let $\cU \subset \cM_{g}^r$ be an open substack which admits a separated good moduli space. If $\cU_{g}^r\subset \cU$, then every curve in $\cU$ satisfies condition $(e_r)$ of \Cref{def:admin}.
\end{lemma}

\begin{proof}
     Notice that because $\cH_g^r$ is closed inside $\cM_g^r$, every obstruction for $\Theta$-completeness (or $\textsf{S}$-completeness) of $\cH_g^r$ is preserved in $\cM_g^r$. The same is true if we restrict to study the open $\cH_g^{r,\circ}\subset \cH_g^r$ because the open embedding is affine, as it is the complement of a Cartier divisor, namely $\Delta^{\rm sep} \cap \cH_g^r \subset \cH_g^r$, where $\Delta^{\rm sep}$ is the closed locus parametrizing curves with a separating node. Therefore, the classification of open substacks admitting separated good moduli spaces in the moduli stack of cyclic covers of $\bP^1$ (see \cite[Theorem 3.44]{GMSArI}) forces us to remove from $\cU$ all the curves $C$ in $\cM_g^r$ which appear as $2:1$-covers of $\bP^1$ with branching divisor with one root of multiplicity $> g+1$. Now suppose $C_1 \in \cU$ is an $A_h$-self-attached hyperelliptic curve with $h>g$. We can then use \Cref{lem:deg-self-attached} to find an isotrivial specialization whose special fiber $C_0$ is a $2:1$-cover of $\bP^1$ with branching divisor with two roots: $q_0$ of multiplicity $h+1> g+1$ and $p_0$ of multiplicity $k+1< g+1$. The two basins of attraction of the $\Gm$-action on the deformation space of $C_0$ are respectively the directions that preserve $q_0$ and the ones that preserve $p_0$ (notice that the generic point of this basin is contained in $\cU_{g}^r$). Finally, we get that $\cU$ intersects both basins non-trivially (since $C_1\in \cU$ and $\cU_{g}^r\subset \cU$), and therefore, by \cite[Lemma 2.10]{GMSArI} applied around $C_0$, we have that $C_0 \in \cU$. We reach a contradiction since we already proved that $C_0 \notin \cU$, thus $C_1 \notin \cU$.
\end{proof}

\subsection{Attached hyperelliptic chains}\label{sub:problematic-chain}

Before explain how conditions $(c_r)$ and $(d_r)$ come up, we make the following observation which will be used throughout the rest of this section.

\begin{remark}\label{rem:irreducibility}
    Let $P$ be a property of $A_r$-stable curves which is locally closed, i.e. it defines a locally closed substack in $\cM_{g,n}^r$. We say that $P$ is irreducible if the corresponding substack is. The following properties are instances of irreducible properties:
    \begin{itemize}
        \item[(a)] having an $A_{2h}$-singularity;
        \item[(b)] having a non-separating  \(A_{2h+1}\)-singularity;
        \item[(c)] having a $A_i$-attached hyperelliptic tail (or even atom) of fixed genus;
        \item[(d)] having a non-separating \(A_i/A_j\)-attached hyperelliptic bridge (or odd atom) of fixed genus.
    \end{itemize}   
    Notice that for $(b)$ and $(d)$ we only write down the non-separating case, since it is the one we are going to need. A similar statement holds in the separating case, provided that we take into account the genera and markings of the different irreducible components obtained by normalizing at the singularities.
    
    The irreducibility of property (a) follows from combining \cite[Theorem~3.8, Propositions~3.12 and 3.13]{GMSArI} whereas the irreducibility of (b) follows from corresponding results in the odd case in \cite[Theorem 3.8, Propositions 3.17 and 3.21]{GMSArI}. The irreducibility of Property (c) and Property (d) follows by using the irreducibility of Property (a) and Property (b) together with the irreducibility of being hyperelliptic (or being an atom).
\end{remark}

\begin{lemma}\label{ex:prob-bridges}
    Let $n,r,g,h$ be four nonnegative integers. Let $\cU \subset \cM_{g,n}^r$ be an open substack which admits a separated good moduli space. If $\cU_{g,n}^r\subset \cU$, then every curve in $\cU$ does not contain a non-separating $A_1/A_1$-attached hyperelliptic bridge of genus $h$, where
    \begin{itemize}
        \item $1\leq h <(r-1)/2$ if $n>0$;
        \item $4\leq 2h+2\leq \min(r,2g-4)$ if $n=0$.
    \end{itemize}
\end{lemma}

\begin{proof}
First, we show that $\cU$ does not contain curves with a non-separating $A_1/A_1$-attached odd atom of genus $1\leq h<(r-1)/2$ (or $4\leq 2h+2\leq \min(r,2g-4)$ if $n=0$). Suppose by contradiction that $\cU$ contains such a curve; then, by openness, the generic point (see \Cref{rem:irreducibility}) of such a locus is also contained in $\cU$, and we can assume this generic point corresponds to a curve $E \to \spec K$, which only has singularities that lie in the attached atom. We show that there exists a family $\Theta_R \setminus 0 = \Spec R \cup_{\Spec Q} \Theta_Q$ that does not extend to a family over $\Theta_R$, where $R$ is a dvr with fraction field $Q$. This is constructed by gluing two $A_r$-stable curves over $\Spec Q$: a family $C \to \spec R$ and an isotrivial family $D \to \Theta_Q$ such that the two generic fibers over $\spec Q$ are isomorphic, as follows (let $k$ be the residue field of $R$):
\begin{itemize}
\item the generic fiber $C_Q$ (and $D_Q$) is the generic point (see \Cref{rem:irreducibility}) of the locus of $A_r$-stable curves with a non-separating $A_{2h+1}$-singularity;
\item the closed fiber $C_k$ is the generic point (see \Cref{rem:irreducibility}) of the locus of $A_r$-stable curves with an $A_{2h+2}$-singularity (which lies in $\cU_{g,n}^r$ thanks to the inequality constraints);
\item the family $D \to \Theta_Q$ is constructed via an isotrivial degeneration of $C_Q$ as in \cite[Proposition 4.47]{GMSArI} (see also \Cref{sub:global geometry}), where the $A_{2h+1}$-singularity of $C_Q$ degenerates to an odd atom of genus $h$ attached to the rest of the curve at two nodes. Up to a field extension $K \subset Q$, we can assume the central fiber to be $E$.
\end{itemize}
By construction, these glue to a family over $\Theta_R \setminus \{0\}$ as in \Cref{fig:ex4.7(1)}, whose topological points land in $\cU$. Indeed, if $n>0$, then $C \to \spec R$ is a family in $\cU_{g,n}^r$ exactly when $r\geq 2h+2$, and $E/K$ lies in $\cU$ by assumption. If $n=0$, we claim that condition $(e_r)$ ensures that $C_k$ is in $\cU_{g,n}^r$ if and only if $g-h-1 \geq 2$. Indeed, the pointed partial normalization $(C_0,p)$ of $C_k$ at the $A_{2h+2}$-singularity is a generic $1$-pointed smooth curve of genus $g-h-1$. Thus, if $g-h-1\geq 3$, the curve $C_0$ is not hyperelliptic (therefore $C_k$ is not a self-attached hyperelliptic curve), whereas if $g=h+3$, $C_0$ is smooth hyperelliptic, but $p$ is not a Weierstrass point.

This provides a morphism $\Theta \to \cU$ which does not extend to a morphism $\Theta_R \to \cM_{g,n}^r$ because of \Cref{prop:alp-fund}. This is a contradiction, since $\cU$ is $\Theta$-complete.
    
\begin{figure}[H]
    \centering
    \begin{minipage}{0.45\textwidth}
        \centering
        \includegraphics[width=0.85\textwidth]{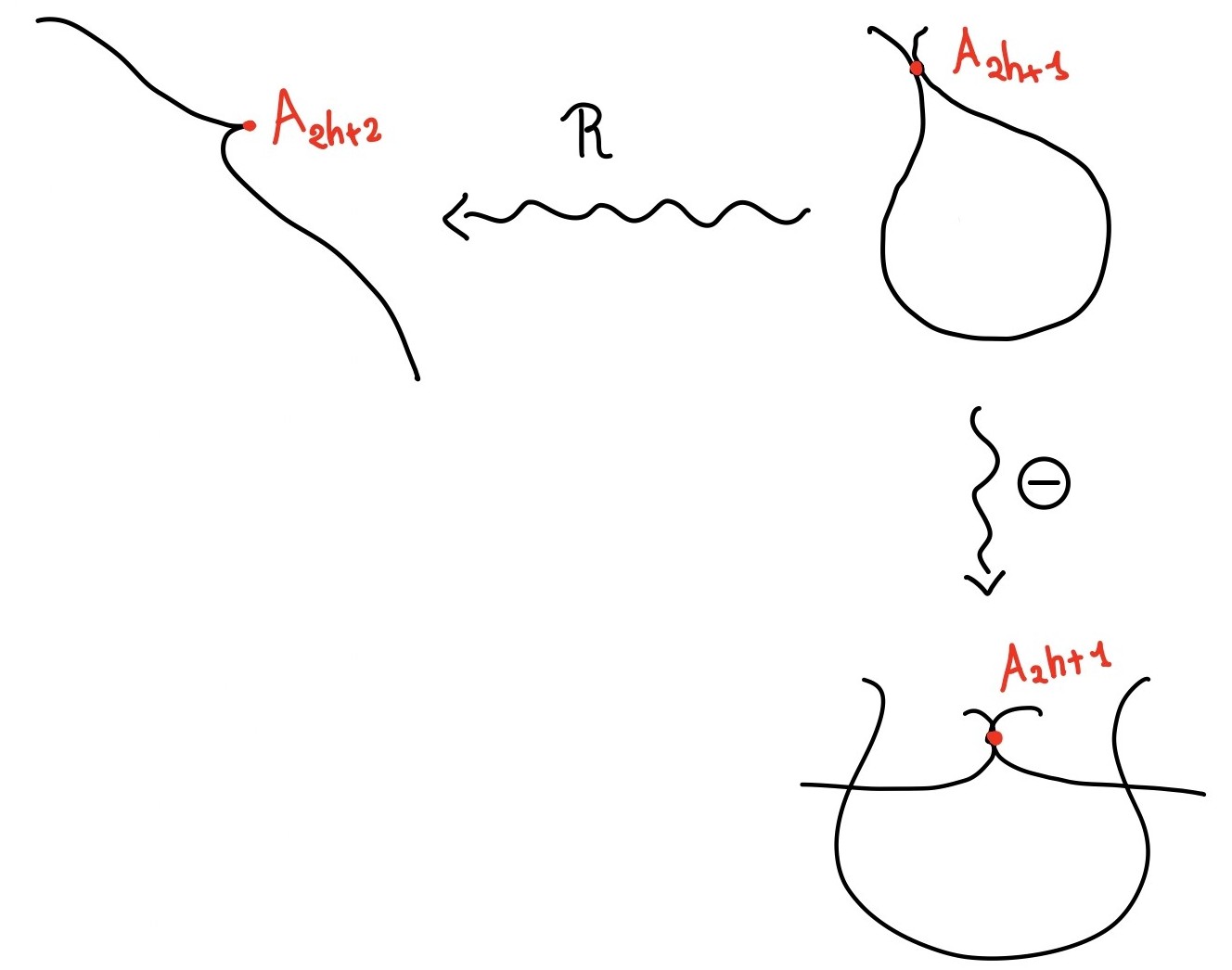}
        \caption{}
        \label{fig:ex4.7(1)}
    \end{minipage}
    \begin{minipage}{0.5\textwidth}
        \centering
        \includegraphics[width=0.96\textwidth]{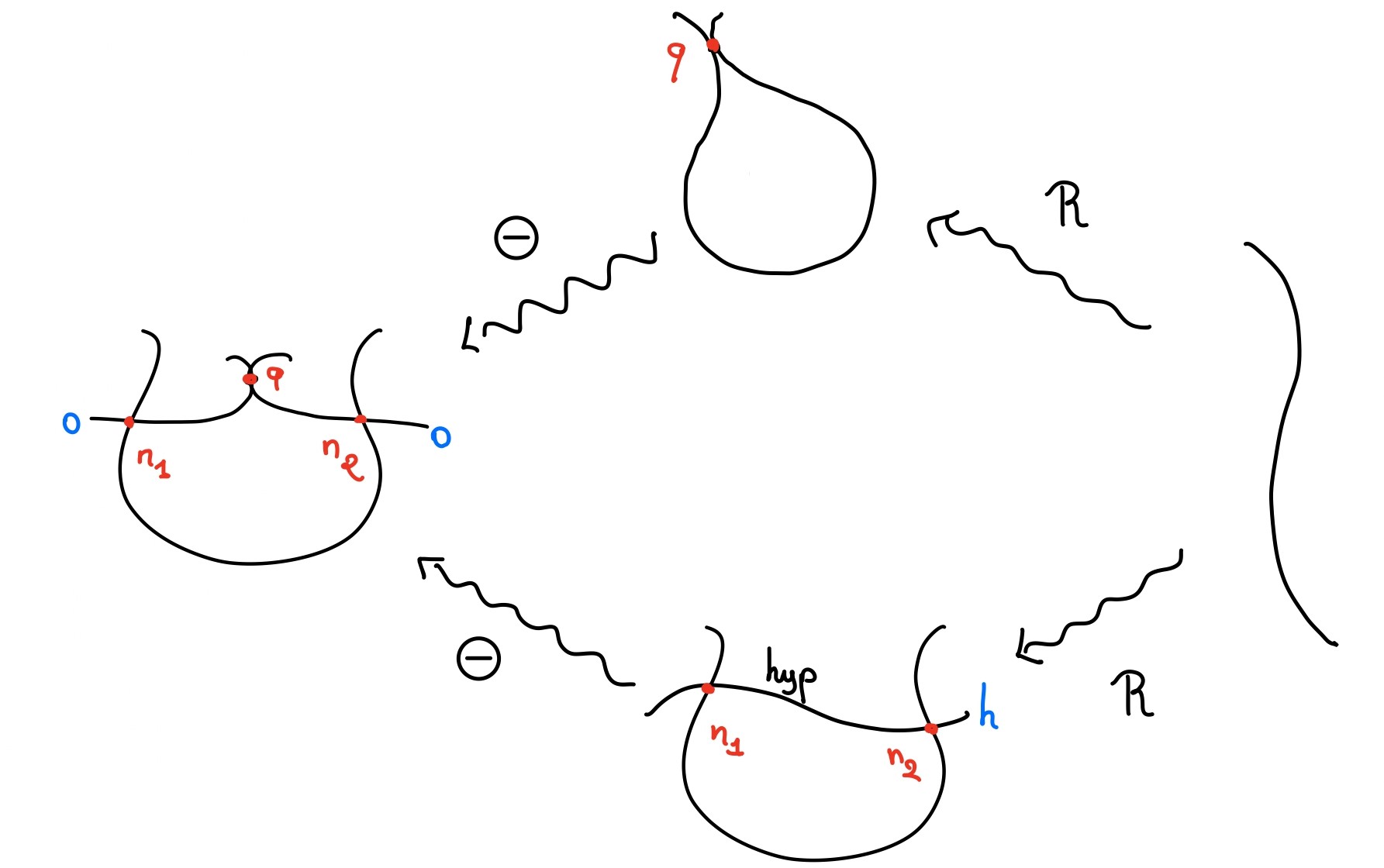}
        \caption{}
        \label{fig:ex4.7(2)}
    \end{minipage}
\end{figure}
Finally, assume by contradiction that $\cU$ contains a point corresponding to a curve $C_{1,0}$ that contains an $A_1/A_1$-attached hyperelliptic $\Gamma_{0,1} \subset C_{0,1}$ of genus $h<(r-1)/2$ (or $4\leq 2h+2\leq \min(r,2g-4)$ if $n=0$). Since $\cU$ is an open substack, we can assume that $\Gamma_{0,1}$ is honestly hyperelliptic and $C_{0,1}-\Gamma_{0,1}$ is smooth. By \Cref{prop:iso-deg-chain}, we can find an isotrivial degeneration of $C_{1,0}$ where the special fiber $C_{0,0}$ is obtained by nodally gluing the $2$-pointed odd atom $\Gamma_{0,1}$ of genus $h$ to $C_{0,0}-\Gamma$. From the discussion above, we know that $C_{0,0}$ is not in $\cU$. The $\Gm$-equivariant deformation space of $C_{0,0}$ in $\cM_{g,n}^r$ is
$$ \Def_{(C_{0,0}-\Gamma,n_1,n_2)}\oplus [T^1_q \oplus ({\rm Cr}_q \oplus T^1_{n_1}\oplus T^1_{n_2})/\Gm]$$
where $q$ is the odd singularity and $n_1,n_2$ are the two nodal attachments. Since $C_{0,0} \notin \cU$ by what we proved above, locally around $C_{0,0}$, $\cU$ is contained in the complement of the $0$-section of the deformation space. Since $\cM_{g,n}^r$ has affine diagonal (thus the local structure map is affine) and $\cU$ is $\textit{S}$-complete, we deduce that one of the two basins of attraction, i.e., either $T_q^1\oplus (0\oplus 0 \oplus 0)$ or $0 \oplus ({\rm Cr}_q \oplus T^1_{n_1} \oplus T^1_{n_2})$, does not intersect $\cU$ (otherwise we would have a diagram of specializations like the one in \Cref{fig:ex4.7(2)}). The isotrivial degeneration of $C_{0,1}$ with special fiber $C_{0,0}$ corresponds to a curve in $T_q^1\oplus (0\oplus 0 \oplus 0)$.

Since the generic point of $0 \oplus ({\rm Cr}_q \oplus T^1_{n_1} \oplus T^1_{n_2})$ is contained in $\cU_{g,n}^r\subset \cU$ via this local picture, we conclude that $C_{0,1} \notin \cU$, which completes the proof.
\end{proof}

\Cref{ex:prob-bridges} explains why we need condition $(c_r)$ in the case of chains of length $1$, also called bridges. We now explain why condition $(d_{r,1})$ must hold for hyperelliptic bridges.

\begin{lemma}\label{ex:cond-dr1}
    Let $n,r,g$ be three nonnegative integers. Let $\cU \subset \cM_{g,n}^r$ be an open substack which admits a separated good moduli space. If $\cU_{g,n}^r\subset \cU$, then every curve in $\cU$ does not contain a non-separating $A_1/A_r$-attached hyperelliptic bridge of genus bounded by $(r-1)/2$.
\end{lemma}
\begin{proof}
    Suppose $r$ is odd. Consider the curve $C$ constructed in the following way, cf. \Cref{fig:ex4.9(1)}:
    \begin{itemize}
        \item let $(\Gamma_0,p,p_0)$ (respectively $(\Gamma_1,q,p_0)$) be an odd atom of genus $h$ with $2h+1<r$ (respectively of genus $(r-1)/2$);
        \item we denote by $(\Gamma,p,q)$ the nodal gluing of $(\Gamma_0,p,p_0)$ with $(\Gamma_1,q,p_0)$ along $p_0$;
        \item we denote by $C$ the curve obtained by nodally gluing $(\Gamma,p,q)$ with a $2$-pointed smooth curve $(C_0,p,q)$ along $p$ and $q$.
    \end{itemize}
    \begin{figure}[H]
        \caption{}
        \centering
        \includegraphics[width=0.4\textwidth]{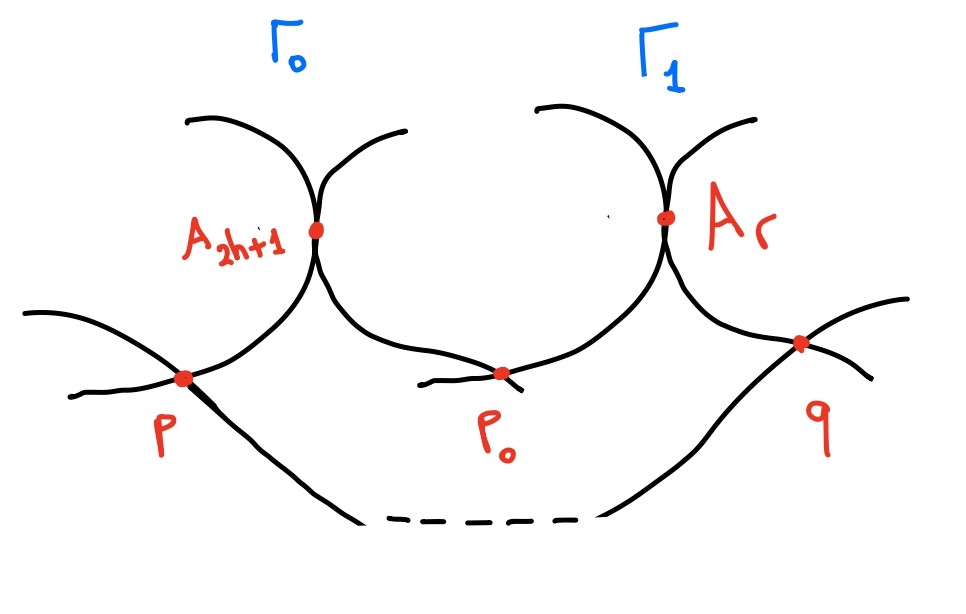}
        \label{fig:ex4.9(1)}
    \end{figure}
     We can describe the deformation space of $C$ as
$$ \Def_C= \Def_{(C_0,p,q)}\oplus T_p^1 \oplus \Def_{(\Gamma_0,p,p_0)}\oplus T_{p_0}^1 \oplus \Def_{(\Gamma_1,p_0,q)} \oplus T_{q}^1 $$
and \Cref{prop:gm-decomp-aut} gives us that $\Aut(C)^{\circ}\simeq \Gm^2$. Using \Cref{rem:alternating-def}, one can find a cocharacter $\Gm\rightarrow \Gm^2$ such that
\begin{itemize}
    \item[(a)] $(\Def_C)^>$ classifies deformations where $C_0$ is fixed, as well as the node $p$ and the $A_{r}$-singularity, while the other singularities deform;
    \item[(b)] $(\Def_C)^<$ classifies deformations where $C_0$ is fixed, as well as the node $q$, the node $p_0$, and the $A_{2h+1}$-singularity, while the other singularities deform.
\end{itemize}
\begin{figure}[H]
        \caption{}
        \centering
        \includegraphics[width=1\textwidth]{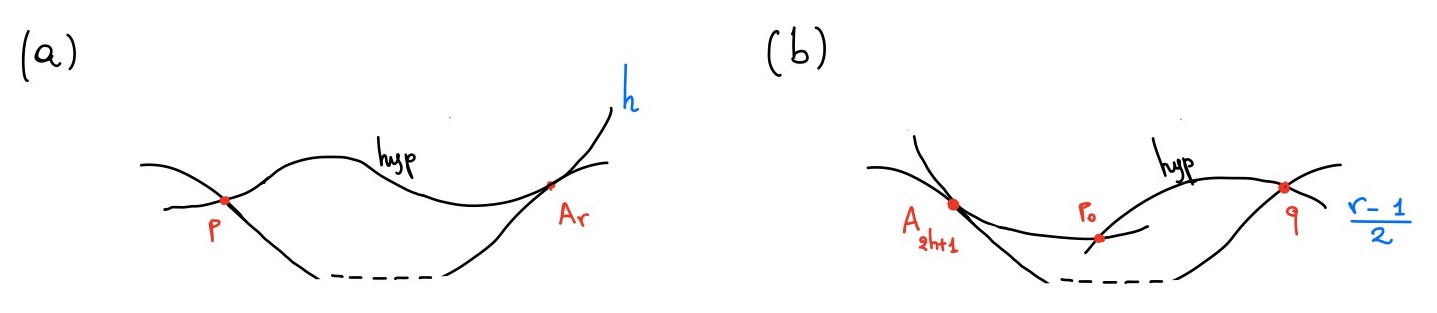}
        \label{fig:ex4.9(2)}
    \end{figure}
Notice that in case (a), we end up with an $A_1/A_r$-attached (non-separating) hyperelliptic bridge (a chain of length one) of genus $<(r-1)/2$. If $\cU$ intersects $(\Def_C)^>$ non-trivially, it means that it generically contains $(a)$. Since $(b)$ is in $\cU_{g,n}^r\subset \cU$, we deduce that $C$ must be in $\cU$ by $\textsf{S}$-completeness. This contradicts \Cref{ex:prob-bridges}, since $C$ has an $A_1/A_1$-attached hyperelliptic bridge of genus $h<(r-1)/2$.
\end{proof}

The next lemma deals with condition $(d_{r,2})$.
\begin{lemma}\label{ex:cond-dr2}
    Let $n,r,g$ be three nonnegative integers. Let $\cU \subset \cM_{g,n}^r$ be an open substack which admits a separated good moduli space. If $\cU_{g,n}^r\subset \cU$, then every curve in $\cU$ does not contain a non-separating $A_r/A_r$-attached hyperelliptic bridge of genus $<(r-1)/2$.
\end{lemma}

\begin{proof}
     Consider the curve $C$ constructed in the following way, cf. \Cref{fig:Chapter 4_5}:
    \begin{itemize}
        \item let $(\Gamma_1,p_0,q)$ be a hyperelliptic rosary of length $3$ with singularities of types $A_{2h+1}$ and $A_r$ with\\ {$1\leq h< (r-1)/2$};
        \item let $(\Gamma_0,p,p_0)$ be an odd atom of genus $(r-1)/2$;
        \item we denote by $(\Gamma,p,q)$ the nodal gluing of $(\Gamma_1,p_0,q)$ with $(\Gamma_0,p,p_0)$ along $p_0$;
        \item we denote by $C$ the curve obtained by nodally gluing $(\Gamma,p,q)$ with a $2$-pointed smooth curve $(C_0,p,q)$ along $p$ and $q$.
    \end{itemize}
Notice that if $r+h=g$, then $(C_0,p,q)$ is isomorphic to $(\bP^{1},0,\infty)$, which is not $A_r$-stable. Nevertheless, the same idea discussed below can be applied to the stabilization of $C$ to obtain the result in this case as well.
    \begin{figure}[H]
        \caption{}
        \centering
        \includegraphics[width=0.6\textwidth]{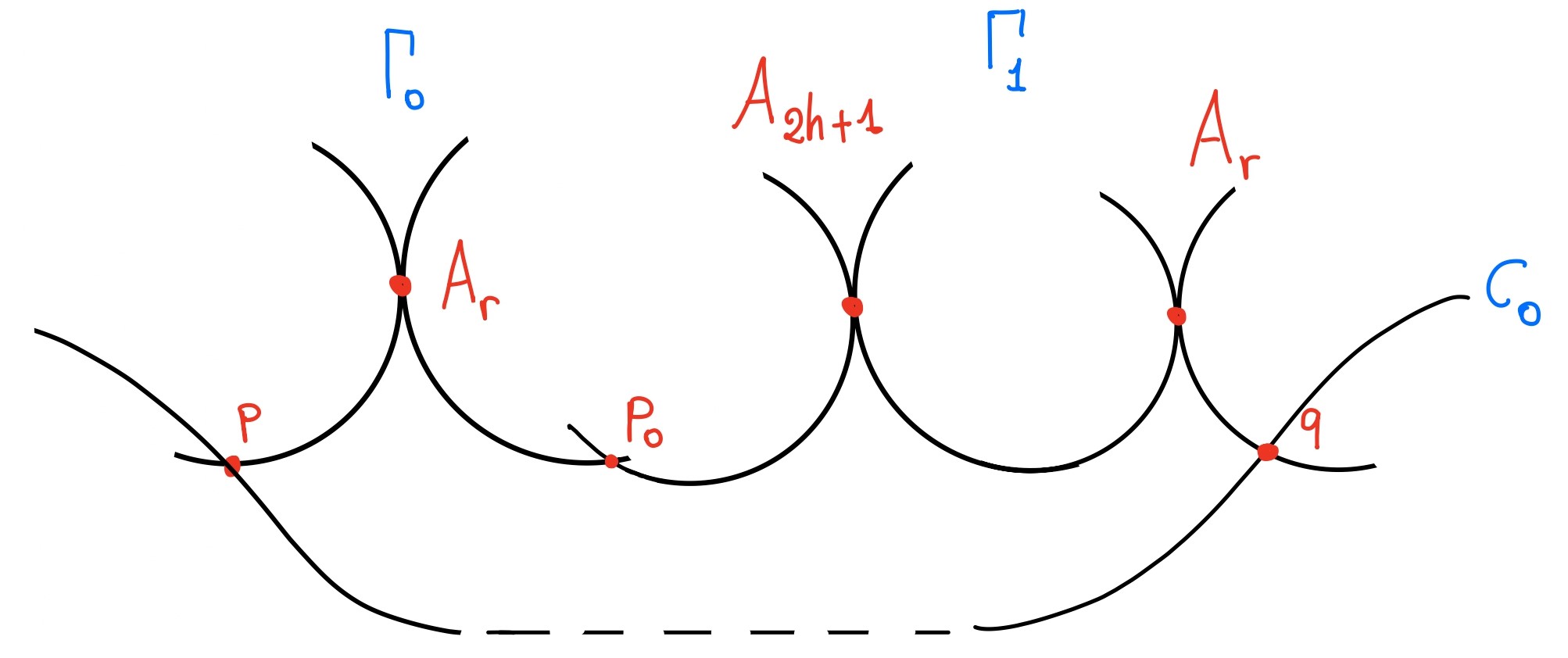}
        \label{fig:Chapter 4_5}
    \end{figure}
     We can describe the deformation space of $C$ as
$$ \Def_C= \Def_{(C_0,p,q)}\oplus T_p^1 \oplus \Def_{(\Gamma_0,p,p_0)} \oplus T_{p_0}^1 \oplus \Def_{(\Gamma_1,p_0,q)} \oplus T_{q}^1 $$
and \Cref{prop:gm-decomp-aut} gives us that $\Aut(C)^{\circ}\simeq \Gm^2$. Using \Cref{rem:alternating-def}, one can find a cocharacter $\Gm\rightarrow \Gm^2$ such that
\begin{itemize}
    \item[(a)] $(\Def_C)^>$ classifies deformations where $C_0$ is fixed, as well as the two $A_{r}$-singularities, while the other singularities deform;
    \item[(b)] $(\Def_C)^<$ classifies deformations where $C_0$ is fixed, as well as all the nodes and the $A_{2h+1}$-singularity, while the other singularities deform.
\end{itemize}

\begin{figure}[H]
        \caption{}
        \centering
        \includegraphics[width=0.9\textwidth]{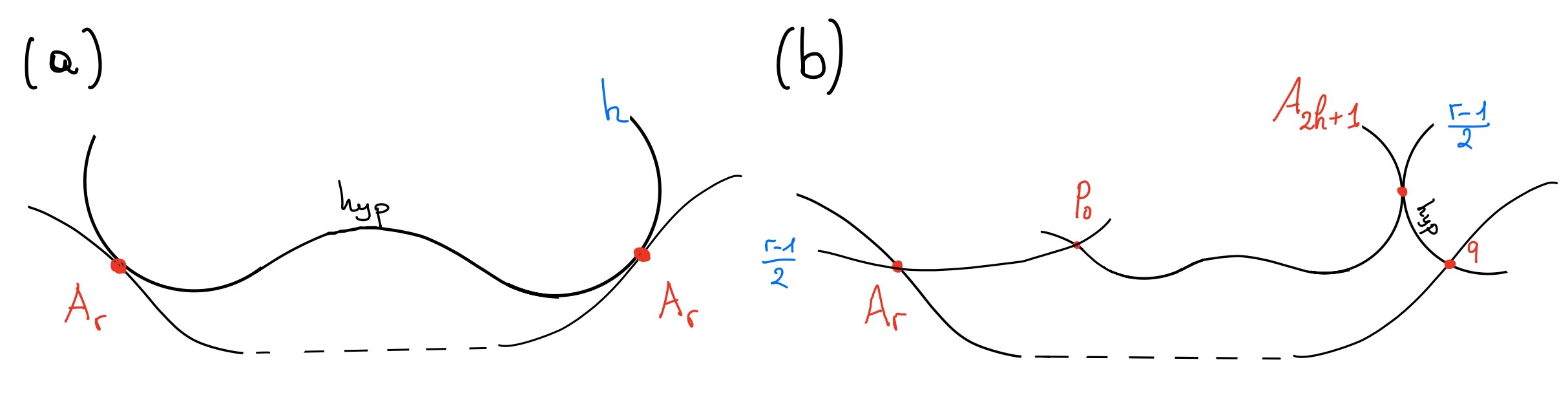}
        \label{fig:Chapter 4_5a}
    \end{figure}

Notice that in case (a), we end up with an $A_r/A_r$-attached (non-separating) hyperelliptic bridge (a chain of length one) of genus $<(r-1)/2$. If $\cU$ intersects $(\Def_C)^>$ non-trivially, it means that it generically contains $(a)$. Since $(b)$ is in $\cU_{g,n}^r\subset \cU$, we deduce that $C$ must be in $\cU$ by $\textsf{S}$-completeness. This is false because of \Cref{ex:cond-dr1}, since $C$ contains an $A_1/A_r$-attached hyperelliptic bridge of genus $h<(r-1)/2$.
\end{proof}

The next lemma explains why conditions $(c_r)$ and $(d_r)=(d_{r,1})+(d_{r,2})$ need to hold for chains, not just bridges.
\begin{lemma}\label{ex:hyp-chains}
    Let $n,r,g$ be three nonnegative integers. Let $\cU \subset \cM_{g,n}^r$ be an open substack which admits a separated good moduli space. If $\cU_{g,n}^r\subset \cU$, then every curve in $\cU$ does not contain a non-separating $A_1/A_1$-attached, $A_1/A_r$-attached, or $A_r/A_r$-attached hyperelliptic chain of genus $<(r-1)/2$ and length $\geq 2$.
\end{lemma}
\begin{proof}
    Let $C$ be a curve in $\cM_{g,n}^r$ constructed as follows, cf. \Cref{fig:Chapter 4_3}:
    \begin{itemize}
        \item let $(\Gamma_0,p,p_0)$ be a $2$-pointed odd atom of genus $h$ with $h<(r-1)/2$ and $(\Gamma_1,q,p_0)$ be a $2$-pointed rosary of hyperelliptic $A$-singularities of length $3$, namely one $A_{2k_1+1}$ (the one closer to $p_0$) and one $A_{2k_2+1}$ singularity; 
        \item let $(\Gamma,p,q)$ be the curve obtained by nodally gluing $\Gamma_0$ and $\Gamma_1$ along $p_0$;
        \item let $(C_0,p,q)$ be a $2$-pointed smooth curve; we denote by $C$ the curve obtained by nodally gluing $(C_0,p,q)$ and $(\Gamma,p_0,p_1)$ along $p$ and $q$.
    \end{itemize}
    \begin{figure}[H]
        \caption{}
        \centering
        \includegraphics[width=0.6\textwidth]{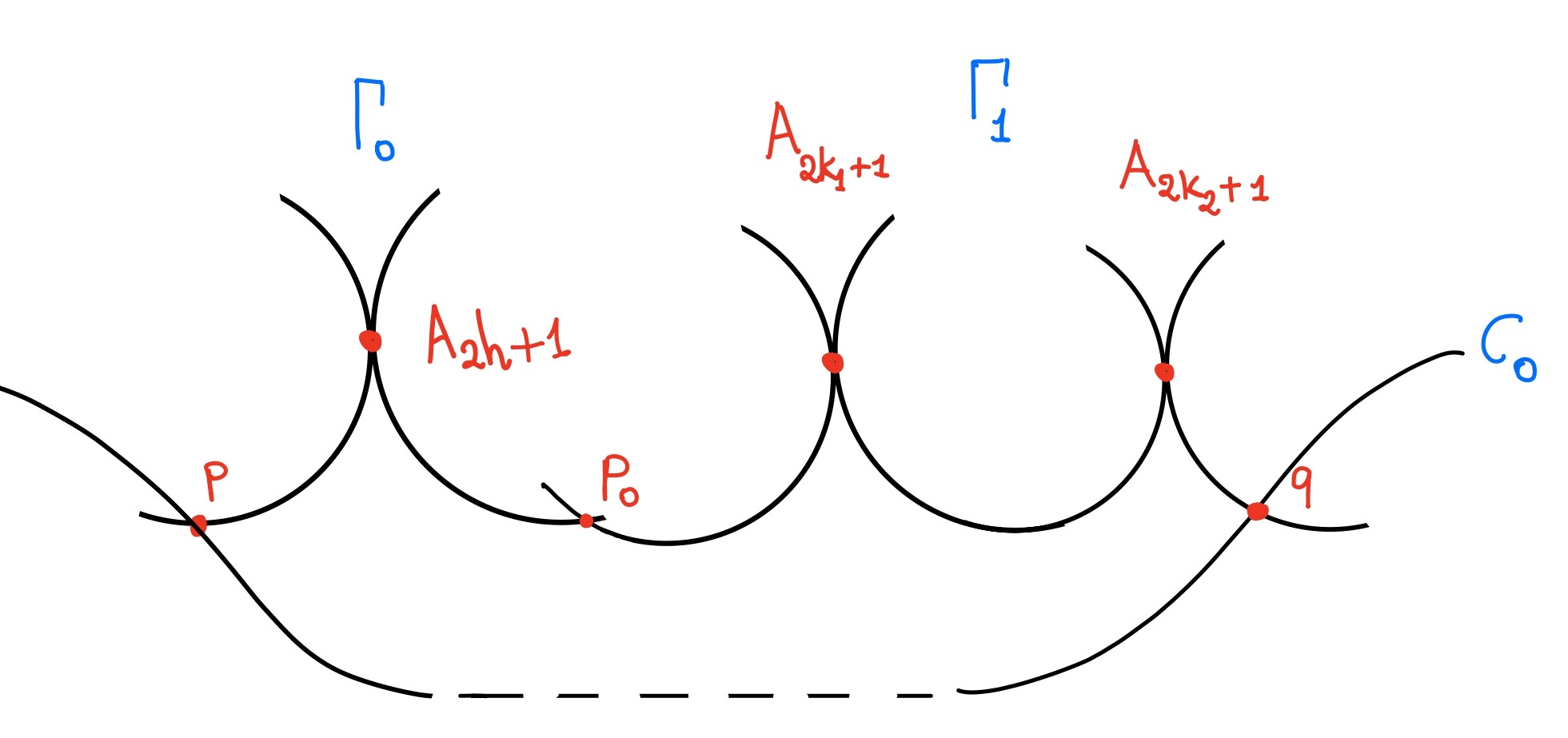}
        \label{fig:Chapter 4_3}
    \end{figure}
    
    We can describe the deformation space of $C$ as
$$ \Def_C= \Def_{(C_0,p,q)}\oplus T_p^1 \oplus \Def_{(\Gamma_0,p,p_0)}\oplus T_{p_0}^1 \oplus \Def_{(\Gamma_1,p_0,q)} \oplus T_{q}^1 $$
and \Cref{prop:gm-decomp-aut} gives us that $\Aut(C)^{\circ}\simeq \Gm^2$. Using the computations carried out in \Cref{rem:alternating-def}, it is easy to find a cocharacter $\Gm\rightarrow \Gm^2$ such that 
\begin{itemize}
    \item[(a)] $(\Def_C)^>$ classifies deformations where $C_0$ is fixed, as well as the nodes $p,q$ and the $A_{2k_1+1}$-singularity, while the three other singularities deform;
    \item[(b)] $(\Def_C)^<$ classifies deformations where $C_0$ is fixed, as well as the node $p_0$ and the $A_{2h+1}$ and $A_{2k_2+1}$ singularities, while the $A_{2k_1+1}$-singularity and the node $p$ deform.
\end{itemize}
\begin{figure}[H]
        \caption{}
        \centering
        \includegraphics[width=0.9\textwidth]{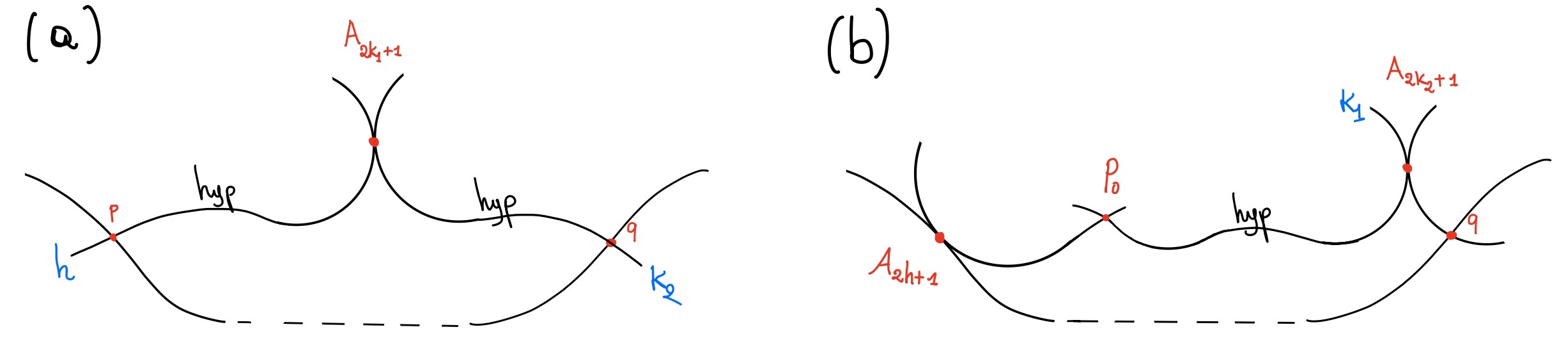}
        \label{fig:Chapter 4_4}
    \end{figure}
Notice that in case (a), we end up with an $A_1/A_1$-attached hyperelliptic chain. If $\cU$ intersects $(\Def_C)^>$ non-trivially, it means that it generically contains $(a)$. Since $(b)$ is in $\cU_{g,n}^r\subset \cU$, we deduce that $C$ must be in $\cU$ by $\textsf{S}$-completeness. This contradicts \Cref{ex:cond-dr1}, since $C$ contains an $A_1/A_1$-attached hyperelliptic bridge of genus $h<(r-1)/2$. Finally, notice that this proof can be upgraded to a proof for chains of arbitrary length by choosing $\Gamma_1$ to be a $2$-pointed rosary of odd length, and the same can be done with the proofs of \Cref{ex:cond-dr1} and \Cref{ex:cond-dr2}, which completes the proof.
\end{proof}
As a consequence of \Cref{ex:prob-bridges}, \Cref{ex:cond-dr1}, \Cref{ex:cond-dr2} and \Cref{ex:hyp-chains}, we obtain the maximality of $\cU$ for the chains.
\begin{corollary}\label{cor:maximality-chains}
     Let $n,r,g$ be three nonnegative integers and suppose that either $n>0$ or $r\leq 2g-4$. Let $\cU \subset \cM_{g,n}^r$ be an open substack which admits a separated good moduli space. If $\cU_{g,n}^r\subset \cU$, then every curve in $\cU$ satisfies conditions $(c_r)$ and $(d_r)$ of \Cref{def:admin}.
\end{corollary}

\subsection{Attached hyperelliptic tails}\label{sub:problematic-tails}

We now explain why condition $(t_r)$ appears.

The first result shows why we need to remove $A_1$-attached hyperelliptic tails of genus $<r/2$.
\begin{lemma}\label{ex:prob-tails}
    Let $n,r,g$ be three nonnegative integers. Let $\cU \subset \cM_{g,n}^r$ be an open substack which admits a separated good moduli space. If $\cU_{g,n}^r\subset \cU$, then every curve in $\cU$ does not contain an $A_1$-attached hyperelliptic tail of genus $h$, where
    \begin{itemize}
        \item $1\leq h <r/2$ if $n>0$;
        \item $3\leq 2h+1\leq \min(r,2g-5)$ if $n=0$.
    \end{itemize}
\end{lemma}
\begin{proof}
    As in the proof of \Cref{ex:prob-bridges}, we handle the case $n>0$ and specify when the inequality is needed for $n=0$. First, we prove that $\cU$ does not contain an $A_1$-attached even atom of genus $h<r/2$.
    Let $R$ be a dvr with fraction field $Q$ and $C\rightarrow \spec R$ be a family of $A_r$-stable curves over $R$ such that
    \begin{itemize}
        \item the generic fiber $C_Q$ is the generic $A_r$-stable curve with an $A_{2h}$-singularity (see  (see \Cref{rem:irreducibility});
        \item the closed fiber $C_k$ is the generic $A_r$-stable curve with an non-separating $A_{2h+1}$-singularity  (see \Cref{rem:irreducibility}).
    \end{itemize}
    Thanks to \Cref{prop:iso-deg-chain}, we can construct an isotrivial degeneration $C_{\Theta} \rightarrow \Theta_Q$ such that the generic fiber coincides with $C_Q$ and the $A_{2h}$-singularity of $C_Q$ degenerates to an even atom of genus $h$ attached to the rest of the curve at a node. Suppose by contradiction that $\cU$ intersects non-trivially the irreducible closed substack parametrizing $A_1$-attached atoms of genus $h$ (see \Cref{rem:irreducibility}). Since $\cU$ is open, we have that the isotrivial degeneration $C_{\Theta} \rightarrow \Theta_Q$ factors through $\cU$. This is equivalent to a morphism $\Theta_R \setminus 0 \rightarrow \cM_g^r$ which does not extend to a morphism from $\Theta_R$ because of \Cref{prop:alp-fund} (see \Cref{fig:ex4.11(1)}). By construction, notice that the morphism from $\Theta_R\setminus 0$ factors through $ \cU$, since the family $C\rightarrow \spec R$ is also contained in $\cU_{g,n}^r\subset \cU$ (exactly as in the proof of \Cref{ex:prob-bridges}, we need the inequality in the case $n=0$ to hold due to condition $(e_r)$). This is a contradiction, since $\cU$ is $\Theta$-complete.
    \begin{figure}[H]
    \centering
    \begin{minipage}{0.45\textwidth}
        \centering
        \includegraphics[width=0.85\textwidth]{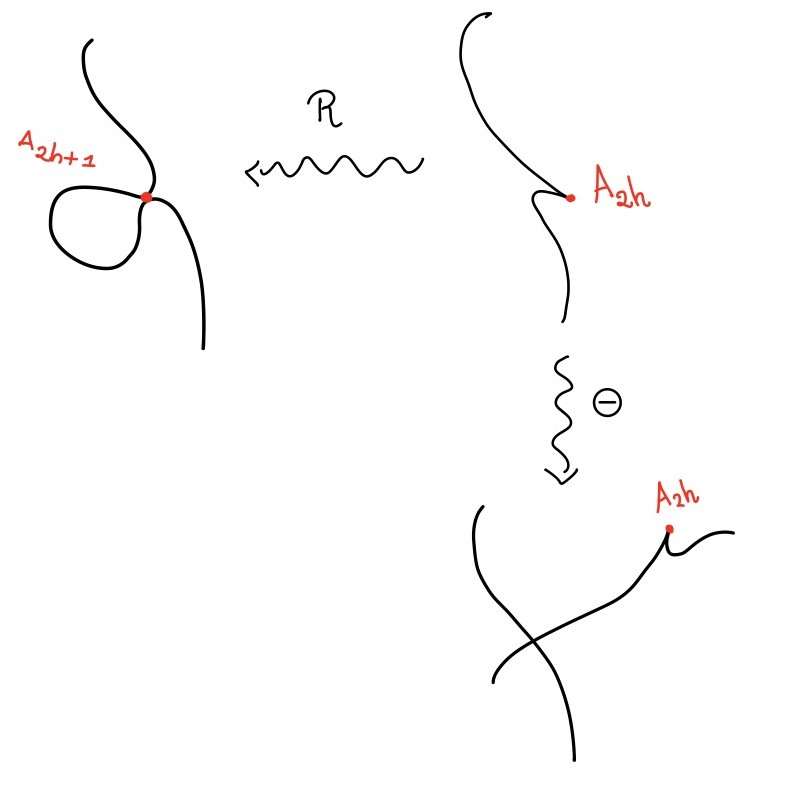}
        \caption{}
        \label{fig:ex4.11(1)}
    \end{minipage}
    \begin{minipage}{0.45\textwidth}
        \centering
        \includegraphics[width=0.85\textwidth]{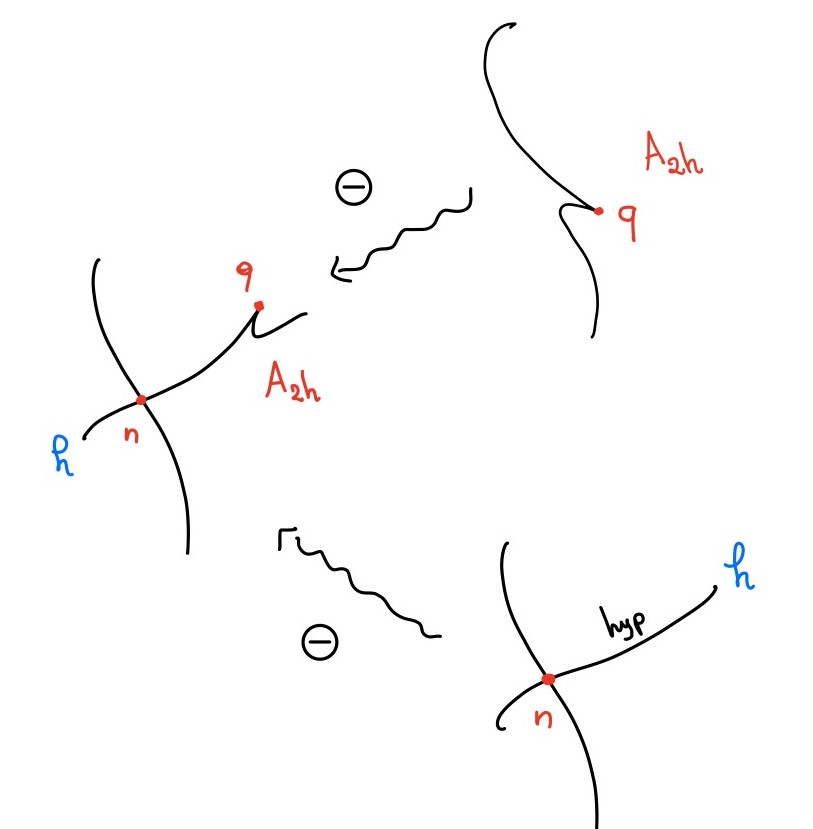}
        \caption{}
        \label{fig:ex4.11(2)}
    \end{minipage}
    \end{figure}
   Therefore, there are no curves in $\cU$ with atoms of the specified form. Let us now show that there are no $A_1$-attached hyperelliptic tails of genus $<r/2$. By contradiction, we can assume there exists a curve $C_{1,0}$ together with an $A_1$-attached hyperelliptic bridge $\Gamma_{0,1}\subset C_{0,1}$ of genus $<r/2$. Since $\cU$ is an open substack, we can assume that $\Gamma_{0,1}$ is honestly hyperelliptic and $C_{0,1}-\Gamma_{0,1}$ is smooth. As discussed in \Cref{sub:global geometry}, we can find an isotrivial degeneration of $C_{1,0}$ where the special fiber $C$ is obtained by nodally gluing the $1$-pointed even atom $\Gamma$ of genus $h$ to $C-\Gamma$; in particular, $C$ is not in $\cU$. The $\Gm$-equivariant deformation space of $C$ in $\cM_{g,n}^r$ is
    $$ \Def_{(C-\Gamma,n)}\oplus [T^1_q  \oplus ({\rm Cr}_q \oplus T^1_{n})/\Gm]$$
    where $q$ is the even singularity and $n$ is the nodal attachments. Since $C \notin \cU$ by what we proved above, locally around $C$, $\cU$ is contained in the complement of the $0$-section of the deformation space. Since $\cM_{g,n}^r$ has affine diagonal (thus the local structure map is affine) and $\cU$ is $\textit{S}$-complete, we deduce that one of the two \emph{basins of attraction}, see \Cref{fig:ex4.11(2)}, i.e., either $T_q^1\oplus (0\oplus 0)$ or $0 \oplus ({\rm Cr}_q \oplus T^1_{n})$, does not intersect $\cU$ (otherwise we would have a diagram of specializations like the one in \Cref{fig:ex4.11(2)}). Notice that $C_{0,1}$ corresponds to a point in $T_q^1\oplus (0 \oplus 0)$.
    Since $0 \oplus ({\rm Cr}_q \oplus T^1_{n})$ is contained in $\cU_{g,n}^r\subset \cU$, we conclude that $C_{0,1} \notin \cU$, which completes the proof.
\end{proof}

\begin{lemma}\label{ex:proble-tail-Ak}
    Let $n,r,g$ be three nonnegative integers. Let $\cU \subset \cM_{g,n}^r$ be an open substack which admits a separated good moduli space. If $\cU_{g,n}^r\subset \cU$, then every curve in $\cU$ does not contain an $A_{2k+1}$-attached hyperelliptic tail of genus $<(r-1)/2$.
\end{lemma}

\begin{proof}
    Let $(D,p_1,p_2)$ be a $2$-pointed rosary of hyperelliptic $A$-singularities of length $3$ and $(C_0,p_1,p_2)$ be a $2$-pointed smooth curve. We denote by $C$ the curve obtained by nodally gluing $(C_0,p_1,p_2)$ and $(D,p_1,p_2)$ along $p_1$ and $p_2$. \Cref{prop:iso-deg-chain} allows us to construct an isotrivial degeneration to $C$ which deforms one of the nodes and one of the $A$-singularities of $D$. Therefore, the generic fiber would consist of two smooth curves meeting in two outer singularities of type $A_{2k+1}$ and a node. If $2k+3\leq r$, we can degenerate these two singularities into an outer (in fact separating) $A_{2k+3}$-singularity. As before, if there exists a curve in $\cU$ which contains an $A_{2k+3}$-attached hyperelliptic tail of genus $<(r-1)/2$, then we can construct a morphism $\Theta_R\setminus 0 \rightarrow \cU$, see \Cref{fig:ex4.12}, which does not extend to a morphism from $\Theta_R$ because of Proposition 2.10 of \cite{AlpFedSmyWyck} (also restated in \Cref{prop:alp-fund}). A similar construction can be carried out if $(D,p_1,p_2)$ has length $3$ or length $1$.
    \begin{figure}[H]
        \caption{}
        \centering
        \includegraphics[width=0.6\textwidth]{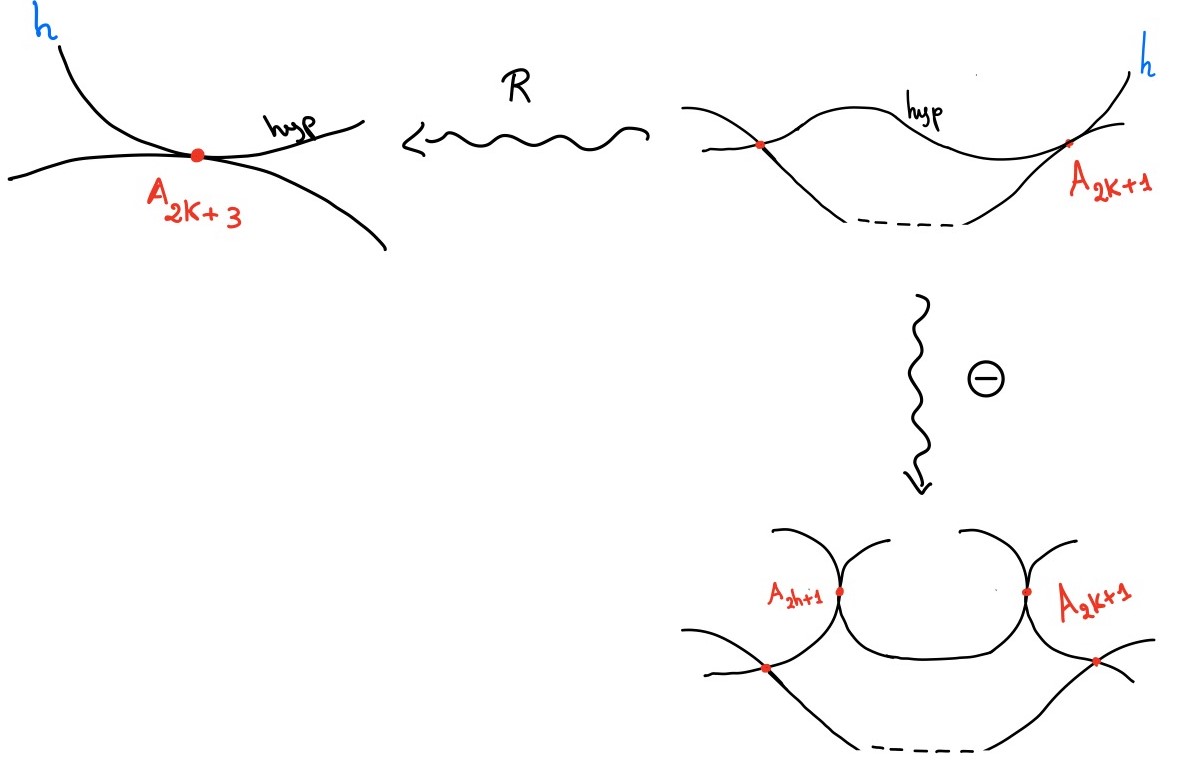}
        \label{fig:ex4.12}
    \end{figure}
    
\end{proof}

\begin{remark}\label{rem:missing-tails}
    Notice that this is not exactly condition $(t_r)$ as in \Cref{def:admin}: in fact, we have not yet removed the $A_{2k+1}$-attached hyperelliptic tails of genus $h=(r-1)/2$ when $r$ is odd for $k\geq 2$. This will follow from the next lemma.
\end{remark}

\begin{lemma}\label{ex:proble-tail-last}
    Let $n,r,g$ be three nonnegative integers. Let $\cU \subset \cM_{g,n}^r$ be an open substack which admits a separated good moduli space. If $\cU_{g,n}^r\subset \cU$, then every curve in $\cU$ does not contain an $A_{2k+1}$-attached hyperelliptic tail of genus $<r/2$.
\end{lemma}

\begin{proof}
    Thanks to \Cref{ex:proble-tail-Ak}, we can assume $r$ is odd and $h=(r-1)/2$ (see also \Cref{rem:missing-tails}). Let $C$ be the curve constructed as follows, see \Cref{fig:ex4.14(1)}:
    \begin{itemize}
        \item Let $(\Gamma_0,p_0,p)$ be a $2$-pointed odd atom of genus $k$ with $k<(r-1)/2$ and $(\Gamma_1,p_0)$ be a $1$-pointed even atom of genus $h$;
        \item we denote by $(\Gamma,p)$ the $1$-pointed curve obtained by nodally gluing $(\Gamma_0,p_0,p)$ with $(\Gamma_1,p_0)$ along $p_0$;
        \item we denote by $C$ the curve obtained by nodally gluing $(\Gamma,p)$ with the $1$-pointed smooth curve $(C_0,p)$ along $p$;
        \begin{figure}[H]
        \caption{}
        \centering
        \includegraphics[width=0.4\textwidth]{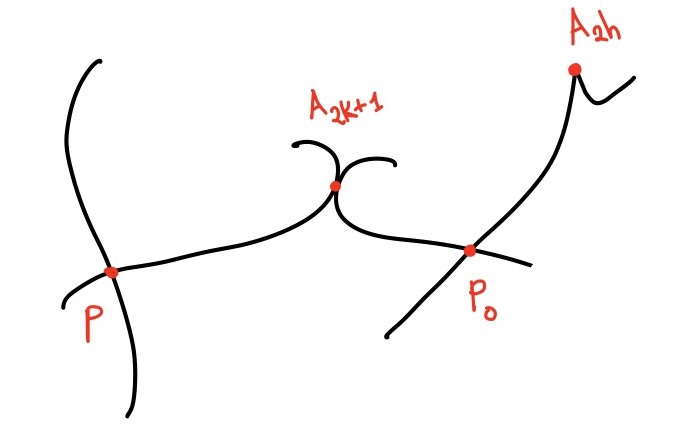}
        \label{fig:ex4.14(1)}
    \end{figure}
    \end{itemize}
     We can describe the deformation space of $C$ as
    $$ \Def_C= \Def_{(C_0,p)}\oplus T_p^1 \oplus \Def_{(\Gamma_0,p_0,p)}\oplus T_{p_0}^1 \oplus \Def_{(\Gamma_1,p_0)} $$
    and \Cref{prop:gm-decomp-aut} gives us that $\Aut(C)^{\circ}\simeq \Gm^2$. Similarly to previous cases where $\Aut(C)^{\circ}\simeq \Gm^2$, a simple computation using \Cref{rem:alternating-def} gives a cocharacter $\Gm\rightarrow \Gm^2$ such that
    \begin{itemize}
        \item[(a)] $(\Def_C)^>$ parametrizes curves with an $A_{2k+1}$-attached hyperelliptic tail of genus $h$;
        \item[(b)] $(\Def_C)^<$ parametrizes curves with an $A_1/A_1$-attached \emph{separating} bridge of genus $k$ with a rational tail that contains an $A_{2h}$-singularity;
    \end{itemize}
    \begin{figure}[H]
        \caption{}
        \centering
        \includegraphics[width=0.9\textwidth]{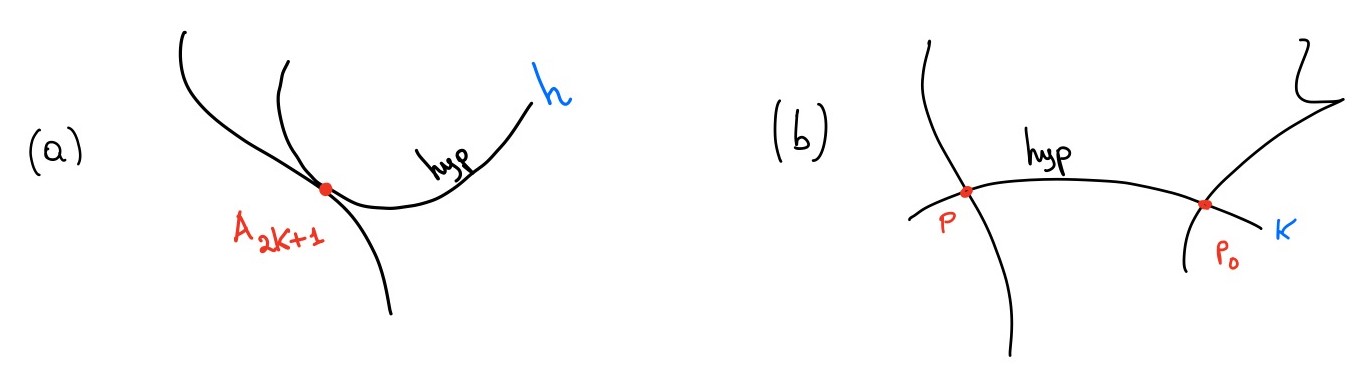}
        \label{fig:ex4.14(2)}
    \end{figure}
    Notice that in case (a), we end up with an $A_{2k+1}$-attached hyperelliptic tail of genus $h$. If $\cU$ intersects $(\Def_C)^>$ non-trivially, it means that it generically contains $(a)$. Since $(b)$ is in $\cU_{g,n}^r\subset \cU$, we deduce that $C$ must be in $\cU$ by $\textsf{S}$-completeness. This contradicts \Cref{ex:prob-tails}, since $C$ contains an $A_1$-attached hyperelliptic tail of genus $h<r/2$.
\end{proof}

\begin{remark}
    Notice that in the previous scenario, we need the $A_{2h}$-singularity to have a non-trivial crimping space, i.e., $r\geq 2h\geq 4$. Nevertheless, it was shown in \cite[Theorem~F]{AltCompClustAlg} that for $r=3$, we still have to remove $A_3$-attached elliptic tails if we want to construct modular compactifications that contain curves with non-separating tacnodes (see also \emph{Case 1} of the proof of \cite[Theorem~3.8.2]{AltCompClustAlg}).
\end{remark}

Finally, we prove the result for the $A_0$-attached tails. It only appears in a very specific ``asymmetric'' situation, and it is strictly related to the fact that the stack $\cH_{g,1}^{2g+1,\circ}$ of $1$-pointed double covers of $\bP^1$ of genus $g$ does not have a good moduli space, even though it has a stratification $\cH_{g,w}^{2g+1,\circ} \subset \cH_{g,1}^{2g+1,\circ}$ where both $\cH_{g,w}^{2g+1,\circ}$ and $\cH_{g,g_1^2}^{r,\circ}:=\cH_{g,1}^{2g+1,\circ}\setminus \cH_{g,w}^{2g+1,\circ}$ have good moduli spaces (which are points in both cases). We believe that these two are exactly the (images of non-trivial) $\Theta$-strata of $\cH_{g,1}^{2g+1,\circ}$ (see also \cite[Remark 3.49]{GMSArI}).

\begin{lemma}\label{ex:n=1-hyp}
Let $n,r,g$ be three nonnegative integers such that $g \geq 2$. Let $\cU \subset \cM_{g,n}^r$ be an open substack which admits a separated good moduli space. If $\cU_{g,n}^r\subset \cU$, then every curve in $\cU$ does not contain an $A_0$-attached hyperelliptic tail of genus $<r/2$.
\end{lemma}
\begin{proof}
     If \(r\leq2g\), there is nothing to prove, so let us assume \(r>2g\). Moreover, if $\cU$ intersects the locus of $A_0$-attached hyperelliptic tails of genus $g$, then we know it intersects the locus of $A_0$-attached honestly hyperelliptic tails, as \(\cU\) is an open substack. Let $R$ be a dvr with fraction field $Q$ and $C\rightarrow \spec R$ be a family of $1$-pointed double covers of $\bP^1$ over $R$ such that the marking of the generic fiber $C_Q$ is not Weierstrass, but it degenerates through $R$ to a Weierstrass point. We can construct an isotrivial degeneration $C_{\Theta} \rightarrow \Theta_Q$ such that the generic fiber coincides with $C_Q$ and the special fiber is a $1$-pointed odd atom of genus $g$, see \Cref{fig:ex4.15}.
     \begin{figure}[H]
        \caption{}
        \centering
        \includegraphics[width=0.6\textwidth]{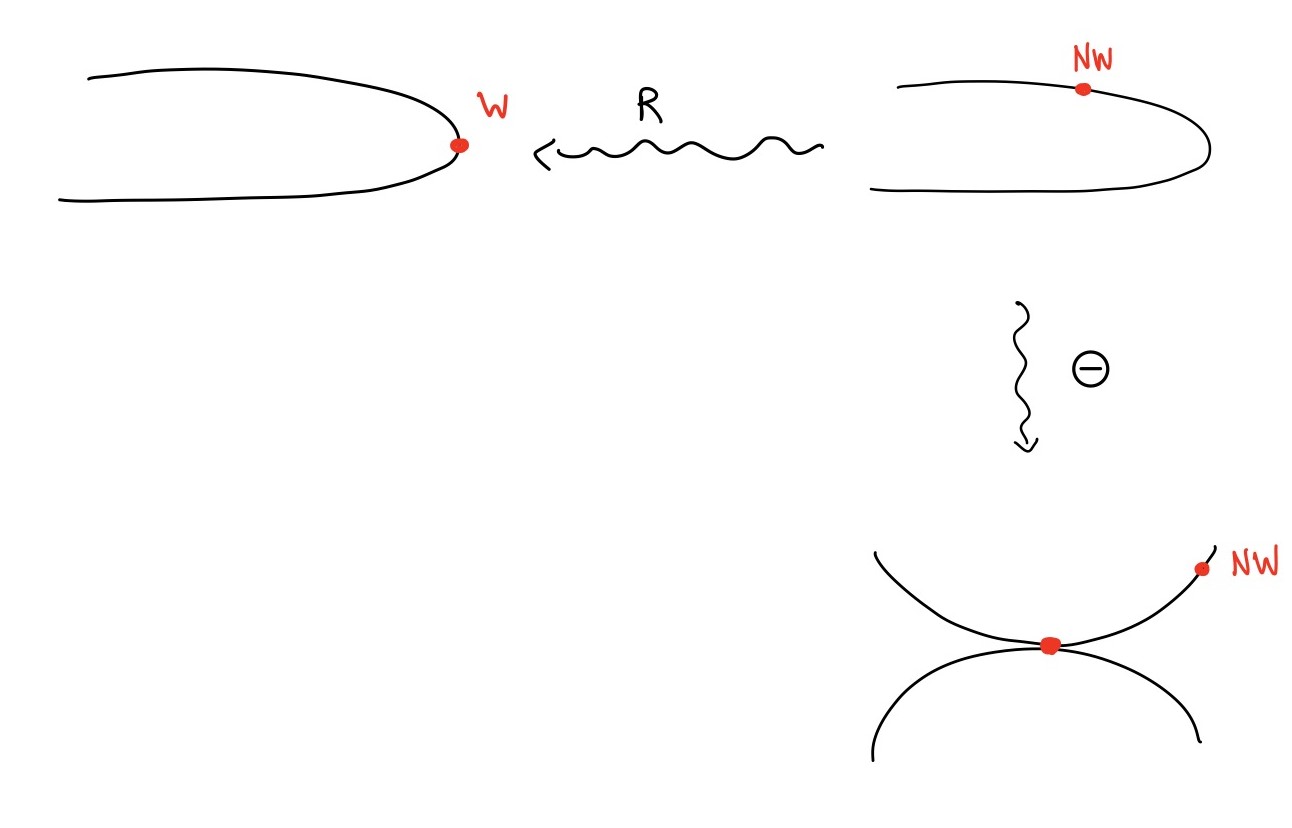}
        \label{fig:ex4.15}
    \end{figure}
      We constructed a morphism $\Theta_R \setminus 0 \rightarrow \cU$ which does not extend to a morphism from $\Theta_R$ because the $1$-pointed odd atom is a closed point in $\cM_{g,1}^{2g+1}$ and does not have smooth Weierstrass points, while marked points need to be smooth.
\end{proof}
\section{Good moduli properties of \texorpdfstring{$\cU_{g,n}^r \subset \cM_{g,n}^r$}{U C M}} \label{sec:open-theta}

In this section, we prove \Cref{theo:local-completeness} by applying \cite[Proposition~2.12]{GMSArI}, which we recall here:
\begin{proposition}[{\cite[Proposition~2.12]{GMSArI}}]\label{prop:local-Theta-complete}
    Let $\cX$ be an algebraic stack smooth over a field with an affine diagonal and linearly reductive stabilizers (at every point). Let $\cZ$ be a closed subset of $\cX$. The following are equivalent:
    \begin{itemize}
        \item[(i)]  for every geometric point $x:\spec k \rightarrow \cZ$ and every cocharacter $\bG_{m,k} \rightarrow G_x$, the (linearized) $G_x$-equivariant closed subset $Z_x \subset T_x \cX$ contains either $T_x \cX^{>}$ or $T_x \cX^0$ (respectively  $T_x \cX^{>}$ or $T_x \cX^{<}$);
        \item[(ii)] the open embedding $\cU:=\cX \setminus \cZ \subset \cX$ is $\Theta$-complete (respectively $\textsf{S}$-complete). 
    \end{itemize} 
\end{proposition}
\begin{remark}
    In the above, \(Z_x\subset T_x\) is defined by pulling back \(\mathcal{Z}\subset\mathcal{X}\) along the coherent completion \(\eta:(\widehat{\mathcal{X}},x)\to(\mathcal{X},x)\), which exists by \cite[Corollary 1.12]{AlpHalRyd}, followed by taking the scheme theoretic
    image of \(\eta^{-1}(\mathcal{Z})\) under the completion map \((\widehat{\mathcal{X}},x)\to ([T_x/G_x],0)\) and then pulling back along the atlas \(T_x\to [T_x/G_x]\), see \cite[Remarks~2.11~and~2.14]{GMSArI}. The key fact used in this construction is that the completions of \((\mathcal{X},x)\) and \(([T_x/G_x], 0)\) are isomorphic (see \cite[Remark~2.9]{GMSArI}).  
\end{remark}

\begin{remark}[Condition $(\star)$]
    \label{rmk:condition_star}
    Under the hypotheses of the above proposition, we have that $\cU \subset \cX$ is $\Theta$- and $\textsf{S}$-complete if and only if for every geometric point $x \colon \spec k \to \cZ$, the following holds:
\begin{itemize}
        \item[$(\star)$] for every cocharacter $\Gm \rightarrow \Stab_{\cX} x$, we have that $Z_x$ contains either $T_ x \cX ^{\leq}$ or $T_ x \cX^{>}$.
\end{itemize}
\end{remark}

Throughout this section, \(\cZ\) denotes the complement of \(\cU_{g,n}^r\) in \(\cM_{g,n}^r\).

\begin{theorem}\label{theo:local-completeness}
    The open embedding $\cU_{g,n}^r \subset \cM_{g,n}^r$ is $\Theta$-complete and $\textsf{S}$-complete.
\end{theorem}

\begin{proof}
    The idea is to verify condition $(\star)$. We denote by $\cZ$ the closed complement of $\cU_{g,n}^r$ in $\cM_{g,n}^r$, and suppose we are given a geometric point $x:\spec k \rightarrow \cZ$. We want to prove that for every cocharacter $\Gm\rightarrow G_x$, the linearized closed subset $Z_x\subset T_x\cX$ contains either $T_x\cX^{>}$ or $T_x\cX^{\leq}$. By contradiction, suppose it does not. Then:
    \begin{itemize}
        \item \Cref{lem:theta-tails} implies that the curve parametrized by $x$ satisfies condition $(t_r)$ of \Cref{def:admin};
        \item \Cref{lem:theta-chain} implies that the curve parametrized by $x$ satisfies conditions $(c_r)$ and $(d_r)$ of \Cref{def:admin};
        \item \Cref{lem:theta-n=0} implies that the curve parametrized by $x$ satisfies condition $(e_r)$ of \Cref{def:admin}.
    \end{itemize}
    This implies that $x$ factors through the admissible open substack $\cU_{g,n}^r$, which is a contradiction.
\end{proof}

\begin{remark}
    Given that $\cM_{g,n}^r$ has affine diagonal, \Cref{theo:local-completeness} would immediately follow if $\cU_{g,n}^r$ were $\Theta$- and $\textsf{S}$-complete (i.e., if $\cU_{g,n}^r$ admitted a separated good moduli space). Indeed, extensions of maps \(\Theta_R\setminus \{0\}\to \mathcal{X}\) to maps \(\Theta_R\to \mathcal{X}\) are unique if \(\mathcal{X}\) has affine diagonal.
\end{remark}

In the following subsections, we prove the lemmas used in the above proof. In particular, we make extensive use of \Cref{prop:local-Theta-complete}, which reduces the problem to studying the isotypical components of $\Gm$-representations (as cocharacters of automorphism groups acting on deformation spaces). For this purpose, we introduce the following notation:
\begin{definition}[Sign]
    \label{def:sign_of_rep}
    A $\Gm$-representation $V$ decomposes as $V = \bigoplus_{i \in \bZ} V_i$, where the action on $V_i$ has weight $i$. We say that $\sgn(V) = +1$ (resp. $-1$, $0$) if $V_i \neq 0$ implies $i >0$ (resp. $i <0$, $i=0$). Moreover, for $a\in \bZ$, we write $\sgn(a) \in \{+1,0,-1\}$ to denote the usual sign.
\end{definition}

\subsection{Condition on tails}

Before proceeding with the proof of \Cref{lem:theta-tails}, we illustrate the argument with a simple example.

\begin{example}\label{ex:easy-warm-up}
    Consider a point of $\cM_{g,n}^r$ corresponding to a curve $C/k$ that is constructed by attaching a $1$-pointed even atom $(\Gamma,p)$ of genus $h$ (where $2h<r$) to a smooth curve $C_0$ at a node $p$.

    Note that $\Gamma \subset C$ is a hyperelliptic tail of genus $h<r/2$, and thus $[C]$ is contained in $\cZ$. First, $\Aut(C)^{\circ}\simeq \Gm$, and we can describe the tangent space as 
    $$T_{[C]}\cM_{g,n}^r \simeq \Def_{(C_0,p)} \oplus T_p^1 \oplus \Def_{(\Gamma,p)}$$
    where $T_p^1$ is the deformation space of the node. Recall that we have the decomposition 
    $$ \Def_{(\Gamma,p)}\simeq \Def_{(\Gamma,p)}^h \oplus \Def_{(\Gamma,p)}^{nh} $$
    where we denote by $\Def_{(\Gamma,p)}^{h}:=T_q^1$ the deformation space of the $A_{2h}$-singularity $q$, which coincides with the deformations that keep $\Gamma$ hyperelliptic and $p$ a Weierstrass point (see \cite[Proposition 4.30]{GMSArI}). Furthermore, we identified the complement $\Def_{(\Gamma,p)}^{nh}$ with the crimping space of the singularity. Using the description of the tangent provided in \cite[Corollary~4.33]{GMSArI} (see also \Cref{rem:alternating-def}) and recalling condition $(t_r)$ in \Cref{def:admin}, we can describe $Z_{[C]}$ as the linear subspace
\[
Z_{[C]} = \Def_{(C_0,p)} \oplus 0 \oplus \Def_{(\Gamma,p)}^h,
\]
where we can assume that the automorphisms induced by the attached atom $\Gm \subset \Aut(C)$ act with weight one on $T_p^1$, thus negatively on $\Def_{(\Gamma,p)}^{h}$ and positively on $\Def_{(\Gamma,p)}^{nh}$.

Suppose there exists a cocharacter $\Gm \rightarrow \Gm\subset \Aut(C)$, i.e., an integer $a \in \bZ$, such that $Z_{[C]}$ contains neither $(T_{[C]}\cM_{g,n}^r)^{>}$ nor $(T_{[C]}\cM_{g,n}^r)^{\leq }$. Note that the action induced by the cocharacter will have weight $a$ on $T_p^1$, act via strictly positive multiples of $a$ on $\Def_{(\Gamma,p)}^{nh}$, act via strictly negative multiples of $a$ on $\Def_{(\Gamma,p)}^{h}$, and act trivially on $\Def_{(C_0,p)}$. Thus, if $a\leq 0$, then $$(T_{[C]}\cM_{g,n}^r)^{>}\subset 0\oplus 0\oplus \Def^h_{(\Gamma,p)}\subset Z_{[C]};$$
    therefore $a>0$. Similarly, the condition $(T_{[C]}\cM_{g,n}^r)^{\leq } \nsubseteq Z_{[C]} $ implies $a \leq 0$. Thus, we have a contradiction. For the sake of completeness, if $\lambda \colon \Gm \subset \Aut(C)$ acts with positive weight on $T^1_p$, we have the following decomposition:
    \begin{align*}
        (T_{[C]}\cM_{g,n}^r)^{>}& = 0\oplus T^1_p \oplus 0 \oplus \Def^{nh}_{(\Gamma,p)};\\
         (T_{[C]}\cM_{g,n}^r)^{<}& = 0\oplus 0\oplus \Def^h_{(\Gamma,p)} \oplus 0;\\
         (T_{[C]}\cM_{g,n}^r)^{0}& = \Def_{(C_0,p)} \oplus 0\oplus 0\oplus 0,
    \end{align*}
    Therefore $(T_{[C]}\cM_{g,n}^r)^{\leq} = Z_{[C]}$.
\end{example}

Following the previous example as a prototype, we now prove the following lemma.

\begin{lemma}\label{lem:theta-tails}
    If $C: \spec k \rightarrow \cZ\subset \cM_{g,n}^r$ is a curve that does not satisfy condition $(t_r)$ of \Cref{def:admin} (i.e., $C$ has a hyperelliptic tail of genus $h<r/2$), then $C$ does satisfy condition $(\star)$.
\end{lemma}

\begin{proof}
Let $\Gamma \subset C$ be a hyperelliptic tail of genus $2h<r$, attached at an $A_{2k+1}$-singularity $p_0$. Note that $(\Gamma,p_0)$ is an honestly hyperelliptic curve by \Cref{lem:cyclic-covers} and \Cref{rem:cyclic-covers}. Let us call $C_0$ the subcurve $C-\Gamma$. Again, we have a decomposition
$$ T_{[C]}\cM_{g,n}^r \simeq \Def_{(C_0,p_0)} \oplus \Def_{p_0} \oplus \Def_{(\Gamma,p_0)}$$
where $\Def_{p_0}$ is the deformation space of the singularity $p_0$, which can be decomposed as the local deformations $T_{p_0}^1$ and the crimping space ${\rm Cr}_{p_0}$. Given a cocharacter $\Gm \rightarrow \Aut(C)$, we have an induced $\Gm$-action on the (one-dimensional) tangent space $T_{p_0}C_0$ (respectively $T_{p_0}\Gamma$) of the point $p_0$ in $C_0$ (respectively in $\Gamma$). Note that $p_0$ is indeed a smooth point of $C_0$. If $k\geq 1$, then the two $\Gm$-actions need to be balanced (i.e., have the same weight; see \cite[Lemma 3.37]{GMSArI}). Suppose this is the case and denote the weight by $a$. If $a=0$, we have
$$0=(T_{[C]}\cM_{g,n}^r)^{>} \subset \Def_{(C_0,p_0)}\oplus 0\oplus 0 \subset Z_{[C]}$$ 
thus we can assume $a\neq 0$. Therefore, we have that $(\Gamma,p_0)$ is actually a $1$-pointed even atom of genus $h$ attached to $C_0$ at $p_0$ by \cite[Corollary 3.50]{GMSArI}. Then it is clear that 
$$ \Def_{(C_0,p_0)}\oplus (0\oplus {\rm Cr}_{p_0}) \oplus \Def_{(\Gamma,p_0)}^h \subset Z_{[C]}$$

We can then repeat an argument similar to \Cref{ex:easy-warm-up} to show that for every integer $a$, either $(T_{[C]}\cM_{g,n}^r)^{>}$ or $(T_{[C]}\cM_{g,n}^r)^{\leq}$ is contained in $\Def_{(C_0,p_0)}\oplus (0\oplus {\rm Cr}_{p_0}) \oplus \Def_{(\Gamma,p_0)}^h$. Therefore, the same is true for $Z_{[C]}$.

We now address the case where the two $\Gm$-actions do not need to be balanced, meaning $p_0$ must be a node. If the action of $\Gm$ on $T_{p_0}C_0$ is trivial, we leave it to the reader to verify that the situation reduces to the one analyzed in \Cref{ex:easy-warm-up}. If this is not the case, \Cref{prop:gm-decomp-aut} (see also \cite[Proposition 4.23]{GMSArI}) shows that there exists a subcurve $\Gamma_1 \subset C_0 \subset C$ containing $p_0$ such that $(\Gamma_1,P_{\Gamma_1} \cup I_{\Gamma_1})$ is either a rosary of hyperelliptic $A$-singularities or does not have any odd $A$-singularities. Assuming the former case first, let $p_1$ be the $A_{2k+1}$-singularity (for $k\geq 1$) that lies in the same rational irreducible component of $\Gamma_1$ as $p_0$ (cf. \Cref{fig:Chapter 5_1}).
\begin{figure}[H]
        \caption{}
        \centering
        \includegraphics[width=0.6\textwidth]{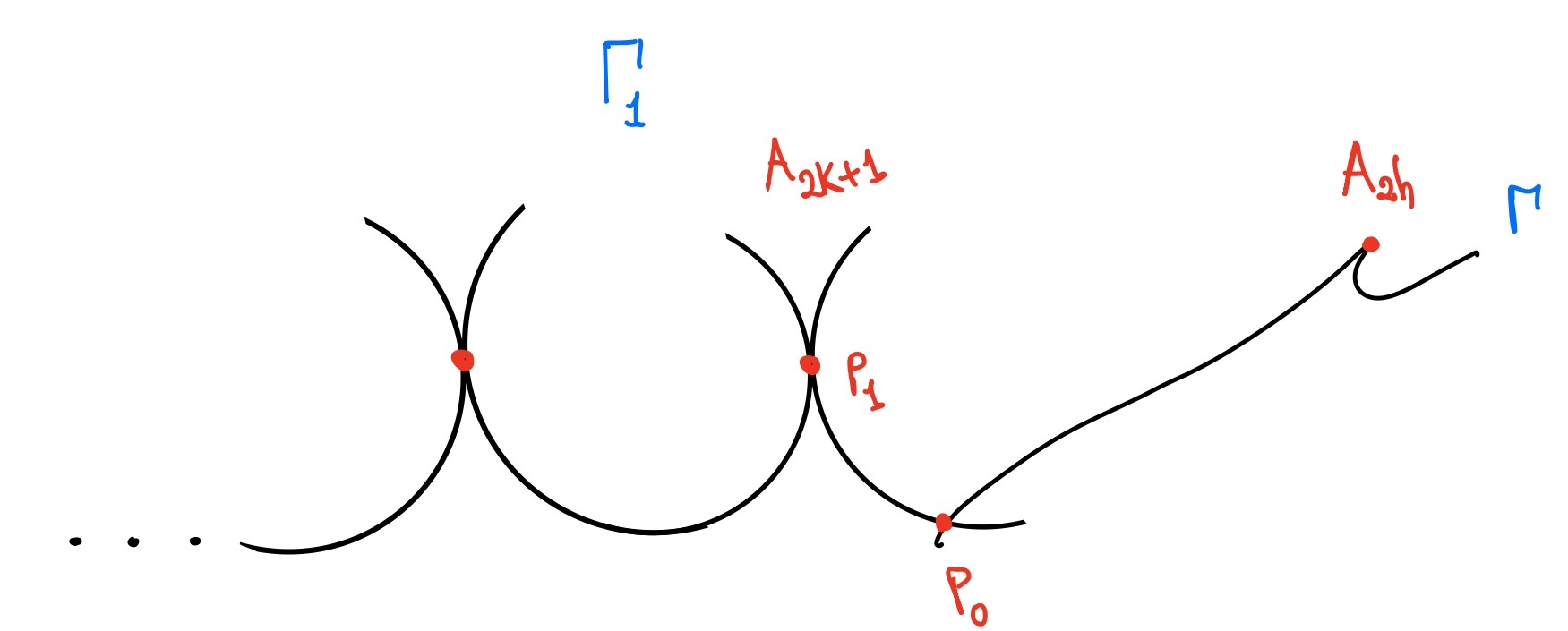}
        \label{fig:Chapter 5_1}
\end{figure}

Then we get the decomposition
$$ T_{[C]}\cM_{g,n}^r \simeq R \oplus \Def_{p_1} \oplus T^1_{p_0} \oplus \Def_{(\Gamma,p_0)}$$
where $R$ is the ($\Gm$-equivariant) complement of the right-most three spaces (i.e., the deformation space of the rest of the curve). As before, recall that we have the decomposition 
$$ \Def_{p_1}\simeq T_{p_1}^1 \oplus {\rm Cr}_{p_1}.$$
By construction, we have that $Z_{[C]}$ contains two linear subspaces, namely $V_1:=R\oplus \Def_{p_1} \oplus 0 \oplus \Def_{(\Gamma,p_0)}^h$ and $V_2:=R\oplus (0 \oplus {\rm Cr}_{p_1}) \oplus T^1_{p_0} \oplus \Def_{(\Gamma,p_0)}^h$, which correspond to the following two generic deformations:\begin{figure}[H]
        \caption{}
        \centering
        \includegraphics[width=0.9\textwidth]{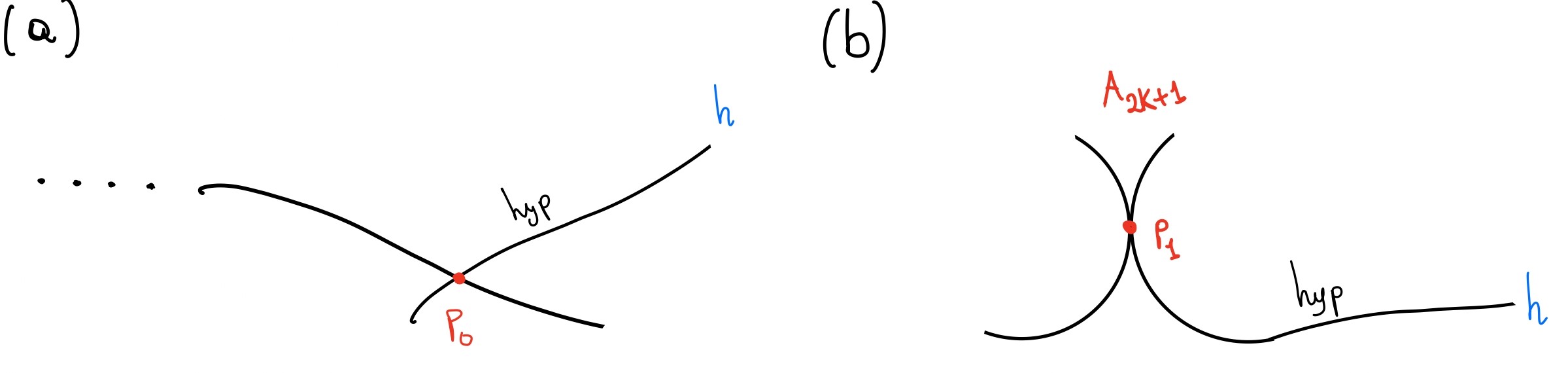}
        \label{fig:Chapter 5_2}
\end{figure}

Given a cocharacter $\Gm \rightarrow \Aut(C)$, we denote by the integer $a$ the weight induced on the (one dimensional) tangent space $T_{p_0}\Gamma$, and by $b$ the weight induced on the (one-dimensional) tangent space $T_{p_0}\Gamma_1$. If $a$ is zero, an easy verification shows that $(T_{[C]}\cM_{g,n}^r)^>$ is contained in either $V_1$ or $V_2$; otherwise, we would have $b>0$ and $b\leq 0$, which is a contradiction.
It then remains to study the case $a\neq 0$, which implies that $(\Gamma,p_0)$ is the $1$-pointed even atom of genus $h$.

We can now describe the action of $\Gm$ on $\Def_{p_1} \oplus T^1_{p_0} \oplus \Def_{(\Gamma,p_0)}$. We use the conventions introduced in \Cref{def:sign_of_rep} and we have the following:
\begin{itemize}
    \item $\Def_{p_1}\simeq T_{p_1}^1 \oplus {\rm Cr}_{p_1}$ with $s(T_{p_1}^1)=\sgn(-b)$ and $s({\rm Cr}_{p_1})=\sgn(b)$;
    \item $s(T^1_{p_0})=\sgn(b+a)$;
    \item $\Def_{(\Gamma,p_0)}\simeq \Def_{(\Gamma,p_0)}^h \oplus \Def_{(\Gamma,p_0)}^{nh}$ where we have $s(\Def_{(\Gamma,p_0)}^h)=\sgn(-a)$ while $s(\Def_{(\Gamma,p_0)}^{nh})=\sgn(a)$.
\end{itemize}

We will prove that either $(T_{[C]}\cM_{g,n}^r)^>$ or $(T_{[C]}\cM_{g,n}^r)^{\leq}$ is contained in either $V_1$ or $V_2$ for any $(a,b) \in \bZ^2$. Suppose this is not the case. Then we have the following implications:
\begin{itemize}
    \item $(T_{[C]}\cM_{g,n}^r)^> \nsubseteq V_1$ implies that $a+b>0$ or $a>0$;
    \item $(T_{[C]}\cM_{g,n}^r)^{\leq} \nsubseteq V_1$ implies that $a+b\leq 0$ or $a\leq 0$;
     \item $(T_{[C]}\cM_{g,n}^r)^> \nsubseteq V_2$ implies that $-b>0$ or $a>0$;
    \item $(T_{[C]}\cM_{g,n}^r)^{\leq} \nsubseteq V_2$ implies that $-b\leq 0$ or $a\leq 0$.
\end{itemize}
To help the reader, we explain the first implication. The others follow in the same manner. Suppose that $a+b\leq 0$ and $a\leq 0$, then the explicit description of (the signs of) the action gives us $$(T_{[C]}\cM_{g,n}^r)^>\subset R\oplus (T_{p_1}^1\oplus {\rm Cr}_{p_1}) \oplus 0 \oplus \Def_{(\Gamma,p_0)}^h = V_1$$
which corresponds to the claimed implication.
An easy verification shows that the intersection of the four conditions on $(a,b)$ above is empty.

Finally, suppose that $\Gamma_1$ does not have any odd $A$-singularity, which implies that it is actually an even atom as well. In fact, in this case \(C=\Gamma\cup \Gamma_1\), since the action on \(\Gamma_1\) is assumed to be non-trivial, and any new special point on \(\Gamma_1\), except the even singularity and the node attaching it to \(\Gamma\), would force the action at the node to be trivial.
\begin{figure}[H]
        \caption{}
        \centering
        \includegraphics[width=0.5\textwidth]{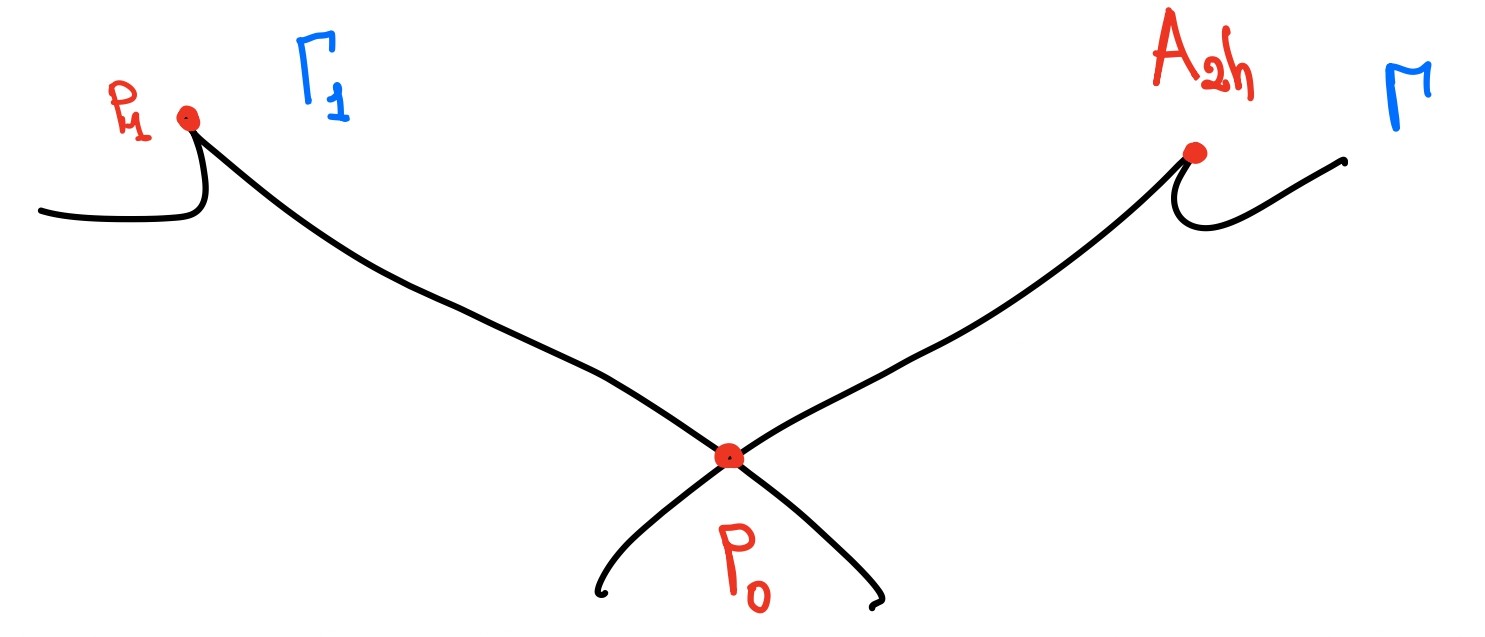}
        \label{fig:Chapter 5_3}
\end{figure}

We denote the irreducible component by $\Gamma_1$ and the even $A$-singularity (of type $A_{2k}$) by $p_1$. In this case, the deformation space is of the form
$$ \Def_{(\Gamma_1,p_0)}\oplus T_{p_0}^1 \oplus \Def_{(\Gamma,p_0)}$$
where we have the decomposition $\Def_{(\Gamma_1,p_0)}\simeq T_{p_1}^1 \oplus {\rm Cr}_{p_1}$. As before, $Z_{[C]}$ contains the representation $V_1:=\Def_{(\Gamma_1,p_0)}\oplus 0 \oplus \Def_{(\Gamma,p_0)}^h $ because of condition $(t_r)$ in \Cref{def:admin}. If $2k<r$, then condition $(t_r)$ gives us that it also contains the representation $V_2:=(T_{p_1}^1 \oplus 0) \oplus 0 \oplus \Def_{(\Gamma,p_0)};$
if $r=2k>2h$, then we have that $Z_{[C]}$ contains the representation
$V_2:=(0\oplus {\rm Cr}_{p_1}) \oplus T_{p_0}^1 \oplus \Def_{(\Gamma,p_0)}^{h}$ 
because of condition $(e_r)$ in \Cref{def:admin}. Indeed, the generic deformation in $V_2$ would give an $A_r$-self-attached hyperelliptic curve of genus $h+k$ with $r>h+k$. We leave it as an exercise to the interested reader to check that either $(T_{[C]}\cM_{g,n}^r)^>$ or $(T_{[C]}\cM_{g,n}^r)^{\leq}$ is contained in $V_1$ or $V_2$ in both cases.

Finally, we leave it to the reader to verify that if $r=2g+1$, $n=1$, and $C$ is an $A_0$-attached hyperelliptic tail of genus $g$, the same holds. The idea is the same, but the situation is simpler because the dimension of \(\Aut(C)\) is one.
\end{proof}

\subsection{Condition on chains}

It is now time to deal with the chain conditions in \Cref{def:admin}. Although the combinatorics is substantially more complicated, the general idea remains the same. To help the reader, we set up and solve the combinatorial problem in \Cref{sec:combinatorics}. We start by illustrating it with the following easy example, analogous to \Cref{ex:easy-warm-up}.

\begin{example}
    Consider a point of $\cM_{g,n}^r$ corresponding to a curve $C/k$ that is constructed by attaching a hyperelliptic chain $\Gamma$ of genus bounded by $(r-1)/2$ to a smooth curve $C_0$ at two nodes $p_1,p_2$. Assume that $\Aut(C)^{\circ}\simeq \Gm$, which implies that $\Gamma$ is a $2$-pointed rosary of hyperelliptic $A$-singularities of even length (since it is also a hyperelliptic chain), i.e., it has an odd number of singularities and thus an even number of components.
    
    We can describe the tangent space 
    $$T_{[C]}\cM_{g,n}^r \simeq \Def_{(C_0,p_1,p_2)} \oplus T_{p_1}^1\oplus T_{p_2}^1 \oplus \Def_{(\Gamma,p_1,p_2)}$$
    where $T_{p_1}^1$ (respectively $T_{p_2}^1$) is the deformation space of the node $p_1$ (respectively $p_2$). Thanks to \Cref{prop:defor-rational-chain}, we have a subrepresentation
    $ \Def_{(\Gamma,p_1,p_2)}^h \subset \Def_{(\Gamma,p_1,p_2)} $
    where $\Def_{(\Gamma,p_1,p_2)}^{h}$ represents the deformation space of $(\Gamma,p_1,p_2)$ that coincides with the deformations preserving the property of being a hyperelliptic chain. We denote by $\Def_{(\Gamma,p_1,p_2)}^{nh}$ the $\Gm$-equivariant complement of $\Def_{(\Gamma,p_1,p_2)}^{h}$ in $\Def_{(\Gamma,p_1,p_2)}$. The main property of this decomposition is that $\Gm$ acts with strictly positive weights on $\Def^h$ and strictly negative weights on $\Def^{nh} \oplus T_{p_1}^1\oplus T_{p_2}^1$, or vice versa. 
    
    Because of condition $(c_r)$ in \Cref{def:admin}, we have that $Z_{[C]}$ contains (in this specific example it is actually equal to) the linear subspace 
    $$\Def_{(C_0,p_1,p_2)} \oplus 0 \oplus 0 \oplus  \Def_{(\Gamma,p_1,p_2)}^h.$$
    Suppose there exists a cocharacter $\Gm \rightarrow \Gm\subset \Aut(C)$, i.e., an integer $a \in \bZ$, such that $Z_{[C]}$ contains neither $(T_{[C]}\cM_{g,n}^r)^{>}$ nor $(T_{[C]}\cM_{g,n}^r)^{\leq }$. Note that the action induced by the cocharacter will have weight $a$ on $T_{p_1}^1$ and $T_{p_2}^1$, act via strictly positive multiples of $a$ on $\Def_{(\Gamma,p_1,p_2)}^{nh}$, act via strictly negative multiples of $a$ on $\Def_{(\Gamma,p_1,p_2)}^{h}$, and act trivially on $\Def_{(C_0,p_1,p_2)}$, again by \Cref{prop:defor-rational-chain}. Exactly as in \Cref{ex:easy-warm-up}, the condition $(T_{[C]}\cM_{g,n}^r)^{>} \nsubseteq Z_{[C]} $ implies that $a > 0$, while the condition $(T_{[C]}\cM_{g,n}^r)^{\leq } \nsubseteq Z_{[C]} $ implies $a \leq 0$. Thus, we have a contradiction.
\end{example}

We are ready for the main lemma for hyperelliptic chains.

\begin{lemma}\label{lem:theta-chain}
    If $C: \spec k \rightarrow \cZ \subset \cM_{g,n}^r$ is a curve that does not satisfy conditions $(c_r)$ or $(d_r)$ of \Cref{def:admin}, then $C$ does satisfy condition $(\star)$.
\end{lemma}

\begin{proof} 
    The idea of the proof is to reduce the statement to a purely combinatorial one, using the natural combinatorial structure of toric representations. We address the combinatorial problem in \Cref{sec:combinatorics}.
    
    Suppose $C$ does not satisfy one of the conditions, i.e., there exists an $A_1/A_1$- (or $A_1/A_r$- or $A_r/A_r$-) attached non-separating hyperelliptic chain of genus bounded by $(r-1)/2$. We will prove that $C$ satisfies condition $(\star)$.
   
    Let $(\Gamma',p_0',p_1')$ be a \emph{minimal} $A_i/A_j$-attached non-separating hyperelliptic chain of genus bounded by $(r-1)/2$ with $i,j\in \{1,r\}$. We define $\Gamma_0'\subset C-\Gamma'$ (respectively $\Gamma_1'\subset C-\Gamma'$) to be the \emph{maximal} connected subcurve of $C-\Gamma'$ that contains $p_0'$ (respectively $p_1'$) and such that $\Aut(C)^{\circ}$ acts non-trivially on every component. Note that $\Gamma_0'$ and/or $\Gamma_1'$ may be empty. We denote by $\Gamma_0$ (respectively $\Gamma_1$) the \emph{minimal} subcurve of $\Gamma_0'$ that contains $p_0'$ (respectively $p_1'$) and such that $\Gamma_0\cap (\Gamma_0'-\Gamma_0)$ (respectively $\Gamma_1\cap (\Gamma_1'-\Gamma_1)$) is an $A_r$-singularity. If $p_0'$ (respectively $p_1'$) is an $A_r$-singularity, we set $\Gamma_0=\emptyset$ (respectively $\Gamma_1=\emptyset$). Moreover, if $p_0'$ (respectively $p_1'$) is not an $A_r$-singularity and $\Gamma_0'$ (respectively $\Gamma_1'$) does not have $A_r$-singularities, we set $\Gamma_0'=\Gamma_0$ (respectively $\Gamma_1'=\Gamma_1$).  Finally, we define $\Gamma:=\Gamma_0 \cup \Gamma' \cup \Gamma_1 $ and we denote by $p_0$ (respectively $p_1$) the intersection $\Gamma_0 \cap (C-\Gamma)$ (respectively $\Gamma_1 \cap (C-\Gamma)$). If $\Gamma_0'$ (respectively $\Gamma_1'$) is empty, we set $p_0:=p_0'$ (respectively $p_1:=p_1'$). Note that because $\Gamma'$ is non-separating, $p_0$ and $p_1$ always exist. See the following picture for an instance where $\Gamma_0=\Gamma_0'$ and $\Gamma_1$ is strictly contained in $\Gamma_1'$ (thus $p_1$ is an $A_r$-singularity).

   \begin{figure}[H]
        \caption{}
        \centering
        \includegraphics[width=0.8\textwidth]{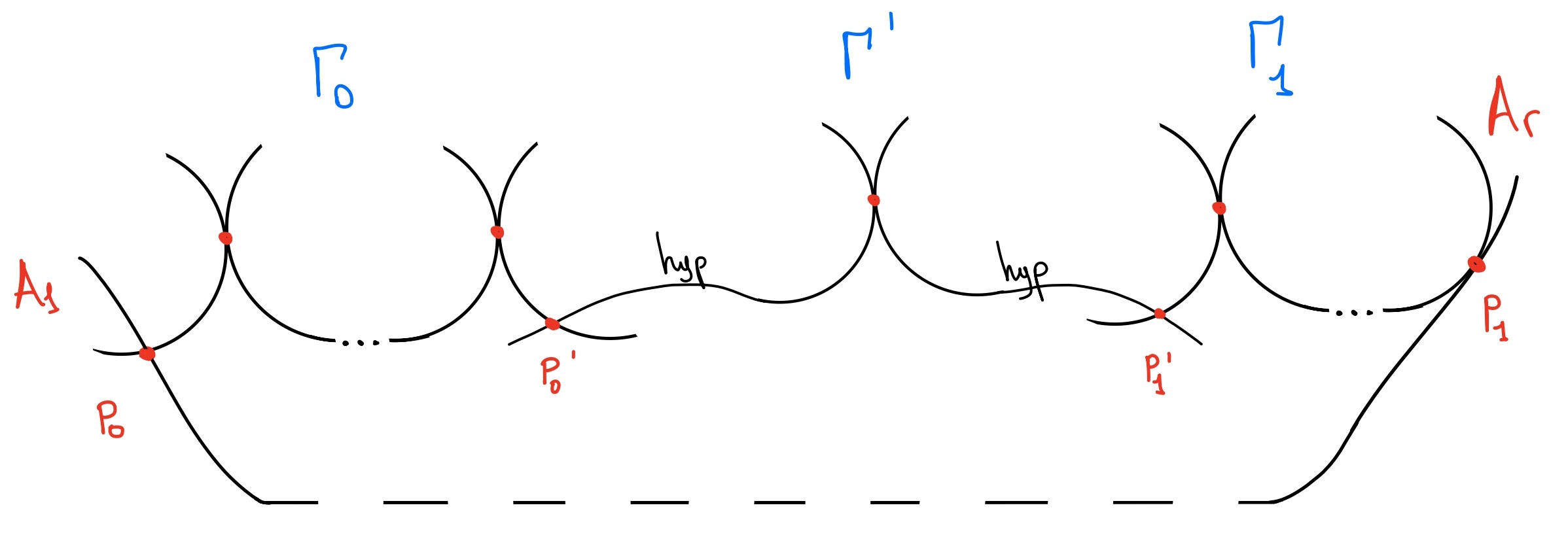}
        \label{fig:Chapter 5_4}
\end{figure}

    By construction, we have the following properties of $(\Gamma,p_0,p_1)$:
    \begin{itemize}
        \item[(a)] $\Gamma$ is a connected subcurve of $C$ intersecting its complement $C-\Gamma$ at $p_0$ and $p_1$; note that the two points are not guaranteed to be distinct (this happens exactly $\Gamma=C$);
        \item[(b)] there are no $A_r$-singularities of $C$ in $\Gamma$ except possibly $p_0$ and $p_1$; 
        \item[(c)] $\Gamma$ contains a non-separating hyperelliptic chain of genus bounded by $(r-1)/2$.
        \item[(d)] $\Gamma_0$ (respectively $\Gamma_1$) is a (possibly empty) sequence of $A_1/A_1$-attached rosaries of hyperelliptic $A$-singularities, except for the chain containing $p_0$ (respectively $p_1$), which could also be an $A_1/A_r$-attached rosary of hyperelliptic $A$-singularities. 
    \end{itemize}
    Note that while $A_1/A_1$-attached rosaries of $A$-singularities have length at least two (i.e., they have at least one $A$-singularity) due to the stability condition, the same is not true for $A_1/A_r$-attached rosaries: indeed, we can have a single $\bP^1$ attached at a node and at a (hyperelliptic) $A_r$-singularity to the rest of the curve. We say that such a chain has length $1$.
    
    Property $(d)$ allows us to describe $\Gamma$ as the sequence
    $$
    \begin{tikzcd}
    {} \arrow[r, "A_1/A_r", "p_0"', no head] & n_k \arrow[r, "A_1", no head] & \dots \arrow[r, "A_1", no head] & n_1 \arrow[r, "A_1", no head] & \Gamma' \arrow[r, "A_1", no head] & m_1 \arrow[r, "A_1", no head] & \dots \arrow[r, "A_1", no head] & m_h \arrow[r, "A_1/A_r", "p_1"', no head] & {}
\end{tikzcd}
    $$
    where the positive integers $n_i$ and $m_j$ indicate the number of hyperelliptic $A$-singularities in the corresponding rosaries (i.e., \(n_i+1\) and \(m_j+1\) indicate the corresponding lengths), while the edges represent the singularities used to connect such chains. The leftmost and rightmost edges represent the points $p_0$ and $p_1$ (possibly equal). Note that for every such rosary, there is a copy of $\Gm$ inside $\Aut(C)^{\circ}$ acting on it.

    We will first address the case where $p_0\neq p_1$.
    In this scenario, we claim the following property of $\Gamma$ holds:
    \begin{itemize}
        \item[$(\bullet)$] the action of $\Aut(C)^{\circ}$ on $T_{p_0}^1$ (respectively on $T_{p_1}^1$) is completely determined by the copy of $\Gm$ associated to the rosary of length $n_k+1$ (respectively $m_h+1$).
    \end{itemize}
    This follows by construction: if $p_0$ is a node, then we know that the action of $\Aut(C)^{\circ}$ on $T_{p_0}(C-\Gamma)$ is trivial because of the maximality of $\Gamma_0'$, which is equal to $\Gamma_0$ if $p_0$ is a node. If $p_0$ is an $A_r$-singularity, then we know that the action is either trivial, or $p_0$ belongs to some $A_1/A_1$-attached rosary of hyperelliptic $A$-singularities, which has a unique $\Gm$-action and is completely determined by the component. 
    
    The idea is now to study the deformation space of $\Gamma$ to verify that for every cocharacter $\Gm \rightarrow \Aut(C)^{\circ}$, $Z_{[C]}$ contains either $(T_{[C]}\cM_{g,n}^r)^{\leq}$ or $(T_{[C]}\cM_{g,n}^r)^{>}$.
    We first address the case where $\Aut(C)^{\circ}$ acts non-trivially on $\Gamma'$, and we will explain later why we can reduce to this case. Therefore, we can describe $\Gamma$ as the following sequence:
    $$
    \begin{tikzcd}
    {} \arrow[r, "A_1/A_r", "p_0"', no head] & n_1 \arrow[r, "A_1", "p_1"', no head] & \dots \arrow[r, "A_1", "p_{k-1}"', no head] & n_k \arrow[r, "A_1/A_r", "p_k"', no head] & {}
    \end{tikzcd}
    $$
    where $n_i$ is a nonnegative integer representing a rosary $\Gamma_i$ of hyperelliptic $A_r$-singularities of length $n_i+1$, for $i=1,\dots,k$. Again, by abuse of notation, we will denote by $p_i$ the node connecting $\Gamma_i$ with $\Gamma_{i+1}$ for $i=1,\dots,k-1$. We denote by $p_0$ the leftmost edge and by $p_k$ the rightmost one.

    To verify $(\star)$, we need to describe $Z_{[C]}$ inside the deformation space of $C$ in $\cM_{g,n}^r$. First of all, we have the following $\Aut(C)^{\circ}$-equivariant decomposition of $T_{[C]}\cM_{g,n}^r$:
    $$ \Def_{(C-\Gamma,p_0,p_1)}\oplus \bigoplus_{i=1}^k (\Def_{(\Gamma_i,p_{i-1},p_i)}^h \oplus \Def_{(\Gamma_i,p_{i-1},p_i)}^{nh}) \oplus \bigoplus_{i=0}^k (T^1_{p_i}\oplus {\rm Cr}_{p_i})$$
    where
    \begin{itemize}
        \item $\Def_{(\Gamma_i,p_{i-1},p_i)}^h$ parametrizes deformations of $\Gamma_i$ that preserve the property of being a chain of hyperelliptic curves;
        \item $\Def_{(\Gamma_i,p_{i-1},p_i)}^{nh}$ is the complementary representation with respect to the action of $\Aut(\Gamma_i,p_{i-1},p_i)^{\circ}\simeq \Gm$. 
    \end{itemize}
    See \Cref{prop:iso-deg-rosary}. Note that except for $i \in \{0,k\}$, ${\rm Cr}_{p_i}$ is zero and $T_{p_i}^1$ is one-dimensional. To ease the notation, we denote 
    \begin{itemize}
        \item by $D(i)_h$ (respectively $D(i)_{nh}$) the subspace $\Def_{(\Gamma_i,p_{i-1},p_i)}^h$ (respectively $\Def_{(\Gamma_i,p_{i-1},p_i)}^{nh}$) for $i=1,\dots,k$;
        \item by $T(i)$ the subspace $T_{p_i}^1$ and by $C(i)$ the subspace ${\rm Cr}_{p_i}$ for $i=0,\dots,k$;
        \item by $D(0)$ the subspace $\Def_{(C-\Gamma,p_0,p_1)}$.
    \end{itemize}
    Moreover, given a component of the given decomposition, say $V$, we denote by $V^{c}$ its complementary representation (i.e., the entire direct sum excluding the summand $V$).
    
    By definition, $Z_{[C]}$ contains the locus where the deformations preserve at least one hyperelliptic chain of genus bounded by $(r-1)/2$. Now pick an integer $1\leq l\leq k$ such that $n_{l}$ is odd, which we know exists by construction. Then we know that 
    $$ T(l-1)^c \cap D(l)_h^c\cap T(l)^c \subset Z_{[C]} $$
    because the left-hand side represents the deformations where $\Gamma_l$ is still an attached hyperelliptic chain of genus bounded by $(r-1)/2$.  More generally, we say that a triplet $1\leq i\leq l\leq j\leq k$ is \emph{odd-minimal} if $$n_t{\rm \; is \; odd \; for \;} i\leq t\leq j \iff  t=l;$$
    for every odd-minimal triplet $i\leq l\leq j$ we can define the subspace
    $$ V_{i,j}(l):= T(i-1)^c \cap \bigcap_{s=i}^{l}D(s)^c_h \cap \bigcap_{s=l+1}^jD(s)^c_{nh} \cap T(j)^c$$
    which will be contained inside $Z_{[C]}$, as $V_{i,j}(l)$ classifies the deformations of $\Gamma$ that make the union $\bigcup_{i\leq s\leq j} \Gamma_s\subset \Gamma$ an attached hyperelliptic chain of genus bounded by $(r-1)/2$ (note that due to condition (b) above, there cannot be hyperelliptic curves of genus $(r-1)/2$ in such a chain). 

     Due to property $(\bullet)$, we can ignore what happens to the subcurve $C-\Gamma$; i.e., it is sufficient to prove that for every cocharacter
     \begin{align}
        \label{eq:cocharacter}
         \chi: \Gm \rightarrow \prod_{i=1}^k \Gm=\Aut(\Gamma,p_0,p_k)^{\circ}\subset \Aut(C)^{\circ}
     \end{align}

     there exists a space $V_{i,j}(l)$ as above that contains either $(T_{[C]}\cM_{g,n}^r)^{\leq}$ or $(T_{[C]}\cM_{g,n}^r)^{>}$. Equivalently, we will prove that there is no cocharacter $\chi$ such that $V_{i,j}(l)$ contains neither $(T_{[C]}\cM_{g,n}^r)^{\leq}$ nor $(T_{[C]}\cM_{g,n}^r)^{>}$. 

    To do so, we need to understand the signs of the weights induced by any cocharacter $\chi:\Gm \rightarrow \Gm^k$. From now on, we will identify the cocharacters  $\chi \in \Hom(\Gm,\Gm^k)$ with vectors of integers $\mathbf{a}\in \bZ^k$. We will use the sign convention introduced in \Cref{def:sign_of_rep}. We have the following relations:
    \begin{itemize}
        \item $s(D(i)_h)=\sgn(a_i)$ while $s(D(i)_{nh})=\sgn(-a_i)$ for every $i=1,\dots,k$;
        \item $s(T(i))=\sgn(a_{i+1}-(-1)^{n_i}a_i)$ for $i=0,\dots,k$ (where we can set $a_{0}=a_{k+1}=0$ by property $(\bullet$)).
    \end{itemize}
    
    Therefore, for every odd-minimal triplet $i\leq l\leq j$ we have that 
    $$ (T_{[C]}\cM_{g,n}^r)^{>} \nsubseteq V_{i,j}(l)$$
    implies that one of the following is true:
    \begin{itemize}
        \item $a_{i}>(-1)^{n_{i-1}}a_{i-1}$;
        \item $a_s>0$ for some $i\leq s\leq l$;
        \item $a_s<0$ for some $l+1\leq s\leq j$;
        \item $a_{j+1}>(-1)^{n_j}a_j$.
    \end{itemize}
    The same is true for $(T_{[C]}\cM_{g,n}^r)^{\leq}$ with complementary inequality conditions. 
    
Note that, if instead $\Aut(C)^{\circ}$ acts trivially on $\Gamma'$, the combinatorial problem remains the same, except that we are verifying the containment only for cocharacters $\chi$ in \eqref{eq:cocharacter} with a trivial \(i\)-th component, where \(\Gamma'=\Gamma_i\). Thus, it is sufficient to solve it in the case where the action is non-trivial. Moreover, the only information we need about the $n_i$'s is their parity, both for the definition of the odd-minimal triplet and for the signs of the action. We can then reduce the problem by assuming that the $n_i$'s are either zeros or ones, depending on their parity. We have finally arrived at the core of the proof, which is a purely combinatorial statement. In \Cref{sec:combinatorics}, we rewrite the combinatorial setup, and \Cref{theo:combi} shows that for every vector of integers $\mathbf{a}\in \bZ^k$, there always exists an odd-minimal triplet $i\leq l\leq j$ such that none of the conditions above (or their respective counterparts with $\leq$) are satisfied. This implies that either $(T_{[C]}\cM_{g,n}^r)^{>}$ or $(T_{[C]}\cM_{g,n}^r)^{\leq}$ is contained in $V_{i,j}(l)$.

    To finish the proof of the theorem, we must consider the case $p_0=p_1$. The combinatorics of this problem is slightly more convoluted because we cannot set $a_{k+1}=0$ and $a_0=0$ by construction as before; indeed, the correct convention is $a_{k+1}:=a_1$ and $a_0:=a_k$, because the curve $\Gamma$ actually coincides with $C$ and is a cycle. Notice that this can only happen if $n=0$. We handle this as before in \Cref{sec:combinatorics}, specifically in \Cref{theo:cyc-combi}.
\end{proof}

\subsection{Condition \texorpdfstring{$(e_r)$}{(e\_r)}}

Finally, we conclude this section by addressing the last condition.

\begin{lemma}\label{lem:theta-n=0}
    If $C\in \cZ \subset \cM_{g,n}^r$ is a curve that does not condition $(e_r)$ of \Cref{def:admin}, then it satisfies condition $(\star)$.
\end{lemma}

\begin{proof}
    It is sufficient to verify the statement for those curves $C$ that have non-trivial, positive-dimensional stabilizers. We can explicitly describe such a curve $C$, which is constructed in one of the following ways (see \Cref{def:self-attached}):
    \begin{itemize}
        \item[(1)] $C$ is obtained by pinching a $\bP^1$ at $0$ and $\infty$, creating a hyperelliptic $A_{2k}$-singularity and a hyperelliptic $A_{2h}$-singularity with $k+h=g$ and $h>g$;
        \item[(2)] $C$ is obtained by gluing two $\bP^1$ at $0$ and $\infty$, creating an $A_{2k+1}$-singularity and a hyperelliptic $A_{2h+1}$-singularity, with $k+h=g$ and $h>g$;
        \item[(3)] $C$ is obtained by pinching a $\bP^1$ at $\infty$, creating an $A_{2g}$-singularity;
        \item[(4)] $C$ is obtained by gluing two $\bP^1$ at $\infty$, creating an $A_{2g+1}$-singularity.
    \end{itemize}
    We will show how case (1) works; we leave it to the reader to verify the remaining cases.
    We know that the deformation space decomposes as
    \[
    \Def_C = \Def_{k}\oplus \Def_h
    \]
    where $\Def_k \coloneq T_k^1 \oplus {\rm Cr}_k$, and analogously for $\Def_h$. The space $T^1_k$ is the deformation space of the $A_{2k+1}$-singularity, and ${\rm Cr}_k$ is the crimping space (and similarly for $h$). Moreover, $\Aut(C)^{\circ} \simeq \Gm$, where
    \[
    s(T_k^1) = s({\rm Cr}_h) = -s({\rm Cr}_k) = -s(T_h^1)
    \]
    with the conventions introduced in \Cref{def:sign_of_rep}.

    We have that $Z$ contains the subspace 
    \[
    V \coloneq (T_k^1 \oplus 0) \oplus (0 \oplus {\rm Cr}_h).
    \]
    It clearly follows that for every cocharacter $a \colon \Gm \rightarrow \Gm$, the subspace $V$ contains either $\Def_C^>$ or $\Def_C^{\leq}$, because otherwise we would have both $a > 0$ and $a \leq 0$.
\end{proof}

\appendix
\section{Combinatorics of problematic sequences}\label{sec:combinatorics}

In this appendix, we deal with the combinatorial problem arising from the local geometry of the open embedding $\cU_{g,n}^r \subset \cM_{g,n}^r$. See the proof of \Cref{theo:local-completeness}.

We will first set up the problem. Fix a positive integer $k$, and suppose we are given a binary sequence $\mathbf{b} \in \{0,1\}^k$ of length $k$. We will draw such a sequence as 
$$
\begin{tikzcd}
{} \arrow[r, no head] & b_1 \arrow[r, no head] & b_2 \arrow[r, no head] & \dots \arrow[r, no head] & b_k \arrow[r, no head] & {}
\end{tikzcd}
$$
where $b_i \in \{0,1\}$ for every $i=1,\dots,k$. We will denote by $e_i$ the edge in the previous picture between $b_i$ and $b_{i+1}$. The leftmost edge will be denoted by $e_0$, while the rightmost edge will be denoted by $e_k$. For every $i\leq j \in \{1,\dots,k\}$, we will denote by $\mathbf{b}(i,j)$ the subsequence in $\{0,1\}^{i-j+1}$ which starts with the edge $e_{i-1}$ and ends with the edge $e_j$.

\begin{definition}
    Fix $k$ a positive integer and a binary sequence $\mathbf{b}$ of length $k$. A subsequence $\mathbf{b}(i,j)$ (or equivalently the pair $i\leq j$) will be called \emph{$1$-minimal} if there exists a unique $h \in \{i,\dots,j\}$ such that $b_{h}$ is $1$.
\end{definition}

We will now introduce all the terminology needed to state the problem.

\begin{definition}\label{def:>}
    A vector of integers $\mathbf{a}\in \bZ^k$ will be called $>$-\emph{problematic} with respect to a $1$-minimal subsequence $\mathbf{b}(i,j)$ of $\mathbf{b}$ with $b_h=1$ ($i\leq h\leq j$) if one of the following conditions holds:
    \begin{itemize}
        \item[(1)] $a_{i}>(-1)^{b_{i-1}}a_{i-1}$;
        \item[(2)] $a_s>0$ for some $i\leq s\leq h$;
        \item[(3)] $a_s<0$ for some $h+1\leq s\leq j$;
        \item[(4)] $a_{j+1}>(-1)^{b_j}a_j$.
    \end{itemize}
    We will denote by $A_{i,j}(>)$ the set of $>$-problematic vectors for the subsequence $\mathbf{b}(i,j)$. A vector of integers $\mathbf{a}\in \bZ^k$ will be called $>$-\emph{problematic} with respect to $\mathbf{b}$ if it is $>$-problematic for every $1$-minimal subsequence. We will denote by $A(>)$ the set of $>$-problematic vectors for $\mathbf{b}$.
    
\end{definition} 

\begin{remark}
    We set $a_{0}$ and $a_{k+1}$ to be zero in the previous definition. For instance, if $i=1$, then condition $(1)$ becomes $a_i>0$, which is already included in condition $(2)$. Moreover, if $j=k$, then condition $(4)$ becomes $(-1)^{b_k}a_k<0$. If $b_k=1$, then $h=k$, and thus condition $(3)$ is empty, while condition $(4)$ coincides with $a_k>0$, which is already contained in $(2)$. If $b_k=0$, then $h<k$, and thus condition $(4)$ is included in condition $(3)$.
\end{remark}

\begin{definition}\label{def:leq}
    A vector of integers $\mathbf{a}\in \bZ^k$ will be called $\leq$-\emph{problematic} with respect to a minimal subsequence $\mathbf{b}(i,j)$ of $\mathbf{b}$ with $b_h=1$ ($i\leq h\leq j$) if one of the following conditions holds:
    \begin{itemize}
        \item[(1)] $a_{i}\leq(-1)^{b_{i-1}}a_{i-1}$;
        \item[(2)] $a_s\leq0$ for some $i\leq s\leq h$;
        \item[(3)] $a_s\geq 0$ for some $h+1\leq s\leq j$;
        \item[(4)] $a_{j+1}\leq (-1)^{b_j}a_j$.
    \end{itemize}
    We will denote by $A_{i,j}(\leq)$ the set of $\leq$-problematic vectors for the subsequence $\mathbf{b}(i,j)$.  A vector of integers $\mathbf{a}\in \bZ^k$ will be called $\leq$-\emph{problematic} with respect to $\mathbf{b}$ if it is $\leq$-problematic for every $1$-minimal subsequence. We will denote by $A(\leq)$ the set of $\leq$-problematic vectors for $\mathbf{b}$.
\end{definition}

\begin{remark}\label{rem:case-finite}
    As before, we set $a_{0}$ and $a_{k+1}$ to be zero in the previous definition. For instance, if $i=1$, then condition $(1)$ becomes $a_i\leq 0$, which is already included in condition $(2)$. Moreover, if $j=k$, then condition $(4)$ becomes $(-1)^{b_k}a_k\leq 0$. If $b_k=1$, then $h=k$, and thus condition $(3)$ is empty, while condition $(4)$ coincides with $a_k\leq 0$, which is already contained in $(2)$. If $b_k=0$, then $h<k$, and thus condition $(4)$ is included in condition $(3)$.
\end{remark}

\begin{definition}\label{def:prob}
    Let $k$ be a positive integer and $\mathbf{b}$ be a binary sequence of length $k$. Then a vector $\mathbf{a}\in \bZ^k$ is called \emph{problematic} (with respect to $\mathbf{b}$) if it is $>$-problematic and $\leq$-problematic. We will denote by $A\subset \bZ^k$ the subset of problematic vectors.
\end{definition}

\begin{remark}
    We have the following straightforward equalities:
    \begin{itemize}
        \item $ A(>)= \bigcap_{1-{\rm min}\, \mathbf{b}(i,j)} A_{i,j}(>);$
        \item $ A(\leq)= \bigcap_{1-{\rm min}\, \mathbf{b}(i,j)} A_{i,j}(\leq);$
        \item $ A = A(>) \cap A(\leq);$
    \end{itemize}
    where the intersection is taken over the $1$-minimal subsequences of $\mathbf{b}$.
\end{remark}

We are ready for the main theorem of the section.

\begin{theorem}\label{theo:combi}
    Let $k$ be a positive integer and $\mathbf{b}$ be a binary sequence of length $k$. There are no problematic vectors for $\mathbf{b}$.
\end{theorem}
The proof of the theorem follows from a straightforward iterative argument, which consists of two lemmas. In the first, we prove the starting point of the iterative argument. In the second, we describe how the iteration proceeds.

\begin{lemma}\label{lem:start}
   Suppose we are under the hypotheses of \Cref{theo:combi} and let $\mathbf{a}\in \bZ^k$ be a problematic vector of integers. Let $h_1<h_2<\dots<h_r$ be the indices such that $b_{h_t}=1$ for $t=1,\dots,r$.  Then 
   $$ a_{h_1}>0 \iff a_{h_2}\leq a_{h_2-1}\leq a_{h_2-2} \leq \dots \leq a_{h_1+1}\leq -a_{h_1}< 0.$$
   The same holds for the opposite sign case, namely:
   $$ a_{h_1}\leq 0 \iff a_{h_2}>a_{h_2-1}> a_{h_2-2} > \dots > a_{h_1+1}> -a_{h_1}\geq 0.$$
\end{lemma}

\begin{proof}
    The ``only if" statement is clear. Assume then that $a_{h_1}>0$, and suppose $a_{h_1+1}+a_{h_1}>0$. If we consider the $1$-minimal subsequence $\mathbf{b}(h_1,h_1)$, condition $(1)$ of \Cref{def:leq} must hold (because the others are not met), and thus $a_{h_1}\leq a_{h_{1}-1}$, which means $a_{h_1-1}>0$. Therefore, we can apply condition $(1)$ of \Cref{def:leq} again for the subsequence $\mathbf{b}(h_1-1,h_1)$, which gives us $a_{h_1-1}\leq a_{h_1-2}$, and this again implies $a_{h_1-2}>0$. We can apply this idea iteratively to get that 
    $$ a_{h_1+1}+a_{h_1}>0 \implies a_1\geq a_2\geq \dots \geq a_{h_1-1}\geq a_{h_1}>0.$$
    This is impossible, as in this case the vector $\mathbf{a}$ would not satisfy any of the conditions of \Cref{def:leq} for the subsequence $\mathbf{b}(1,h_1)$. Thus, we have that $a_{h_1+1}\leq -a_{h_1}<0$. 

    We will now repeat the same idea to get $a_{h_1+2}\leq a_{h_1+1}$. Suppose $a_{h_1+2}> a_{h_1+1}$; then condition $(1)$ of \Cref{def:leq} must hold for the subsequence $\mathbf{b}(h_1,h_1+1)$, giving us again $a_{h_1}\leq a_{h_{1}-1}$, and thus $a_{h_1-1}>0$. Iterating the same idea as before, we reach a contradiction. Thus, $a_{h_1+2}\leq a_{h_1+1}$. We can iterate this process until we meet the index $h_2$, because of the $1$-minimality. We leave it to the reader to check the remaining case.
\end{proof}

\begin{lemma}\label{lem:alternating}
    Suppose we are under the hypotheses of \Cref{theo:combi} and let $\mathbf{a}\in \bZ^k$ be a problematic vector of integers. Let $h_1<h_2<\dots<h_r$ be the indices such that $b_{h_t}=1$ for $t=1,\dots,r$.  Then for every $s=2,\dots,k-1$, we get that 
   $$a_{h_s}> a_{h_s-1}> a_{h_s-2} > \dots >a_{h_{s-1}+1}> -a_{h_{s-1}}\geq 0$$
   implies (and is in fact equivalent to)
   $$a_{h_{s+1}}\leq a_{h_{s+1}-1}\leq a_{h_{s+1}-2} \leq \dots \leq a_{h_{s}+1}\leq -a_{h_s}< 0.$$
   The same holds for the opposite sign case, namely:
   $$a_{h_s}\leq  a_{h_s-1}\leq a_{h_s-2} \leq \dots \leq a_{h_{s-1}+1} \leq -a_{h_{s-1}} < 0$$
   implies (and is in fact equivalent to)
   $$a_{h_{s+1}}>a_{h_{s+1}-1}> a_{h_{s+1}-2}> \dots > a_{h_{s}+1}> -a_{h_s}\geq  0.$$
\end{lemma}

\begin{proof}
    The proof works exactly as in the previous case. Indeed, it is enough to check the conditions of \Cref{def:leq} for the subsequences $\mathbf{b}(h_s,j)$ for $j=h_s,\dots,h_{s+1}-1$. We leave it to the reader to verify the details.
\end{proof}

\begin{remark}\label{rem:simplify}
    Notice that the proof only uses the inequalities $a_{h_s}>a_{h_s-1}$ and $a_{h_s}>0$ to deduce the statement.
\end{remark}
\begin{proof}[Proof of \Cref{theo:combi}]
    Suppose there exists $\mathbf{a}\in \bZ^k$, a problematic vector of integers. Let $h_1<h_2<\dots<h_r$ be all the indices such that $b_{h_t}=1$ (for $t=1,\dots,r$). Suppose that $a_{h_1}>0$. By \Cref{lem:start}, this implies
    $$ a_{h_2}\leq a_{h_2-1}\leq a_{h_2-2} \leq \dots \leq a_{h_1+1}\leq -a_{h_1}< 0.$$
    By applying \Cref{lem:alternating} iteratively, we get that either 
    $$ a_{h_r}\leq a_{h_r-1}\leq a_{h_r-2} \leq \dots \leq a_{h_{r-1}+1}\leq -a_{h_{r-1}}< 0$$
    or 
    $$a_{h_{r}}>a_{h_{r}-1}> a_{h_{r}-2}> \dots > a_{h_{r-1}+1}> -a_{h_{r-1}}\geq  0.$$
    In both cases, we reach a contradiction, confirming that $\mathbf{a}$ does not satisfy the conditions in \Cref{def:>} (or in \Cref{def:leq}) for the subsequence $\mathbf{b}(h_r,k)$.
\end{proof}

\addtocontents{toc}{\SkipTocEntry}
\subsection*{Cyclic case}
We will generalize the statement of \Cref{theo:combi} to the case where the sequence is actually a circuit, i.e., in the picture 
$$
\begin{tikzcd}
{} \arrow[r, no head] & b_1 \arrow[r, no head] & b_2 \arrow[r, no head] & \dots \arrow[r, no head] & b_k \arrow[r, no head] & {}
\end{tikzcd}
$$
the leftmost and rightmost edges coincide. To do this formally, we extend the binary sequence $\mathbf{b}\in \{0,1\}^k$ to an infinite binary sequence $\mathbf{b}^{\circ} \in \{0,1\}^{\bZ}$ which is $k$-periodic, i.e., $\mathbf{b}^{\circ}(n):=\mathbf{b}([n])$ where $[n]$ is the equivalence class of the integer $n$ modulo $k$. The definition of problematic vectors remains the same, but \Cref{rem:case-finite} is no longer applicable. Notice that \Cref{lem:alternating} is still valid, as the proof does not rely on the finiteness of $k$. Given a vector of integers $\mathbf{a}\in \bZ^k$ of length $k$, we denote by $\mathbf{a}^{\circ}$ its periodic extension:
$$
\begin{tikzcd}
\bZ \arrow[r, "\pi_k"', two heads] \arrow[rr, "\mathbf{a}^{\circ}", bend left] & \bZ/(k) \arrow[r, "\mathbf{a}"'] & {\{0,1\}}
\end{tikzcd}
$$
where $\pi_k$ is the usual quotient map. We will denote by $a_i^{\circ}$ the integer $\mathbf{a}^{\circ}(i)$ for some index $i \in \bZ$.

\begin{definition}
    We say that $\mathbf{a}\in \bZ^k$ is problematic for $\mathbf{b}^{\circ}$ if $\mathbf{a}^{\circ}$ satisfies the conditions of \Cref{def:prob} for the periodic sequence $\mathbf{b}^{\circ}$.
\end{definition}

\begin{lemma}\label{lem:cyc-alt}
    Let $\mathbf{a}\in \bZ^k$ be a vector of integers such that 
    there exists an index $1\leq i\leq k$ with the following properties:
    \begin{itemize}
        \item $b_i=1$;
        \item $a_i^{\circ}>0$;
        \item $a_i^{\circ}>(-1)^{b_{i-1}^{\circ}}a_{i-1}^{\circ}$ and $a_i^{\circ}>-a_{i+1}^{\circ}$.
    \end{itemize}
    Then $\mathbf{a}$ is not problematic for $\mathbf{b}^{\circ}$. The same statement holds for $\leq$.
\end{lemma}

\begin{proof}
    Suppose that $\mathbf{a}$ is problematic. By the periodicity of $\mathbf{b}^{\circ}$ and $\mathbf{a}^{\circ}$, we can assume $i=1$ without loss of generality. If $1= h_1 < h_2 < \dots < h_r\leq k$ are the indices where $b_{h_t}=1$, then we can iteratively apply \Cref{lem:alternating} (see \Cref{rem:simplify}) to get that 
    \begin{itemize}
        \item if $r$ is even, then $$a_{h_1}\leq -a_{h_2} < a_{h_3}\leq \dots \leq -a_{h_r} < a_{h_1+k}^{\circ}=a_{h_1};$$
        \item if $r$ is odd, then 
        $$ a_{h_1}\leq -a_{h_2} < a_{h_3}\leq \dots < a_{h_r} \leq  -a_{h_1+k}^{\circ}=-a_{h_1}.$$
    \end{itemize}
    In both cases, we reach a contradiction.
\end{proof}

\begin{lemma}\label{lem:cyc-not-alt}
    Let $\mathbf{a}\in \bZ^k$ be a vector of integers such that 
    there exists an index $1\leq i\leq k$ with the following properties:
    \begin{itemize}
        \item $b_i=1$;
        \item $a_i^{\circ}>0$;
        \item $a_i^{\circ}\leq (-1)^{b_{i-1}^{\circ}}a_{i-1}^{\circ}$ and $a_i^{\circ}\leq -a_{i+1}^{\circ}$. 
    \end{itemize}
    Then $\mathbf{a}$ is not problematic for $\mathbf{b}^{\circ}$. The same statement holds for $\leq$.
\end{lemma}

\begin{proof}
    Suppose $\mathbf{a}$ is problematic and consider the maximal $1$-minimal sequence $\mathbf{b}(j_1,j_2)$ containing $i$. By periodicity, we can assume $j_1=1$, and we set $j:=j_2$. Because of the hypotheses, we know that $a_{i-1}>0$ and $a_{i+1}<0$. Because $\mathbf{a}$ is problematic, the condition on the subsequence $\mathbf{b}(i-1,i+1)$ gives us that either $a_{i+2}\leq a_{i+1}< 0$ or $a_{i-2}\geq a_{i-1}>0$. Now take the largest interval $2 \leq i_1\leq i\leq i_2\leq j-1$ such that for every $i_1\leq t\leq i$ we have 
    $$a_t\leq (-1)^{b_{t-1}}a_{t-1}$$ and for every $i\leq t\leq i_2$ we have 
    $$a_t\geq a_{t+1}.$$
    Then we claim that either $i_1=2$ or $i_2=j-1$. In fact, we get that $a_t>0$ for every $i_1-1\leq t\leq i$ and $a_t<0$ for every $i\leq t\leq i_2+1$. If we consider the sequence $\mathbf{b}(i_1-1,i_2+1)$, then one of the two relations above must be true for either $t=i_1-1$ or $t=i_2+1$. This contradicts the maximality of the interval $[i_1,i_2]$. Without loss of generality, we can assume $i_2=j$, which implies $a_j\leq a_{j-1}<0$. But now we have 
    \begin{itemize}
        \item $b_j=1$, because of the maximality of the $1$-minimal sequence $\mathbf{b}(1,j)$;
        \item $a_j\leq 0$;
        \item $a_j\leq a_{j-1}$.
    \end{itemize}
    Therefore, we can apply \Cref{lem:cyc-alt} to get the result. 
\end{proof}

\begin{theorem}\label{theo:cyc-combi}
    Let $\mathbf{b}\in \{0,1\}^k$ be a binary sequence of length $k$. If there exists an index $1\leq i \leq k$ such that $b_i=1$, then there are no problematic vectors of integers $\mathbf{a}\in \bZ^k$ for the periodic infinite binary sequence $\mathbf{b}^{\circ}$.
\end{theorem}

\begin{proof}
    The theorem follows by combining \Cref{lem:cyc-alt} and \Cref{lem:cyc-not-alt}.
\end{proof}

\begin{remark}
    Notice that \Cref{theo:cyc-combi} actually implies \Cref{theo:combi}. This can be seen by enlarging the binary sequence $\mathbf{b} \in \{0,1\}^k$ by adding one element to it and setting it to $0$. We denote by $\mathbf{b}'\in \{0,1\}^{k+1}$ the new binary sequence of length $k+1$. Now the problematic vectors of integers with respect to $\mathbf{b'}^{\circ}$ which have $a_{k+1}=0$ coincide exactly with the problematic vectors with respect to $\mathbf{b}$. 
\end{remark}
\bibliographystyle{alpha}
\bibliography{References.bib}

@article{StructureOfInstability,
  title={On the structure of instability in moduli theory},
  author={Daniel Halpern-Leistner},
  journal={arXiv: Algebraic Geometry},
  year={2014},
  url={https://api.semanticscholar.org/CorpusID:117058823}
}

@phdthesis{VanDerWyck,
    author={van der Wyck, Frederick D. W.},
    year={2010},
    title={Moduli of singular curves and crimping},
    journal={ProQuest Dissertations and Theses},
    pages={144},
    isbn={978-1-124-09171-6},
    school={Harvard University},
    url={https://www.proquest.com/dissertations-theses/moduli-singular-curves-crimping/docview/613377231/se-2},
}

@article{Viviani,
    title={ON THE FIRST STEPS OF THE MINIMAL MODEL PROGRAM FOR THE MODULI SPACE OF STABLE POINTED CURVES},
    DOI={10.1017/S1474748021000116},
    journal={Journal of the Institute of Mathematics of Jussieu},
    publisher={Cambridge University Press},
    author={Codogni, Giulio and Tasin, Luca and Viviani, Filippo},
    year={2023},
    pages={145–211}
}

@article{hassett2013log,
  title={Log minimal model program for the moduli space of stable curves: the first flip},
  author={Hassett, Brendan and Hyeon, Donghoon},
  journal={Annals of Mathematics},
  pages={911--968},
  year={2013},
  publisher={JSTOR}
}

@article{AlpFedSmyWyck,

title={Second flip in the Hassett-Keel program: A local description }, 
volume={153},  
journal={Compositio Mathematica}, 
publisher={London Mathematical Society}, 
author={Alper, J. and Fedorchuck, M. and Smyth, D.I. and van der Wyck, F.}, 
year={2017}, 
pages={1547-1583}}

@article {ExistenceOfModuli,
    AUTHOR = {Alper, Jarod and Halpern-Leistner, Daniel and Heinloth,
              Jochen},
     TITLE = {Existence of moduli spaces for algebraic stacks},
   JOURNAL = {Invent. Math.},
  FJOURNAL = {Inventiones Mathematicae},
    VOLUME = {234},
      YEAR = {2023},
    NUMBER = {3},
     PAGES = {949--1038},
      ISSN = {0020-9910,1432-1297},
   MRCLASS = {14D23 (14A20)},
  MRNUMBER = {4665776},
       DOI = {10.1007/s00222-023-01214-4},
       URL = {https://doi.org/10.1007/s00222-023-01214-4}}

@Article{Per1,
  author  = {Pernice, M.},
  title   = {{The moduli stack of $A_r$-stable curves}},
  journal = {Preprint on arXiv},
  year    = {2023},
  url     = {https://arxiv.org/abs/2302.10877},
}

@Article{AlpHalRyd,
      title={The \'etale local structure of algebraic stacks}, 
      author={Alper, J. and Hall, J. and Rydh, D.},
      year={2023},
      eprint={1912.06162},
      archivePrefix={arXiv},
      primaryClass={math.AG},
      url={https://arxiv.org/abs/1912.06162}, 
      journal={arXiv}
}

@Article{Per2,
  author  = {Pernice, M.},
  title   = {{Hyperelliptic $A_r$-stable curves (and their moduli stack)}},
  journal = {Transactions of the American Mathematical Society},
  year    = {2024},
  volume = {377},
  pages = {4133-4169}
}

@article{Cat,
author={Catanese, F.},
title={{Pluricanonical Gorenstein Curves}},
journal={Enumerative Geometry and Classical Algebraic Geometry},
year={1982},
publisher={Birkh{\"a}user Boston},
address={Boston, MA},
pages={51--95},
isbn={978-1-4684-6726-0},
doi={10.1007/978-1-4684-6726-0_4},
url={https://doi.org/10.1007/978-1-4684-6726-0_4}
}

@book{LauMorBai,
	Author={Laumon, G. and  Moret-Bailly, L.},
	Title={{Champs algébriques}},
	Year={2000},
	Publisher={Springer-Verlag, Berlin},	
}

@article{CasmarLaza,
 ISSN = {00029947},
 URL = {http://www.jstor.org/stable/23513483},
 author = {Casalaina Martin, S. and Laza, R.},
 journal = {Transactions of the American Mathematical Society},
 number = {5},
 pages = {2271--2295},
 publisher = {American Mathematical Society},
 title =  {{Simultaneous semi-stable reduction for curves with ADE singularities}},
 volume = {365},
 year = {2013}
}

@article{Fedor,
  title = {{Moduli spaces of hyperelliptic curves with A and D singularities}},
  volume = {276},
  ISSN = {1432-1823},
  url = {http://dx.doi.org/10.1007/s00209-013-1201-6},
  DOI = {10.1007/s00209-013-1201-6},
  number = {1–2},
  journal = {Mathematische Zeitschrift},
  publisher = {Springer Science and Business Media LLC},
  author = {Fedorchuk,  M.},
  year = {2013},
  pages = {299–328}
}

@article{DeligneMumford,
 author = {Deligne, Pierre and Mumford, D.},
 title = {The irreducibility of the space of curves of a given genus},
 fjournal = {Publications Math{\'e}matiques},
 journal = {Publ. Math., Inst. Hautes {\'E}tud. Sci.},
 issn = {0073-8301},
 volume = {36},
 pages = {75--109},
 year = {1969},
 language = {English},
 doi = {10.1007/BF02684599},
 url = {https://eudml.org/doc/103899},
 zbMATH = {3289080},
 Zbl = {0181.48803}
}

@article{PseudostableSchubert,
 author = {Schubert, David},
 title = {A new compactification of the moduli space of curves},
 fjournal = {Compositio Mathematica},
 journal = {Compos. Math.},
 issn = {0010-437X},
 volume = {78},
 number = {3},
 pages = {297--313},
 year = {1991},
 language = {English},
 url = {https://eudml.org/doc/90093},
 zbMATH = {9343},
 Zbl = {0735.14022}
}

@article{AlpFedSmyExistence,
 author = {Alper, Jarod and Fedorchuk, Maksym and Smyth, David Ishii},
 title = {Second flip in the {Hassett}-{Keel} program: existence of good moduli spaces},
 fjournal = {Compositio Mathematica},
 journal = {Compos. Math.},
 issn = {0010-437X},
 volume = {153},
 number = {8},
 pages = {1584--1609},
 year = {2017},
 language = {English},
 doi = {10.1112/S0010437X16008289},
 zbMATH = {6764719},
 Zbl = {1403.14038}
}

@article{AlpFedSmyProjectivity,
 author = {Alper, Jarod and Fedorchuk, Maksym and Smyth, David Ishii},
 title = {Second flip in the {Hassett}-{Keel} program: projectivity},
 fjournal = {IMRN. International Mathematics Research Notices},
 journal = {Int. Math. Res. Not.},
 issn = {1073-7928},
 volume = {2017},
 number = {24},
 pages = {7375--7419},
 year = {2017},
 language = {English},
 doi = {10.1093/imrn/rnw216},
 zbMATH = {7004507},
 Zbl = {1405.14063}
}

@article{smyth2011modular,
  title={Modular compactifications of the space of pointed elliptic curves II},
  author={Smyth, David Ishii},
  journal={Compositio Mathematica},
  volume={147},
  number={6},
  pages={1843--1884},
  year={2011},
  publisher={London Mathematical Society}
}

@article{hassett2009log,
  title={Log canonical models for the moduli space of curves: the first divisorial contraction},
  author={Hassett, Brendan and Hyeon, Donghoon},
  journal={Transactions of the American Mathematical Society},
  volume={361},
  number={8},
  pages={4471--4489},
  year={2009}
}

@phdthesis{AltCompClustAlg, 
author = {Gori, Davide},
title = {Alternative Compactifications of $M_{g,n}$ via Cluster Algebras and their Birational Geometry},
year = {2026},
url = {https://iris.uniroma1.it/handle/11573/1759004},
note = {\url{https://iris.uniroma1.it/handle/11573/1759004}},
school = {Universit\`a La Sapienza, Roma},
language = {English}
}

@misc{GMSArI,
 author = {Davide Gori and Ludvig Modin and Michele Pernice},
 title = {The local geometry of the stack of {$A_r$}-stable curves},
 year = {2026},
 howpublished = {Preprint, {arXiv}:2603.29853 [math.{AG}] (2026)},
 url = {https://arxiv.org/abs/2603.29853},
 arXiv = {arXiv:2603.29853}
}

@incollection {GMS-VectBun,
    AUTHOR = {Alper, Jarod and Belmans, Pieter and Bragg, Daniel and Liang,
              Jason and Tajakka, Tuomas},
     TITLE = {Projectivity of the moduli space of vector bundles on a curve},
 BOOKTITLE = {Stacks {P}roject {E}xpository {C}ollection ({SPEC})},
    SERIES = {London Math. Soc. Lecture Note Ser.},
    VOLUME = {480},
     PAGES = {90--125},
 PUBLISHER = {Cambridge Univ. Press, Cambridge},
      YEAR = {2022},
      ISBN = {978-1-009-05485-0},
   MRCLASS = {14H60 (14D20)},
  MRNUMBER = {4480534},
MRREVIEWER = {Leticia\ Brambila-Paz},
}

@misc{AlpHalLei,
      title={{The intrinsic approach to moduli theory}}, 
      author={Alper, Jarod and Halpern-Leistner, Daniel},
      year={2026},
      eprint={2603.21412},
      archivePrefix={arXiv},
      primaryClass={math.AG},
      url={https://arxiv.org/abs/2603.21412}, 
}

@article{hassett2003moduli,
  title={Moduli spaces of weighted pointed stable curves},
  author={Hassett, Brendan},
  journal={Advances in Mathematics},
  volume={173},
  number={2},
  pages={316--352},
  year={2003},
  publisher={Elsevier}
}
\end{document}